\documentclass[12pt,reqno]{amsart}
\usepackage{extsizes,amsrefs}

\usepackage{etex}

\usepackage{hyperref}
\usepackage{txfonts,amsmath,amstext,amsthm,amscd,amsopn,verbatim,amssymb,amsfonts,mathtools,enumitem}

\usepackage{todonotes}

\usepackage[bbgreekl]{mathbbol}
\usepackage{tikz}

\usetikzlibrary{matrix}
\usetikzlibrary{patterns, shapes}
\usetikzlibrary{arrows}
\usetikzlibrary{calc,3d}
\usetikzlibrary{decorations,decorations.pathmorphing, decorations.pathreplacing}
\usetikzlibrary{through}
\tikzset{ext/.style={circle, draw,inner sep=1pt},int/.style={circle,draw,fill,inner sep=1pt},nil/.style={inner sep=1pt}}
\tikzset{exte/.style={circle, draw,inner sep=3pt},inte/.style={circle,draw,fill,inner sep=3pt}}
\tikzset{diagram/.style={matrix of math nodes, row sep=3em, column sep=2.5em, text height=1.5ex, text depth=0.25ex}}
\tikzset{diagram2/.style={matrix of math nodes, row sep=0.5em, column sep=0.5em, text height=1.5ex, text depth=0.25ex}}

\usepackage{tikz-cd}

\theoremstyle{plain}

\newtheorem{thm}{Theorem}[section]

\newtheorem{prop}[thm]{Proposition}
\newtheorem{cor}[thm]{Corollary}
\newtheorem{lemma}[thm]{Lemma}

\theoremstyle{definition}

\newtheorem{ex}[thm]{Example}
\newtheorem{rem}[thm]{Remark}

\newtheorem{notation}[thm]{Notation}

\newcommand{\p}{\partial}

\newcommand{\R}{{\mathbb{R}}}

\newcommand{\Z}{{\mathbb{Z}}}

\newcommand{\Q}{{\mathbb{Q}}}

\newcommand{\HGC}{{\mathrm{HGC}}}

\newcommand{\Graphs}{{\mathsf{Graphs}}}

\newcommand{\TGC}{{\mathrm{TGC}}}
\newcommand{\UTT}{{\mathrm{UTT}}}

 \newcommand{\mP}{\mathcal{P}}

 \newcommand{\mA}{\mathcal{A}}
 \newcommand{\mB}{\mathcal{B}}

\newcommand{\ICG}{\mathrm{IG}}

\newcommand{\mF}{\mathcal{F}}

\newcommand{\Def}{\mathrm{Def}}

\newcommand{\Poiss}{\mathsf{Pois}}

\newcommand{\op}{\mathcal}

\newcommand{\Lie}{\mathsf{Lie}}

\newcommand{\Com}{\mathsf{Com}}

\newcommand{\FM}{\mathcal{F}}
\newcommand{\IF}{\mathcal{IF}}
\newcommand{\maxx}{\mathrm{max}}

\newcommand{\bpm}{\begin{pmatrix}}
\newcommand{\epm}{\end{pmatrix}}

\newcommand{\MC}{\mathsf{MC}}

\newcommand{\mU}{\mathcal{U}}

\newcommand{\hotimes}{\mathbin{\hat\otimes}}

\newcommand{\Op}{\mathrm{Op}}
\newcommand{\La}{\Lambda}

\DeclareMathOperator{\Emb}{Emb}
\DeclareMathOperator{\Imm}{Imm}
\DeclareMathOperator{\Embbar}{\overline{Emb}}

\newcommand{\Mor}{\mathrm{Mor}}

\newcommand{\gr}{\mathrm{gr}}

\DeclareMathOperator{\TopCat}{Top}

\newcommand{\beq}[1]{\begin{equation}\label{#1}}
\newcommand{\eeq}{\end{equation}}

\newcommand{\IBimod}{\mathrm{IBimod}}

\newcommand{\hofiber}{\mathrm{hofiber}}

\newcommand{\SO}{\mathrm{SO}}

\newcommand{\Res}{\mathrm{Res}}
\newcommand{\Ind}{\mathrm{Ind}}
\newcommand{\ICGp}{\mathrm{pIG}}
\DeclareMathOperator*{\cro}{\mathrm{cr}}
\newcommand{\ar}{\mathrm{ar}}

\newcommand{\iHom}{{\mathcal{H}om}}
\newcommand{\HopfColl}{\DGCA^{\Sigma}}
\newcommand{\HopfGM}{\DGCA^{\Gamma}}
\DeclareMathOperator*{\colim}{\mathrm{colim}}
\newcommand{\mC}{{\mathcal C}}
\newcommand{\DGCA}{\mathsf{dgCom}}
\newcommand{\sset}{\mathsf{sset}}
\newcommand{\nuCom}{\Com_+}

\newcommand{\cosk}{\mathrm{cosk}}
\newcommand{\Indec}{\mathrm{Indec}}

\newcommand{\COp}{\mathsf{C}}
\newcommand{\ggR}{\mathcal R} % graded ground ring for section 3
\newcommand{\fg}{\mathfrak g}

\begin{document}
\title{On the rational homotopy type of embedding spaces of manifolds in $\R^n$}

\author{Benoit Fresse}
\address{Univ. Lille, CNRS, UMR 8524 -- Laboratoire Paul Painlev\'e, F-59000 Lille, France}
\email{Benoit.Fresse@univ-lille.fr}

\author{Victor Turchin}
\address{Department of Mathematics\\
Kansas State University\\
%138 Cardwell Hall\\
Manhattan, KS 66506, USA}
%\address{MPIM\\ Vivatsgasse 7\\ 53111 Bonn\\ Germany}
\email{turchin@ksu.edu}

\author{Thomas Willwacher}
\address{Department of Mathematics \\ ETH Zurich \\
R\"amistrasse 101 \\
8092 Zurich, Switzerland}
\email{thomas.willwacher@math.ethz.ch}

%\date{January 30, 2017}
 %\subjclass[2020]{58D10, 55P62, 18F50 (primary); 18G85, 18M70, 18M75, 55P48, 57R40, 57R42 (secondary).}

\thanks{B.F. acknowledges support from the Labex CEMPI (ANR-11-LABX-0007-01) and from the FNS-ANR project OCHoTop (ANR-18CE93-0002-01). V.T. has benefited from a visiting position of the Labex CEMPI (ANR-11-LABX-0007-01) at the Universit\'e de Lille and a visiting  position at the Max Planck Institute for Mathematics in Bonn for the achievement of this work. He  was also partially supported by the Simons Foundation travel grant, award ID:~519474.  B.F. and V.T. are grateful to the ETH for hospitality.
T.W. has been partially supported by the NCCR SwissMAP funded by the Swiss National Science Foundation, and the ERC starting grant GRAPHCPX (ERC StG 678156)}

%\thanks{The author was partially supported by the Swiss National Science Foundation (grant 200020-105450).}
% \subjclass[2000]{16E45; 53D55; 53C15; 18G55}
% \date{}
%\keywords{Formality, Deformation Quantization, Operads}

\begin{abstract}
We study the spaces of embeddings of manifolds in a Euclidean space.
More precisely we look at the homotopy fiber of the inclusion of these spaces to the spaces of immersions.
As a main result we express the rational homotopy type of connected components of those embedding spaces through combinatorially defined $L_\infty$-algebras of diagrams.
\end{abstract}

\maketitle

%\setcounter{tocdepth}{1}
%\tableofcontents

\sloppy

\section{Introduction}

\subsection{Background}
Let $L$ be a smooth compact submanifold possibly with boundary in $\R^m$.
Let $\Emb(L,\R^n)$ be the space of smooth embeddings $L\hookrightarrow\R^n$, $n-m\geq 2$.
We allow $L$ to be disconnected with components of possibly different dimensions.
For example, one can choose $L=\coprod_{k=1}^r S^{m_k}\subset\R^m$, $m=\maxx(m_k)+1$.
The sets of isotopy classes of such links were studied by Haefliger in~\cites{Haefliger1,Haefliger2,Haefliger2.5}.

Let $i: L\hookrightarrow\R^n$ be the trivial embedding, which we define by composing the canonical inclusion $L\subset\R^m$ with a fixed linear embedding of $\R^m$
into $\R^n$ (for instance, we can choose to identify $\R^m$ with the subspace spanned by the $m$ first coordinate axes of $\R^n$).
Define the \emph{space of embeddings modulo immersions} as the homotopy fiber over the trivial embedding $i$
of the obvious map from the space of smooth embeddings to the space of immersions:
\beq{equ:hofib}
\Embbar(L,\R^n) := \hofiber(\Emb(L,\R^n)\to\Imm(L,\R^n)).
\eeq
Roughly, $\Embbar(L,\R^n)$ can be thought of as a space of embeddings trivial as immersions.
More precisely a point in such a space is a pair $(f,h)$, where $f: L\hookrightarrow\R^n$ is a smooth embedding and $h: L\times [0,1]\to\R^n$
is a regular homotopy from $f$ to the inclusion $i$.
The rational homology groups of these spaces $\Embbar(L,\R^n)$ have been extensively studied in~\cites{AroneLambVol,Turchin2}
under the codimension restriction %requirement
 $n\geq 2m+1$. (Note that the embedding spaces are connected in this codimension range.)
In this paper we enlarge the range of codimension to $n\geq m+2$ and we describe the rational homotopy type of the components of such spaces.
In addition we show that the results of rational homotopy theory allow one to obtain some information about the set of isotopy classes $\pi_0\Embbar(L,\R^n)$.

For technical reasons we replace $L$ by its open tubular neighborhood $N(L)\subset\R^m$.
Note that the map $L\hookrightarrow N(L)$ induces a weak homotopy equivalence on our embedding spaces (see~\cite[Proof of Proposition 1.2]{STT}):
\beq{eq:L_NL}
\Embbar(N(L),\R^n)\xrightarrow{\simeq}\Embbar(L,\R^n).
\eeq

We also want to study the case where some of the components of the manifold $L$ extend to infinity.
We assume for simplicity that any such unbounded component of $L$ is closed and coincides with a Euclidean subspace near infinity.
We then consider the space of embeddings modulo immersions $\Embbar_\p(L,\R^n)$
in which embeddings and immersions coincide with the inclusion $L\subset\R^m\subset\R^n$
outside some ball of a sufficiently big radius.
For example, one can take $L = \coprod_{i=1}^r\R^{m_i}\subset\R^m$, $m = \maxx(m_i)+1$ (or just $m = m_1$ if $r=1$),
a disjoint union of $r$ affine spaces $\R^{m_i}$ parallel to the coordinate plane
spanned by the first $m_i$ basis vectors.
These spaces $\Embbar_\p(\coprod_{i=1}^r\R^{m_i},\R^n)$, called \emph{spaces of string links modulo immersions}, have been studied in~\cite{Turchin2,Turchin3,Song,STT}.
It was shown in~\cite[Corollary~2.6]{STT} that the rational homotopy groups of these spaces are computed by certain hairy graph complexes, provided that $n\geq 2\maxx(m_i)+2$.
Moreover, it was conjectured in~\cite[Conjecture~3.1]{STT} that the result holds in the lower range $n\geq\maxx(m_i)+3$.
This conjecture is proven by our Corollary~\ref{cor_main}.
Note that the sets $\pi_0\Embbar(\coprod_{i=1}^r S^{m_i},\R^n) = \pi_0\Embbar_\p(\coprod_{i=1}^r \R^{m_i},\R^n)$
form finitely generated abelian groups
in this range (see~\cite{Haefliger1,Haefliger2,Haefliger2.5}, \cite[Theorem~4.2, Lemma~4.7]{STT}, Proposition~\ref{pro:pi0}).
We study these spaces and we discuss the comparison of our results with Haefliger's in sections~\ref{ss:strings}-\ref{ss:emb_long}.

For technical reasons, we again replace such a ``long manifold'' $L$
by $N_\infty(L) := N(L)\cup(\R^m\setminus D^m_R)$,
the tubular neighborhood of $L$ in $\R^m$ union the complement of a ball of a sufficiently big radius $R$.
We still have a weak homotopy equivalence:
\beq{eq:L_NL2}
\Embbar_\p(N_\infty(L),\R^n)\xrightarrow{\simeq}\Embbar_\p(L,\R^n).
\eeq
Note that the bounded framework (the case where $L\subset\R^m$ is a compact submanifold) can be reduced to the unbounded one
by taking $N_{\infty}(L) = N(L)\cup(\R^m\setminus D^m_R)$,
where $D^m_R$ is a ball that contains $N(L)$,
so that:
\beq{eq:L_NL3}
\Embbar(N(L),\R^n)\simeq\Embbar_\p(N_\infty(L),\R^n).
\eeq

\subsection{Main results}
From now on we assume that $M\subset\R^m$ is the complement of a compact submanifold in $\R^m$ (possibly with boundary).
We also consider the space $M_* := M\cup\{\infty\}$, topologized as a subset of $S^m = \R^m\cup\{\infty\}$.
The Goodwillie--Weiss Taylor tower of $\Embbar_\p(M,\R^n)$ has been expressed by Arone and the second author through derived mapping spaces of infinitesimal bimodules
(see \cite[Proposition~6.9]{Turchin2}, \cite[Section~6]{Turchin4}).
The results obtained in these references, combined with the convergence result of Goodwillie, Klein and Weiss~\cite{GK,GKW,GW}\footnote{For the convergence one needs the ambient dimension $n$ minus the handle dimension of $M_*$ to be at least three.},
imply that, for $n-m\geq 3$ or $n-m\geq 2$ and $M\subsetneq\R^m$,
one has a weak equivalence
\beq{equ:AT}
\Embbar_\p(M,\R^n)\simeq\IBimod^h_{\FM_m}(\IF_M,\FM_n),
\eeq
where $\FM_k$ is the Fulton--MacPherson operad equivalent to the little discs operad $E_k$ (see~\cite{Salvatore}),
$\IF_M$ is a sequence of locally compactified configuration spaces of points in $M_*$,
and $\IBimod^h_{\FM_m}(\dots)$ is the derived mapping space in the category of infinitesimal bimodules (see Section~\ref{ss:FM}).
In the case of a bounded submanifold, one has a similar formula in which instead of infinitesimal bimodules,
the mapping space of right modules is used (see~\cite[Proposition~6.9]{Turchin2}, \cite[Proposition~6.1]{WBdB}, \cite[Theorem~2.1]{Turchin4}).
The additional infinitesimal left action is necessary to encode the behavior of embeddings at infinity.
We assume that the reader is familiar with this formula and with the underlying algebraic objects.

The main goal of this paper is to understand the right-hand side of \eqref{equ:AT} in the realm of the rational homotopy theory.
We formulate our result in terms of graph complexes.
We use the Sullivan rational homotopy theory and the Sullivan differential graded algebras of piece-wise linear differential forms $\Omega^*(X)$
which can be associated to any space $X$ (for instance, we are going to consider the case $X = M_*$).
We say, by abusing classical conventions, that a differential graded commutative algebra $R$
defines a \emph{Sullivan model} of the space $X$
when $R$ is quasi-isomorphic to the Sullivan differential graded algebra $\Omega^*(X)$.\footnote{In the literature, a Sullivan model usually refers to an object equipped with a connected cell complex structure in the model category of differential graded commutative algebras.
But we do not need to require the existence of such a cell complex structure for our models in general.
Therefore, we abusively use the phrase ``Sullivan model'' to refer to any differential graded commutative algebra quasi-isomorphic to $\Omega^*(X)$.
Besides we do not make any assumption on the spaces $X$ to which we apply this concept.
In particular, we do not necessarily assume that $X$ is connected,
so that the definition of a connected cell complex model of the differential graded algebra $\Omega^*(X)$
does not make sense in our context
in general.}
%Recall also the Sullivan rational homotopy theory does not apply to the study of the rational homotopy type of spaces
%in such cases. (We have to assume the connectedness of the space, at least, together with mild finiteness conditions
%in order to get such results.)
%The point is that our aim is to determine the rational homotopy type of the spaces of embeddings,
%and not the rational homotopy type of the embedded spaces themselves,
%which may therefore not be good objects for the rationalization
%in the category of spaces.}

Let $A$ be some (possibly non-unital) differential graded commutative algebra.
(In our result, we take for $A$ the augmentation ideal of a Sullivan model of the pointed space $M_*$.)
We denote by $\HGC_{A,n}$ the complex of $\Q$-linear series of graphs with ``hairs'' (external legs),
each of which decorated by an element of our algebra,
as in the following example:
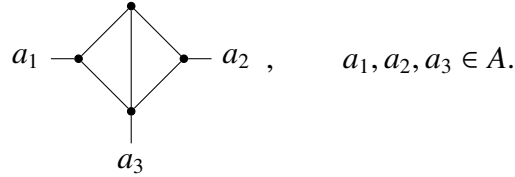
\[
\begin{tikzpicture}[scale=.7,baseline=-.65ex]
\node[int] (v1) at (-1,0){};
\node[int] (v2) at (0,1){};
\node[int] (v3) at (1,0){};
\node[int] (v4) at (0,-1){};
\node (w1) at (-2,0) {$a_1$};
\node (w2) at (2,0) {$a_2$};
\node (w3) at (0,-2) {$a_3$};
\draw (v1)  edge (v2) edge (v4) edge (w1) (v2) edge (v4) (v3) edge (v2) edge (v4) (v4) edge (w3) (v3) edge (w2);
\end{tikzpicture}
\,,\quad\quad
a_1,a_2,a_3\in A.
\]
We assume that the graphs of this complex $\HGC_{A,n}$ are connected and have internal vertices of valence greater than or equal to three.
Multiple edges and tadpoles (edges connecting a vertex to itself) are allowed.
We also assume that the decoration of the hairs is $\Q$-multilinear in $A$.

We equip this decorated hairy graph complex $\HGC_{A,n}$ with a homological grading (as opposed to the differential graded algebra $A$,
which we equip with a cohomological grading, as usual in the Sullivan rational homotopy theory).
To be explicit, we endow a graph $\Gamma\in\HGC_{A,n}$ with $\# E$ edges, $\# V$ internal vertices,
and with hairs labelled by homogeneous elements $a_1,\ldots,a_k\in A$,
with the homological degree
\[
\deg(\Gamma) = (n-1)\# E-n\# V -\sum_{i=1}^k |a_i|,
\]
where $|a_i|$ denotes the cohomological degree of the elements $a_i$ in $A$.

We endow the decorated hairy graph complex with a differential graded $L_\infty$-algebra structure, which we briefly sketch for the moment,
referring to sections~\ref{sec:HGC} and~\ref{sec:HGCLie} for more details.
The differential $\delta: \HGC_{A,n}\rightarrow\HGC_{A,n}$ consists of three pieces
\begin{equation}\label{equ:delta}
\delta  = d_A+\delta_{split} + \delta_{join}.
\end{equation}
The piece $d_A$ is induced by the internal differential of the differential graded commutative algebra $A$.
The piece $\delta_{split}$ is given by the sum of the following splitting operation on internal vertices:
\begin{equation}\label{equ:deltasplit}
\begin{tikzpicture}[baseline=-.65ex]
\node[int] (v) at (0,0) {};
\draw (v) edge +(-.3,-.3)  edge +(-.3,0) edge +(-.3,.3) edge +(.3,-.3)  edge +(.3,0) edge +(.3,.3);
\end{tikzpicture}
\mapsto\sum\begin{tikzpicture}[baseline=-.65ex]
\node[int] (v) at (0,0) {};
\node[int] (w) at (0.5,0) {};
\draw (v) edge (w) (v) edge +(-.3,-.3)  edge +(-.3,0) edge +(-.3,.3)
 (w) edge +(.3,-.3)  edge +(.3,0) edge +(.3,.3);
\end{tikzpicture},
\quad\text{so that}
\quad\delta_{split}\Gamma = \sum_{v\,\text{vertex}} \pm\Gamma\,\text{split}\,v.
\end{equation}
The piece $\delta_{join}$ is defined by the sum of the operations that consists of joining a subset of the hairs of our graph
into one hair and multiplying the corresponding decorations in $A$:
\begin{align}\label{equ:deltajoin}
\delta_{join}
\begin{tikzpicture}[baseline=-.8ex]
\node[draw,circle] (v) at (0,.3) {$\Gamma$};
\node (w1) at (-.7,-.5) {$a_1$};
\node (w2) at (-.25,-.5) {$a_2$};
\node (w3) at (.25,-.5) {$\dots$};
\node (w4) at (.7,-.5) {$a_k$};
\draw (v) edge (w1) edge (w2) edge (w3) edge (w4);
\end{tikzpicture}
= \sum_{\substack{S\subset {\rm hairs} \\ |S|\geq 2 }} \pm
\begin{tikzpicture}[baseline=-.8ex]
\node[draw,circle] (v) at (0,.3) {$\Gamma$};
\node (w1) at (-.7,-.5) {$a_1$};
\node (w2) at (-.25,-.5) {$\dots$};
\node[int] (i) at (.4,-.5) {};
\node (w4) at (.4,-1.3) {$\scriptstyle \prod_{j\in S}a_j$};
\draw (v) edge (w1) edge (w2) edge[bend left] (i) edge (i) edge[bend right] (i) (w4) edge (i);
\end{tikzpicture}\,.
\end{align}
The higher $L_\infty$-operations are defined similarly to~$\delta_{join}$. In brief, the $r$th $L_\infty$-operation $\ell_r(\Gamma_1,\dots,\Gamma_r)$, $r\geq 2$,
is the sum of the graphs that we obtain by joining hair subsets of the graphs $\Gamma_1,\dots,\Gamma_r$
all together and by multiplying the corresponding decorations.
For example, the Lie bracket has the following schematic description:
\beq{equ:bracketpic}
\left[\begin{tikzpicture}[baseline=-.8ex]
\node[draw,circle] (v) at (0,.3) {$\Gamma$};
\draw (v) edge +(-.5,-.7) edge +(-.25,-.7) edge +(0,-.7) edge +(.25,-.7) edge +(.5,-.7);
\end{tikzpicture},
\begin{tikzpicture}[baseline=-.65ex]
\node[draw,circle] (v) at (0,.3) {$\Gamma'$};
\draw (v) edge +(-.5,-.7) edge +(-.25,-.7) edge +(0,-.7) edge +(.25,-.7) edge +(.5,-.7);
\end{tikzpicture}
\right]
= \sum\begin{tikzpicture}[baseline=-.8ex]
\node[draw,circle] (v) at (0,.3) {$\Gamma$};
\node[int] (i) at (.5,-.5) {};
\draw (v) edge +(-.5,-.7) edge +(0,-.7) edge (i) edge[bend left] (i) edge[bend right] (i);
\node[draw,circle] (vv) at (1,.3) {$\Gamma'$};
\draw (vv) edge (i) edge[bend left] (i) edge[bend right] (i) edge +(0,-.7) edge +(.5,-.7) (i) edge (.5,-1);
\end{tikzpicture}\,,
\eeq
where we suppress the $A$ decorations on hairs for simplicity. (They are multiplied whenever hairs are joined.)
Note that, by convention, this Lie bracket has degree $-1$ (just as the Whitehead product in homotopy theory)
and we adopt a similar convention for the higher $L_{\infty}$-operations. Thus, in comparison to the standard grading convention
for $L_{\infty}$-algebra structures, we shift the degree of the $L_{\infty}$-operations
by one (see Section~\ref{ss:gen}).

Our main result is the following.

\begin{thm}\label{thm:main1}
Let $M\subset\R^m$ be a complement to a compact submanifold (possibly with boundary) and let $R$ be a Sullivan model of the pointed space $M_* = M\cup\{\infty\}$ (with the base-point at infinity).
We assume that $R$ is equipped with an augmentation (corresponding to the base point).
For $n-m\geq 2$, we have a weak homotopy equivalence
\[
 \IBimod^h_{\FM_m}(\IF_M,\FM_n^\Q)\simeq\MC_\bullet(\HGC_{\bar R,n}),
\]
where $A = \bar R$ denotes the augmentation ideal of our differential graded commutative algebra $R$
and $\MC_\bullet(\HGC_{\bar R,n})$ denotes the simplicial set of Maurer--Cartan forms
with values in the complete $L_\infty$-algebra $\HGC_{\bar R,n}$.\footnote{We recall the explicit definition of a Maurer--Cartan element in a complete $L_\infty$-algebra
and the definition of this simplicial set of Maurer--Cartan forms
in Section~\ref{ss:MC}.}
\end{thm}

To relate our computation to the right-hand side of \eqref{equ:AT} one may invoke the following result.

\begin{thm}\label{thm:FM}
Let $M$ be a complement to a compact submanifold in $\R^m$. For $n-m\geq 3$
or $n-m\geq 2$ and $M\subsetneq\R^m$,
the natural map
\beq{equ:FM}
\IBimod^h_{\FM_m}(\IF_M,\FM_n)\to\IBimod^h_{\FM_m}(\IF_M,\FM_n^\Q)
\eeq
defines a rational equivalence of nilpotent spaces componentwise
and is finite-to-one at the $\pi_0$-level.
\end{thm}

From Theorem \ref{thm:main1} and equivalence \eqref{equ:AT}, it follows that one can compute the rational homotopy type of the connected components $\Embbar_\p(M,\R^n)_\psi$
of the embedding spaces $\Embbar_\p(M,\R^n)$
through our hairy graph complexes.

\begin{cor}\label{cor_main}
Under the assumptions of Theorem \ref{thm:FM}, we have the following statements:
\begin{enumerate}
\item\label{cor_main:nilpotence}
Every component of $\Embbar_\p(M,\R^n)$ is nilpotent.
\item\label{cor_main:connected_components}
We have a naturally defined finite-to-one map from the set of connected components of the space $\Embbar_\p(M,\R^n)$
to the set of Maurer--Cartan elements of the $L_{\infty}$-algebra $\HGC_{\bar R,n}$
modulo gauge equivalence:
\beq{equ:MC}
m: \pi_0\Embbar_\p(M,\R^n)\to\MC(\HGC_{\bar R,n})/\sim.
\eeq
\item\label{cor_main:Quillen_model}
The model of the rational homotopy type of a connected component $\Embbar_\p(M,\R^n)_\psi$
is given by the positive degree truncation
of the twisted $L_{\infty}$-algebra $\HGC_{\bar R,n}^{m(\psi)}$,
which we associate to the Maurer--Cartan element $m(\psi)\in\MC(\HGC_{\bar R,n})$
corresponding to $\psi\in\Embbar_\p(M,\R^n)$.
\item\label{cor_main:unknot}
The Maurer--Cartan element that corresponds to the trivial embedding $i: M\hookrightarrow\R^n$ is $m(i)=0$.
\end{enumerate}
\end{cor}

Throughout the paper, the \emph{positive degree truncation} of an $L_\infty$-algebra $\mathfrak{g}$
refers to the $L_\infty$-subalgebra $\mathfrak{g}_{>0}\subset\mathfrak{g}$
which agrees with $\mathfrak{g}_k$ in degree $k\geq 2$,
is given by the kernel $\ker(\mathfrak{g}_1\xrightarrow{d}\mathfrak{g}_0)$ in degree $k = 1$,
and vanishes in degree $k\leq 0$.
%  $$\begin{cases} \mathfrak{g}_i,&i\geq 2,\\ \ker(\mathfrak{g}_1\xrightarrow{d}\mathfrak{g}_0),&i=1,\\
%0,& i\leq 0.\end{cases}$$

The third statement of the corollary (\ref{cor_main:Quillen_model}) implies the identity
\[
\pi_k^\Q\Embbar_\p(M,\R^n)_\psi\cong H_k(\HGC_{\bar R,n}^{m(\psi)}),
\]
for all $k\geq 1$,
where $\pi_k^\Q$ denotes the rationalization of the abelian group $\pi_k$
for $k\geq 2$,
and the Malcev completion of the nilpotent group $\pi_1$
for $k=1$ (see~\cite[Appendix A.3]{Quillen}, see also \cite[Chapter I.8]{FrI} for a general introduction to the Malcev completion of groups).
The product on $\pi_1^\Q = H_1(\HGC_{\bar R,n}^{m(\psi)})$ is given by the Baker-Campbell-Hausdorff formula (see for example~\cite[Theorem~1.1]{Be}).
The action of $\pi_1^\Q$ on $\pi_k^\Q$, $k\geq 2$, is expressed as the exponent of the adjoint action of $H_1(\HGC_{\bar R,n}^{m(\psi)})$
on $H_k(\HGC_{\bar R,n}^{m(\psi)})$ (see \cite[Section~12.5.1]{BFMT}).
The rational homology groups of $\Embbar_\p(M,\R^n)_\psi$ can also be computed as the Chevalley--Eilenberg homology
of the $L_{\infty}$-algebra $(\HGC_{\bar R,n}^{m(\psi)})_{>0}$ (see~\cite[Corollary~1.3]{Be}, \cite{Getzler}, Section~\ref{s:Sullivan}).
The last statement of the corollary (\ref{cor_main:unknot})
follows from the fact that the component of the trivial embedding
under the composition
of~\eqref{equ:AT} and~\eqref{equ:FM} is sent to the component of the map $\IF_M\to\FM_m\xrightarrow{*}\FM^\Q_n$
that factors through the commutative operad in sets $\Com$ (see Section~\ref{ss:relative}).
Recall that $\Com$ is given by the base-point in each arity $\Com(r) = *$,
so that we may write $\Com = *$ for this operad
in what follows.

Theorem~\ref{thm:FM} follows from an application of Mienn\'e's theory of Postnikov decompositions of operads
and of bimodules over operads (see~\cite{MienneThesis,MienneMemoir}).
This theory enables one to adapt Haefliger's proof that the rationalization of the space of sections
of a nilpotent fibration can be computed by a model (see~\cite{Haefliger3,Haefliger4}).
We also refer to~\cite{Sullivan} for the claim that the rationalization induces a finite-to-one correspondence
on the sets of homotopy classes of maps with values in a nilpotent space.
The proof of these counterparts of the claims of Theorem~\ref{thm:FM}
in the context of spaces can be obtained by using a Postnikov decomposition of the target object
of our mapping spaces.
In the context of operads, we actually need to consider a decomposition by arity in addition to the decomposition by Postnikov sections.
To ensure the convergence of this double decomposition, we need to prove that the source object of our mapping space
is equivalent to a cell complex of free infinitesimal bimodules
of bounded dimension arity-wise,
and that the fibers of the arity decomposition of the target object have a connectivity $n(r)$ that tends to $\infty$
faster than the dimension bounds of the cells
of the source object.
We basically check that these assumptions are fulfilled by the infinitesimal bimodules $\IF_M$ and $\FM_n$
in order to get the conclusion of Theorem~\ref{thm:FM}
for our mapping spaces.
We explain this verification in detail in Section~\ref{ss:proof_thm_FM}.

%in our case, we have to use that the maps between the stages
%of the arity tower for $\IBimod^h_{\FM_m}(\IF_M,\FM_n)$ become higher and higher connected
%when the arity increases to infinity. (The latter tower is equivalent to the Goodwillie-Weiss tower.)

The results outlined in this introduction enable us to define a map
\[
\pi_0\Embbar_\p(M,\R^n)\cong\pi_0\IBimod^h_{\FM_m}(\IF_M,\FM_n)\to\pi_0\MC_\bullet(\HGC_{\bar R,n}) = \MC(\HGC_{\bar R,n})/{\sim},
\]
and hence, an invariant of ``$M$-knots''. We know that this map is finite-to-one. However, we do not know yet how this map can be efficiently computed.
It is likely that it can be expressed in terms of Bott-Taubes-Kontsevich type configuration space integrals.
It is also possible that it can be expressed more simply in terms of the rational homotopy type of the complement of the embedding $\R^n\setminus\psi(M)$
and maybe in addition taking into account a chain version of the Alexander duality.

\subsection{Range improvement}\label{ss:range1}
We reiterate that even though Theorem~\ref{thm:main1} and Corollary~\ref{cor_main} are stated for a manifold $M\subset\R^m$, which is a complement
to a compact submanifold, because of the equivalences~\eqref{eq:L_NL} and \eqref{eq:L_NL3}, Corollary~\ref{cor_main} applies to any compact manifold
$L$ (with components of possibly different dimensions) embeddable in $\R^m$.
For $\bar R$ in $\HGC_{\bar R,n}$, we take a Sullivan model of~$L$. Similarly, because of equivalence~\eqref{eq:L_NL2}, Corollary~\ref{cor_main} can be applied to a closed submanifold $L\subset\R^m$, which near infinity looks like a finite disjoint
union of affine subspaces.
In the latter case for $R$ we take an augmented Sullivan model of the one-point compactification $L_*=L\cup\{\infty\}$ of $L$,
considered as pointed at~$\infty$.

The range when Corollary~\ref{cor_main} applies can be slightly improved. Namely, the manifold $L$ does not need to be embeddable in~$\R^m$, it is enough
if it admits an immersion $i\colon L\looparrowright\R^m$. In case $L$ is not compact, $i$ is supposed to be proper, injective outside a compact subset,
and to have as image near infinity a disjoint union of affine subspaces. We can then similarly define spaces $\Embbar(L,\R^n)$ and $\Embbar_\p(L,\R^n)$.
Define $N(L)$ as the normal disc bundle over $L$, and $N_\infty(L)=:M$ as $N(L)$ union $\R^m\setminus D^m_R$ the complement to a closed disc of some big radius $R$.
The immersion $i$ can be extended to an immersion $N(L)\looparrowright\R^m$, respectively $N_\infty(L)\looparrowright\R^m$.
Then in this more general situation, the equivalences (\ref{eq:L_NL}-\ref{eq:L_NL3}) still hold.

 For the generalized Theorem~\ref{thm:main1} and Corollary~\ref{cor_main}, the manifold $M$ can be taken as follows. First let $|M_*|$ be any compact pointed $m$-manifold with a basepoint~$*$ in
its interior $M_*\subset |M_*|$. Let also
\[
i\colon |M_*|\looparrowright S^m=\R^m\cup\{\infty\},
\]
be an immersion such that $i^{-1}(\infty)=*$. The manifold $M$ is then defined as $M:=M_*\setminus \{*\}$. We  consider
the spaces $\Emb_\p(M,\R^n)$, $\Imm_\p(M,\R^n)$, $\Embbar_\p(M,\R^n)$ of embeddings, immersions, and embeddings modulo immersions for which
the corresponding maps $M\to\R^n$ coincide with $M\xrightarrow{i}\R^m\subset\R^n$ near~$*$.

Corollary~\ref{cor_main}(\ref{cor_main:nilpotence}-\ref{cor_main:Quillen_model}) for a manifold $M$ as above can be proved by exactly the same arguments, provided $n-m\geq 3$ or $n-m\geq 2$ and
$M_*$ has no component $S^m$. We need the codimension-two-requirement because in the proof we use the relative formality of the little discs operads (see Subection~\ref{ss:relative}).
Besides, we need to make sure that the ambient dimension $n$ minus the handle dimension of $M_*$ is greater than or equal to three,
because this assumption is necessary for the convergence of the Goodwillie-Weiss tower~\cite{GK,GKW,GW},
and for this reason, we do not allow the manifold $M$ to have a closed
component in the case where the codimension is two\footnote{Such a component must be $S^m$, as an immersion of a closed $m$-manifold in an $m$-manifold is always a covering map.}.
In Section~\ref{ss:range2} we explain how the main steps in the proof need to be adjusted for this more general situation.

\subsection{Plan of the paper}
The paper is organized as follows.
In Section \ref{sec:preliminaries}, we set our notation for later use and we briefly recall constructions of the literature that we use throughout the paper.
More specifically, we review Pirashvili's Dold--Kan theory of $\Gamma$-modules, which we use in our study of our mapping spaces of infinitesimal bimodules over operads,
and we review relative formality results for the little discs operads, which we use to reduce our mapping spaces to hairy graph complexes
in the rational homotopy theory setting.
We also briefly review the definition of the Fulton--MacPherson operad $\FM_n$ and of the infinitesimal bimodule $\IF_M$
%and we check the validity of the claims of Theorem~\ref{thm:FM}
in this section.
We introduce the hairy graph complexes $\HGC_{A,n}$ in Section~\ref{s:hopf}
and we complete the proof of our main theorems afterwards in Section~\ref{sec:the proofs}.
We eventually address some examples and applications in Section~\ref{s:examples} where we also compare our results with previous work of the literature.% on the classification of embeddings in higher dimensions.

\subsection*{Acknowledgement} The authors thank G.~Arone, U.~Buijs, Y.~F\'elix, P.~Lambrechts, A.~Skopenkov, D.~Stanley and D.~Tanr\'e for communication.

\section{Preliminaries}\label{sec:preliminaries}

\subsection{Generalities}\label{ss:gen}
We generally work over the ground field $\Q$ (i.e., all vector spaces, algebras, etc. will be defined over $\Q$).
The phrase \emph{differential graded} will be abbreviated by \emph{dg}.
By a dg vector space we usually mean a cochain complex, which can also be considered as a chain complex by inversing the grading.

In general, we apply grading conventions inspired by topological applications.
For instance, we apply cohomological conventions for dg commutative algebras as usual in rational homotopy theory (in particular, we assume that the differential of a dg commutative algebra
increases degrees by $1$), but we adopt homological conventions
for dg Lie algebras and $L_\infty$-algebras (thus, we consider a differential that decreases degrees by $1$
in the case of a dg Lie algebra).
In addition, we assume that the bracket of a dg Lie algebra has degree $-1$ (like the Whitehead product in homotopy theory).
We similarly assume that all higher brackets operations of an $L_{\infty}$-algebra have degree $-1$.
With these grading conventions the Maurer--Cartan elements lie in degree~$0$.

We denote the category of non-negatively graded dg commutative algebras by $\DGCA$.
We equip it with the usual model structure, in which the weak equivalences are the quasi-isomorphisms and the fibrations are the maps that are surjective in all degrees.

In general, for a simplicial model category $\mC$, we use the name of the category to refer to the mapping space $\mC(A,B)$, which we can associate to any objects $A,B\in\mC$.
We also use the notation $\mC(A,B)$ when $\mC$ is not simplicial. (In this case, we assume that we have a canonical choice of simplicial frame or of cosimplicial frame.)
The derived mapping space, denoted by $\mC^h(A,B)$, is defined by taking $\mC^h(A,B) = \mC(Q_A,R_B)$ for a cofibrant replacement $Q_A$ of $A$
and a fibrant replacement $R_B$ of $B$.

We denote the morphism sets of a category by~$\Mor_\mC(\dots)$ in general, but we will simplify this notation for certain categories of diagrams in dg vector spaces
and in dg commutative algebras (see Section~\ref{ss:hopf_model_str}).
%At one instance we also use the graded vector space hom on $I$-shaped diagrams in graded vecor spaces,  which we denote by $\iHom_I(\cdots)$.

\subsection{Topological spaces and simplicial sets}\label{sec:Milnor_equivalence}
By $\TopCat$ we denote the category of compactly generated (possibly non-Hausdorff) topological spaces.
This category is cartesian closed, like the category of simplicial sets $\sset$,
and moreover we have a Quillen equivalence of model categories
that preserves the cartesian closed structure (see~\cite[Section~2.4]{Hov}\footnote{This category is denoted by $\textbf{K}$ in this reference.}):
\[
|\,.\,|: \sset\rightleftarrows\TopCat :S_\bullet,
\]
where we consider the geometric realization functor $|\,.\,|$ on the one hand
and the singular complex functor $S_\bullet$
on the other hand.
By a ``space'' we understand either an object in $\TopCat$ or a simplicial set.
Because of this equivalence, sometimes we will be sloppy and will not make a difference between these two categories.

\subsection{On $\Gamma$-modules and $\Omega$-modules}\label{sec:gammaomega}
Let $\Gamma$ be the category of finite pointed sets. We use the expression `right $\Gamma$-module' for the category of contravariant functors $T: \Gamma^{op}\to\mC$
with values in any category $\mC$,
whereas the expression `left $\Gamma$-module' refers to the category of covariant functors $T: \Gamma\to\mC$.
We also consider the category $\Omega$ of finite (non-pointed) sets with surjections as morphisms.
We then use the expression `right $\Omega$-module' for the category of contravariant functors $T: \Omega^{op}\to\mC$
and `left $\Omega$-module' for the category of covariant functors $T: \Omega\to\mC$.
We denote the category of left $\Gamma$-modules in $\mC$ by $\mC^\Gamma$, and the category of left $\Omega$-modules by $\mC^\Omega$.

Suppose now that $\op M$ is a left $\Gamma$-module in some abelian category $\mC$.
For $S_*=\{*\}\sqcup S$ a pointed set and $s\in S$, we define $\pi_{s,S_*}: S_*\to S_*\setminus\{s\}$
to be the map that sends $s$ to the base-point and all other elements to themselves.
Then we have a left $\Omega$-module $\cro\op M$, the cross-effect, such that $\cro\op M(S)\subset\op M(S_*)$
is the joint kernel of the maps $\op M(\pi_{s,S_*})$.
The following result is due to Pirashvili (see~\cite{Pirashvili}).

\begin{prop}[Pirashvili]\label{prop:crosseffect}
The cross-effect functor $\cro: \mC^\Gamma\to\mC^\Omega$ between left $\Gamma$-modules and left $\Omega$-modules in an abelian category $\mC$ is an isomorphism of categories.
In particular, for a pair of left $\Gamma$-modules $A,B\in\mC^\Gamma$, we have the identity:
\[
\Mor_{\mC^\Gamma}(A,B) = \Mor_{\mC^\Omega}(\cro A,\cro B).
\]
\end{prop}

In what follows, we apply this statement in the case where $\mC$ is the category of dg vector spaces.
We also consider left $\Gamma$-modules in the category of dg commutative algebras (which is non abelian).
We use the terminology `left Hopf $\Gamma$-modules' for this category of left $\Gamma$-modules.
In general, if $\mC$ is a cofibrantly generated model category, then we can equip the category of left $\Gamma$-modules in $\mC$
with the projective model structure.
The weak equivalences in this model category $\mC^\Gamma$ are the objectwise weak equivalences and the fibrations are the objectwise fibrations.

\subsection{Right $\Gamma$-modules in topological spaces and in simplicial sets}\label{sec:gammaomega2}
For our purpose, we also consider the injective model structure on the category of right $\Gamma$-modules in simplicial sets $\sset^{\Gamma^{op}}$.
The weak equivalences of this model category are the objectwise weak equivalences again, while we take objectwise cofibrations as cofibrations.

The category $\sset^{\Gamma^{op}}$ can also be equipped with a Reedy model structure, defined in~\cite{BM2},
with the same class of weak-equivalences as the projective and injective model structures,
but where the class of cofibrations and the fibrations depends on the definition of latching maps
and of matching maps in the category of right $\Gamma$-modules.
This definition of a Reedy model structure on $\sset^{\Gamma^{op}}$ is a generalization
of the classical definition of the Reedy model category of diagrams
on a Reedy indexing category in the case where the indexing category, like $\Gamma^{op}$,
contains non trivial automorphisms (see also~\cite{BM2} for this subject).

The Reedy model structure on the category of right $\Gamma$-modules in simplicial sets
is Quillen equivalent to both the injective model structure
and the projective model structure (use that the identity functor carries the Reedy cofibrations to injective cofibrations,
the Reedy fibrations to projective fibrations,
and preserves all weak-equivalences).
This observation implies that we can use any of these model structures to compute the (derived) mapping spaces
of right $\Gamma$-modules $\sset^{\Gamma^{op},h}(\cdots)$.
%(we can also use that mapping spaces only depends on the choice of the class of weak-equivalences by the result of~\cite[Proposition 4.4]{DwyerKan}).

We can also adapt the definition of the Reedy model structure for the category of right $\Gamma$-modules in topological spaces.
(We will explain in Section~\ref{sec:htpy theory ibimod} that this Reedy model structure is a particular case
of the Reedy model structures associated to the categories of operadic infinitesimal bimodules
that we use in this paper.)
The Quillen equivalence of Section~\ref{sec:Milnor_equivalence}, between topological spaces and simplicial sets,
extends to a Quillen equivalence between the Reedy model category of right $\Gamma$-modules in topological spaces
and the Reedy model category of right $\Gamma$-modules in simplicial sets:
\[
|\,.\,|: \sset^{\Gamma^{op}}\rightleftarrows\TopCat^{\Gamma^{op}} :S_\bullet,
\]
We can actually see that the functors $|\,.\,|$ and $S_\bullet$
preserve all weak-equivalences of right $\Gamma$-modules (not only the weak-equivalences
between cofibrant or fibrant objects),
because such an assertion holds in the model categories of simplicial sets and topological spaces.
We can again use this equivalence to pass from results on the homotopy of right $\Gamma$-modules in topological spaces
to results on the homotopy of right $\Gamma$-modules in simplicial sets.

\subsection{Rational homotopy theory of $\Gamma$-modules}\label{sec:gamma_rational_homotopy}
By standard results of rational homotopy theory (see \cite[Section 8]{BG}, see also \cite[Section II.7.2]{FrII}),
the Sullivan functor of piece-wise linear differential forms $\Omega^*$
and the adjoint Sullivan realization functor $G$
define a Quillen adjunction
\[
G: \DGCA\rightleftarrows\sset^{op} :\Omega^*.
\]
(Recall that $\DGCA$ denotes the category of non-negatively graded dg commutative algebras.)
The derived unit of this Quillen adjunction sends a space $X$ to its rationalization.
We explicitly have $X^\Q := LG(\Omega^*(X))$, for any $X\in\sset$,
where $LG$ denotes the derived functor of $G$.

In what follows, we also apply the Sullivan functor of piece-wise linear differential forms to topological spaces.
Then we set by an abuse of notation $\Omega^*(X) = \Omega^*(S_{\bullet}(X))$,
where we consider the image of our space $X$
under the singular complex functor $S_{\bullet}$ (see Section~\ref{sec:Milnor_equivalence}).
Note that the existence of the Quillen adjunction
implies that we have an equivalence of mapping spaces $\sset^h(X,LG(R))\simeq\DGCA^h(R,\Omega^*(X))$
when we consider the image of dg commutative algebra $R$
under the derived functor of the Sullivan realization functor $G$.
In particular, we have $\sset^h(X,Y^\Q)\simeq\DGCA^h(R,\Omega^*(X))$ when we consider the rationalization of a space $Y$,
for any choice of a dg commutative algebra $R$
quasi-isomorphic to $\Omega^*(Y)$. (Recall that we use the phrase ``Sullivan model of the space $Y$''
for any choice of such a dg commutative algebra $R$.)

By objectwise application of the functors $\Omega^*$ and $G$, we obtain an adjunction:
\beq{equ:GOmega2}
G: \DGCA^\Gamma\rightleftarrows(\sset^{\Gamma^{op}})^{op} :\Omega^*,
\eeq
between the category of right $\Gamma$-modules in simplicial sets $\sset^{\Gamma^{op}}$ and the category of left $\Gamma$-modules in dg commutative algebras $\DGCA^\Gamma$.
(Recall that we use the phrase `left Hopf $\Gamma$-module' to refer to this category of left $\Gamma$-modules $\DGCA^\Gamma$.)
The above adjunction is clearly a Quillen adjunction for the projective model structure on $\DGCA^\Gamma$
and the injective model structure on $\sset^{\Gamma^{op}}$
(the functor $\Omega^*$ preserves the weak equivalences and carries the cofibrations of the injective model structure on $\sset^{\Gamma^{op}}$
to fibrations in the projective model category of diagrams $\DGCA^\Gamma$
since it does so objectwise)\footnote{The same observation holds if we equip $\DGCA^\Gamma$ and $\sset^{\Gamma^{op}}$
with the Reedy model structure (using the general result of~\cite{BM} in the case of the category $\DGCA^\Gamma$),
but we only use the case of the projective and injective model structures in this paper,
because the left $\Gamma$-module in dg commutative algebras
that we consider in our applications
is naturally cofibrant with respect to the projective model structure.}.
Thus we can upgrade the classical Sullivan rational homotopy theory of spaces to right $\Gamma$-modules.

We can obviously compute the model of a right $\Gamma$-module $\Omega^*(X)$ objectwise
since the functor $\Omega^*$ preserves all weak-equivalences of simplicial sets.
In what follows, we also apply the Sullivan model functor to right $\Gamma$-modules in topological spaces
and we still write $\Omega^*(X) = \Omega^*(S_\bullet(X))$ in this case,
by the same abuse of notation as in the category of spaces.
We easily check that a cofibrant object of the category of left Hopf $\Gamma$-modules $\DGCA^\Gamma$
is cofibrant in the category of dg commutative algebras objectwise.
We deduce from this observation that the rationalization of a right $\Gamma$-module $X^\Q$
reduces to the rationalization of the spaces
underlying our object objectwise.
We still have the relation $\sset^{\Gamma^{op},h}(X,Y^\Q)\simeq\DGCA^{\Gamma,h}(R,\Omega^*(X))$ at the mapping space level,
for any choice of a left Hopf $\Gamma$-module $R$ quasi-isomorphic to $\Omega^*(Y)$.
We more generally have the relation $\sset^{\Gamma^{op},h}(X,LG(R))\simeq\DGCA^{\Gamma,h}(R,\Omega^*(X))$
when we take the image of a left Hopf $\Gamma$-module $R$
under the derived realization functor $LG$.
In subsequent arguments, we also use an extension of this relation for mapping spaces of right $\Gamma$-modules in topological spaces
by using that the Quillen equivalence between the categories of right $\Gamma$-modules in simplicial sets
and in topological spaces gives an equivalence at the mapping space level.

%We use the Reedy model structure on operads so that $\FM_m$ is cofibrant.

\subsection{Homotopy theory of operads and infinitesimal bimodules}\label{sec:htpy theory ibimod}
We let $\Sigma$ be the category of finite sets with bijections as morphisms and $\Lambda$ be the category of finite sets with injective maps as morphisms.
We again consider the category of right $\Sigma$-modules, which we define as the category of contravariant functors $T: \Sigma^{op}\to\mC$
with values in any category $\mC$,
the category of right $\Lambda$-modules, which consists of the contravariant functors $T: \Lambda^{op}\to\mC$,
and the symmetrically defined categories of left $\Sigma$-modules
and of left $\Lambda$-modules.
In the literature, the expression `symmetric sequence' or `collection' is also used for our categories of right $\Sigma$-modules.
In~\cite{FrI,FrII}, the expression `$\Lambda$-sequence' is used for the category of right $\Lambda$-modules while the expression `covariant $\Lambda$-sequence'
is used for the category of left $\Lambda$-modules.

In what follows, we notably consider the category of topological right $\Sigma$-modules $\TopCat^{\Sigma^{op}}$
and the category of topological right $\Lambda$-modules $\TopCat^{\La^{op}}$.
We equip $\TopCat^{\Sigma^{op}}$ with the projective model structure and $\TopCat^{\La^{op}}$ with the Reedy model structure (see~\cite[Sections II.8.1 and II.8.3]{FrII}).

The topological operads that we use in this paper are \emph{reduced} in the sense that their arity zero component is reduced to a point.
Recall that the underlying collection of a reduced operad inherits a right $\Lambda$-module structure.
We use this observation to equip the category of reduced operads with the Reedy model structure transferred from the category of right $\Lambda$-modules (see~\cite[Section II.8.4]{FrII}).
The Fulton--MacPherson operad $\FM_m$, of which we briefly recall the definition in the next subsection, is Reedy cofibrant.

Recall that an \emph{infinitesimal bimodule} over an operad $\mP$ is a right $\Sigma$-module $M$, endowed with a right $\mP$-module structure,
governed by right composition products $\circ_i: M(k)\otimes\mP(\ell)\to M(k+\ell-1)$, $k\geq 1$, $\ell\geq 0$, $1\leq i\leq k$,
together with a compatible \emph{infinitesimal left $\mP$-action},
governed by composition products of the form $\circ_i: \mP(k)\otimes M(\ell)\to M(k+\ell-1)$, $k\geq 1$, $\ell\geq 0$, $1\leq i\leq k$.
For details, see~\cite{DFT}. The category of infinitesimal $\mP$-bimodules is denoted by $\IBimod_\mP$.

For any reduced well-pointed topological operad $\mP$, the category $\IBimod_\mP$ can be equipped with the projective model structure,
transferred by adjunction from the projective model structure on $\TopCat^{\Sigma^{op}}$.
The category $\IBimod_\mP$ can also be equipped with the Reedy model structure,
which is transferred by adjunction from the Reedy model structure on $\TopCat^{\La^{op}}$.
(By~\cite[Theorem 5.1]{DFT}, this construction returns a valid model structure
as soon as the operad $\mP$
is well-pointed in the sense that the operadic unit defines
a cofibration of topological spaces in arity one $*\rightarrow\mP(1)$.)
%We say that the operad $\mP$ is \emph{Reedy-admissible} if this construction gives a valid model structure on $\IBimod_\mP$, to which we refer as the Reedy model structure.
%The class of Reedy-admissible operads includes the commutative operad by~\cite{BF} and the class of $\Sigma$-cofibrant operads by~\cite[Theorem 5.1]{DFT}.
%We conjecture that all reduced well-pointed operads are admissible, but we only use these cases in this paper.
%
% can be given two model structures: the projective model structure and the Reedy model structure.
% The first one is transferred from the projective model structure on $\TopCat^{\Sigma^{op}}$, while the second one is transferred from the Reedy model structure
% on $\TopCat^{\La^{op}}$ (see~\cite{DFT}).
%
Note that the derived mapping spaces $\IBimod_\mP^h(\cdots)$ coming from the projective and Reedy model structures
are equivalent (see~\cite[Theorem~5.9]{DFT},
or use the general result of \cite[Proposition 4.4]{DwyerKan}).
%\todo{Note that the Reedy model structure is defined without assuming the operads to be Sigma-cofibrant in the category of topological spaces.
%Thomas: I have modified some sentences so as to not imply that the Reedy structure always exists.
%I have still revised the writing (Benoit).}

In the paper, we use that the structure of a right $\Gamma$-module is the same as the structure of an infinitesimal $\Com$-bimodule,
where $\Com$ is the set-theoretic operad of unital commutative algebras (the commutative operad),
given by the one-point set $\Com(r) = *$
in every arity $r\geq 0$ (see \cite[Lemma~4.3]{Turchin1}).
Thus, with our notation, we have a category identity $\TopCat^{\Gamma^{op}} = \IBimod_{\Com}$.
The Reedy model structure of the category of infinitesimal bimodules $\IBimod_{\Com}$, which we consider in this paragraph,
is also identified with the Reedy model structure of the category of right $\Gamma$-modules in topological spaces
which we considered in the previous paragraphs.

\subsection{Fulton--MacPherson compactified configuration space}\label{ss:FM}
The Fulton--MacPherson operad, denoted by $\FM_n$, is a classical model of $E_n$-operad.
This operad was introduced by Getzler-Jones \cite{GetzJones}, by using a real differential-geometric variant of the Fulton--MacPherson compactifications of the configuration spaces of points
in $\R^n$.\footnote{This type of compactification is sometimes called Fulton--MacPherson--Axelrod--Singer compactification.}
Recall that the points of $\FM_n(S)$, where $S$ is any finite sets, are represented by rooted trees with leaves indexed by $S$,
and where each internal vertex with $k$ children is labelled by a configuration of $k$ points
in $\R^n$
modulo scaling and translation.
The internal vertices are required to have at least $2$ children, as in the following picture, where we take $S = \{1,\dots,6\}$
as index set for the leaves of our tree:
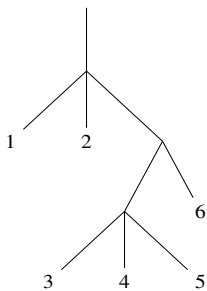
\[
\resizebox{3cm}{4cm}
{
\begin{tikzpicture}
\draw node {}
child{
child { node {1} }
child {node {2} }
child {
child {
child {node {3} }
child {node {4} }
child {node {5} }
 }
child {node {6} }
 }
};
\end{tikzpicture}
}
\]

We consider a similarly defined infinitesimal $\FM_n$-bimodule $\IF_n$ \cite[Section~6]{Turchin4}\footnote{In \cite{Turchin4}, this infinitesimal bimodule is denoted by $C_*[\bullet,S^n]$.}.
Points in $\IF_n(S)$ are rooted trees with leaves indexed by $S$, one distinguished vertex, the pearl, and with the following decorations.
Non-pearl vertices with $k$ children are labelled by a configuration of $k$ points in $\R^n$ modulo scaling and translation.
The pearl (with say $k$ children) is labelled by a configuration of $k$ points in $\R^n$ (but not modulo scaling or translation).
The pearl is allowed to have any arity $\geq 0$, while all the other internal vertices have an arity $\geq 2$.
If the pearl has arity zero, then this means that all the points escape to infinity in $\R^n\cup\{\infty\} = S^n$.
The following picture gives the representation of such a pearled tree.
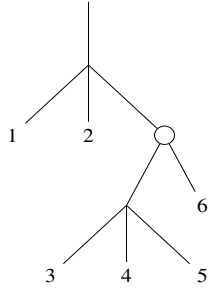
\[
\resizebox{3cm}{4cm}{
\begin{tikzpicture}
\draw node {}
child{
child { node {1} }
child {node {2} }
child { node[ext]{\text{ } \text{ }}
child {
child {node {3} }
child {node {4} }
child {node {5} }
 }
child {node {6} }
 }
};
\end{tikzpicture}
}
\]

Now let $M\subset\R^m$ be an open subset containing a neighborhood of~$\infty$.
Recall that we write $M_* := M\cup\{\infty\}$ for the space obtained by adding to $M$ the point at infinity in $S^m = \R^m\cup\{\infty\}$.
We consider the infinitesimal $\FM_m$-bimodule $\IF_M\subset\IF_m$
composed of those configurations
that lie in the pre-image of $M_*^{\times r}$ under the map $\IF_m(r)\to(S^m)^r$.
As a space, $\IF_M(k)$ is the Fulton--MacPherson--Axelrod--Singer local compactification of the configuration space $C_*(k,M_*)$ of $k+1$ points in $M_*$, one
of which is fixed to be~$\infty$.

We need the following result.

\begin{prop}\label{prop:EMReedyCof}
The infinitesimal $\FM_m$-bimodule $\IF_M$ is cofibrant in $\IBimod_{\FM_m}$ equipped with the Reedy model structure.
\end{prop}

\begin{proof}
Implicitly this result appeared in the proof of \cite[Theorem 6.5]{Turchin4}.
Recall that, by~\cite[Theorem~5.4]{DFT}, an infinitesimal bimodule over a reduced operad $\op P$ is cofibrant in the Reedy model structure
if and only it is cofibrant as an infinitesimal bimodule
over $\op P_{>0}$ in the projective model structure,
where $\op P_{>0}$ is the operad such that ${\op P}_{>0}(0) = \emptyset$ and ${\op P}_{>0}(k) = {\op P}(k)$, for $k>0$.
%, is cofibrant in $\IBimod_{(\FM_m)_{>0}}$ equipped with the projective model structure.
The fact that $\IF_M$ is cofibrant in $\IBimod_{(\FM_m)_{>0}}$  is analogous to \cite[Lemma~2.2]{Turchin4}\footnote{In \cite{Turchin4}, the second author considers
the projective model structure on $\IBimod_{\FM_m}$, in which for the cofibrancy of an object, the action of $\FM_m(0)$ should also be free.
Because of this the cofibrant object considered in \cite{Turchin4} and denoted by $\widetilde{C}_*[\bullet,M_*]$
is more complicated -- it is obtained from $C_*[\bullet,M_*]=:\IF_M$ by adding non-labelled hairs of length $\leq 1$.}.
\end{proof}

%\todo{Determine the dimension bound of the cells of the object $\IF_M$ and explain the applications to the proof of Theorem~\ref{thm:FM}.}

\subsection{Relative rational formality of the little disks operads: Recollections from \cite{FW,TW}}\label{ss:relative}
It is shown in \cite{FW} that the natural maps $\FM_m\to\FM_n$, induced by the inclusion $\R^m\to\R^n$, are rationally formal for $n-m\geq 2$.
Concretely, reformulating slightly \cite[Theorems C and D]{FW}, one has a commutative diagram of reduced operads
\[
\begin{tikzcd}
\FM_m\ar{r} & \FM_n^\Q \\
R_m \ar{r} \ar{u}{\sim} \ar{d} & R_n^\Q \ar{u}{\sim} \ar{d}{\sim} \\
 * \ar{r} & H_n
\end{tikzcd}
\]
relating the map $\FM_m\to\FM_n^\Q$ to one factorizing through the commutative operad $\Com = *$.
Here $R_m$ and $R_n$ are intermediate operads whose precise form is not relevant at this point, whereas $H_n = LG(H^*(\FM_n))$
is defined by taking the image of the cohomology of the topological operad $\FM_n$
under the derived functor of the Sullivan realization functor $LG$. (We then regard the object $H^*(\FM_n)$
as a dg Hopf cooperad equipped with a trivial differential.)

Recall that, for $n\geq 2$, the cohomology of the Fulton--MacPherson operad $H^*(\FM_n)$ is identified with the $n$-Poisson cooperad $\Poiss_n^c$,
the dual of the operad $\Poiss_n$ that governs the category of graded commutative algebras
equipped with a Poisson bracket of degree $n-1$.
Let $\Com^c$ denote the dual cooperad of the operad of commutative algebras in vector spaces.
We have a morphism of graded Hopf cooperads $H^*(\FM_n) = \Poiss_n^c\rightarrow\Com^c$,
dual to the morphism of graded operads $\Com\rightarrow\Poiss_n$,
which reflects the obvious restriction of structure functor from the category of graded Poisson algebras governed by the operad $\Poiss_n$
to the category of graded commutative algebras.
We have $LG(\Com^c) = G(\Com^c) = *$ and the morphism $*\rightarrow H_n$ in the above diagram is identified with the image of the morphism
of graded Hopf cooperads $H^*(\FM_n)\rightarrow\Com^c$
under the derived functor of the Sullivan realization functor $LG$.

Let us note that by generalities on model categories we may assume that the upper left vertical arrow of our diagram is a fibration\footnote{Let $A\leftarrow B\to C$ be a zigag in a model category. We may factorize the second map into $B\xhookrightarrow{\sim} B' \twoheadrightarrow C$ and then lift to obtain a zigzag $A \leftarrow B' \twoheadrightarrow C$.
In our case we apply this construction to the arrow category equipped with the projective model structure.}.
Then we may use that $\FM_m$ is Reedy cofibrant and invert the arrow. Furthermore, note that, trivially, all objects on the right act on themselves considered as operadic infinitesimal bimodules.
For later reference we shall hence record the following diagram of operads (three left-hand columns) and infinitesimal bimodules (right-hand column).
\beq{equ:formality diag}
\begin{tikzcd}
\FM_m \ar{r}{=} \ar{dr} & \FM_m \ar{r} & \FM_n^\Q & \FM_n^\Q \ar[loop left]{}\\
& R_m \ar{r} \ar{u}{\sim} \ar{d} & R_n^\Q \ar{u}{\sim} \ar{d}{\sim} & R_n^\Q \ar{u}{\sim} \ar{d}{\sim}\ar[loop left]{}\\
& * \ar{r} & H_n & H_n\ar[loop left]{}
\end{tikzcd}
\eeq
We will make use of the restriction functors associated to the operad morphism of the right-hand side for our infinitesimal bimodule structures.
We can for instance regard the object $H_n$ as an infinitesimal $\Com$-bimodule
by restriction through the operad morphism $\Com = *\rightarrow H_n$.
Recall that this  implies that $H_n$ inherits a natural right $\Gamma$-module structure
in the category of topological spaces.
We go back to this subject in Section~\ref{sec:graph-Gamma-module}.

\subsection{Graph complexes and graph (co)operads}\label{sec:graphs}
We shall use the graph cooperads $\Graphs_n$ defined by Kontsevich \cite{K2}.
Let us briefly recall the definition.
An admissible graph with $r$ external and $k$ internal vertices is an undirected graph such that the following holds.
\begin{itemize}
\item The external vertices are numbered $1,\dots,r$.
\item There is at least one external vertex in any connected component.
\item All internal vertices have at least valence 3.
\end{itemize}
Tadpoles and multiple edges are allowed. Here is an example of an admissible graph
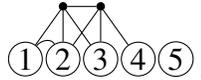
\[
 \begin{tikzpicture}
  \node[ext] (v1) at (0,0) {1};
  \node[ext] (v2) at (0.5,0) {2};
  \node[ext] (v3) at (1,0) {3};
  \node[ext] (v4) at (1.5,0) {4};
  \node[ext] (v5) at (2,0) {5};
  \node[int] (w1) at (0.5,.7) {};
  \node[int] (w2) at (1.0,.7) {};
  \draw (v1) to [out=50,in=120] (v2) (v1) edge (w1)  (v2) edge (w1) edge (w2) (v3) edge (w1) edge (w2) (v4) edge (w2) (w1) edge (w2);
 \end{tikzpicture}\, .
\]
A graph with $\#E$ edges and $\#V$ internal vertices is assigned degree $(n-1)\#E-n\#V$.   An $n$-orientation on an admissible graph is the following data:
\begin{itemize}
\item For $n$ even it is an ordering of the set of edges up to even permutations.
\item For $n$ odd it is an ordering of the set of half-edges and vertices up to even permutations.
\end{itemize}
We call an admissible graph with orientation data an oriented graph.

We now define the space $\Graphs_n(r)$ as the $\Q$-linear combinations of isomorphism classes of ($n$-)oriented admissible graphs with $r$ external vertices, together with an orientation. We identify an oriented graph with minus the same graph with opposite orientation.

The spaces $\Graphs_n(r)$ assemble into a dg Hopf cooperad.
The differential is obtained by edge contraction.
  \begin{align*}
   \delta
   \begin{tikzpicture}[baseline=-.65ex]
    \node[ext] (v) at (0,0) {$i$};
    \node[int](w) at (0,.5) {};
    \draw (v) edge +(-.5,.5) edge +(.5,.5) edge (w) (w) edge +(-.2,.5) edge +(.2,.5);
   \end{tikzpicture}
&=
   \begin{tikzpicture}[baseline=-.65ex]
    \node[ext] (v) {$i$};
    \draw (v) edge +(-.5,.5) edge +(-.2,.5) edge +(.2,.5) edge +(.5,.5);
   \end{tikzpicture}
   &
   \delta
   \begin{tikzpicture}[baseline=-.65ex]
    \node[int] (v) at (0,0) {};
    \node[int](w) at (0,.3) {};
    \draw (v) edge +(-.5,.5) edge +(.5,.5) edge (w) (w) edge +(-.2,.5) edge +(.2,.5);
   \end{tikzpicture}
&=
   \begin{tikzpicture}[baseline=-.65ex]
    \node[int] (v) {};
    \draw (v) edge +(-.5,.5) edge +(-.2,.5) edge +(.2,.5) edge +(.5,.5);
   \end{tikzpicture}
  \end{align*}
The dg commutative algebra structure is given by gluing graphs along external vertices.
  \[
      \begin{tikzpicture}[baseline=-.65ex]
  \node[ext] (v1) at (0,0) {1};
  \node[ext] (v2) at (0.5,0) {2};
  \node[ext] (v3) at (1,0) {3};
  \node[ext] (v4) at (1.5,0) {4};
  \node[int] (w1) at (0.5,.7) {};
  \draw (v1) to [out=50,in=120] (v2) (v1) edge (w1) (v2) edge (w1) (v3) edge (w1) ;
 \end{tikzpicture}
 \wedge
     \begin{tikzpicture}[baseline=-.65ex]
  \node[ext] (v1) at (0,0) {1};
  \node[ext] (v2) at (0.5,0) {2};
  \node[ext] (v3) at (1,0) {3};
  \node[ext] (v4) at (1.5,0) {4};
  \node[int] (w2) at (1.0,.7) {};
  \draw   (v2)  edge (w2) (v3)  edge (w2) (v4) edge (w2);
 \end{tikzpicture}
 =
     \begin{tikzpicture}[baseline=-.65ex]
  \node[ext] (v1) at (0,0) {1};
  \node[ext] (v2) at (0.5,0) {2};
  \node[ext] (v3) at (1,0) {3};
  \node[ext] (v4) at (1.5,0) {4};
  \node[int] (w1) at (0.5,.7) {};
  \node[int] (w2) at (1.0,.7) {};
  \draw (v1) to [out=50,in=120] (v2) (v1)  edge (w1) (v2) edge (w1) edge (w2) (v3) edge (w1) edge (w2) (v4) edge (w2);
 \end{tikzpicture}
  \]
Finally the cooperadic cocompositions are defined by subgraph contraction.
\[
    \begin{tikzpicture}[baseline=-.65ex]
  \node[ext] (v1) at (0,0) {1};
  \node[ext] (v2) at (0.5,0) {2};
  \node[ext] (v3) at (1,0) {3};
  \node[ext] (v4) at (1.5,0) {4};
  \node[ext] (v5) at (2,0) {5};
  \node[int] (w1) at (0.5,.7) {};
  \node[int] (w2) at (1.5,.7) {};
  \draw (v1) to [out=50,in=120] (v2) (v1) edge (w1) (v2) edge (w1)  (v3) edge (w1) edge (w2) (v4) edge (w2) (v5) edge (w2);
 \end{tikzpicture}
 \mapsto
     \begin{tikzpicture}[baseline=-.65ex]
  \node[ext] (v1) at (0,0) {1};
  \node[ext] (v2) at (0.5,0) {2};
  \node[ext] (v3) at (1,0) {*};
  \node[int] (w1) at (0.5,.7) {};
  \draw (v1) to [out=50,in=120] (v2) (v1) edge (w1) (v2) edge (w1)  (v3) edge (w1);
 \end{tikzpicture}
 \otimes
     \begin{tikzpicture}[baseline=-.65ex]
  \node[ext] (v3) at (1,0) {3};
  \node[ext] (v4) at (1.5,0) {4};
  \node[ext] (v5) at (2,0) {5};
  \node[int] (w2) at (1.5,.7) {};
  \draw (v3) edge (w2) (v4) edge (w2) (v5) edge (w2);
 \end{tikzpicture}
+
 \begin{tikzpicture}[baseline=-.65ex]
  \node[ext] (v1) at (0,0) {1};
  \node[ext] (v2) at (0.5,0) {2};
  \node[ext] (v3) at (1,0) {*};
  \node[int] (w1) at (0.5,.7) {};
  \node[int] (w2) at (1,.7) {};
  \draw (v1) to [out=50,in=120] (v2) (v1) edge (w1) (v2) edge (w1)  (v3) edge (w1) (v3) edge (w2) (v3) to [out=110,in=210] (w2) (v3) to [out=70,in=330] (w2);
 \end{tikzpicture}
 \otimes
     \begin{tikzpicture}[baseline=-.65ex]
  \node[ext] (v3) at (1,0) {3};
  \node[ext] (v4) at (1.5,0) {4};
  \node[ext] (v5) at (2,0) {5};
 % \node[int] (w2) at (1.5,.7) {};
 % \draw (v3) edge (w2) (v4) edge (w2) (v5) edge (w2);
 \end{tikzpicture}
  \]%\todo{The cocomposition must have one more summand.}

Note that each dg commutative algebra $\Graphs_n(r)$ is quasi-free, generated by the internally connected graphs $\ICG_n(r)\subset\Graphs(r)$,
so that, as graded commutative algebra $\Graphs_n(r) = S(\ICG_n(r))$.
There is a natural map of dg Hopf cooperads
\[
\Graphs_n\to H^*(\FM_n)
\]
given on the algebra generators $\ICG_n$ by sending any graph with internal vertices to zero, and by sending the edge between vertices $i$ and $j$ to the generator $\omega_{ij}\in H^*(\FM_n)$.
We then have the following important statement.

\begin{thm}[Kontsevich \cite{K2}, Lambrechts-Voli\'c \cite{LVformal}]
The map of dg Hopf cooperads $\Graphs_n\to H^*(\FM_n)$ is a quasi-isomorphism for every $n\geq 2$.
\end{thm}

\subsection{The left $\Gamma$-module structure of the graph cooperad}\label{sec:graph-Gamma-module}
We recalled in Section~\ref{ss:relative} that the graded Hopf cooperad $H^*(\FM_n)$ is equipped with a morphism $H^*(\FM_n)\rightarrow\Com^c$,
where $\Com^c$ is the dual cooperad of the operad that governs the category of commutative algebras in vector spaces.
We explicitly have $\Com^c(r) = \Q$.
We have an immediate description of the restriction of this morphism to the dg Hopf cooperad of graphs
as the morphism $\Graphs_n\rightarrow\Com^c$
such that $\gamma\mapsto 0$ if the graph $\gamma$ has at least one edge.

We mention in Section~\ref{ss:relative} that the operad $H_n = LG(H^*(\FM_n))$, which we define by taking the image of the cohomology cooperad $H^*(\FM_n)$
under the left derived functor of the Sullivan realization functor $LG$,
inherits the structure of an infinitesimal bimodule over the commutative operad,
and as a consequence, forms a right $\Gamma$-module
in the category of topological spaces.
We can dually see that the cohomology cooperad $H^*(\FM_n)$ inherits the structure of an infinitesimal bicomodule
over the commutative cooperad. We also have an isomorphism between the category of infinitesimal $\Com^c$-bicomodules
in a category and the category of left $\Gamma$-modules (just as we observed
that the category of infinitesimal $\Com$-bimodules in topological spaces
is isomorphic to the category of right $\Gamma$-modules). Hence, providing $H^*(\FM_n)$ with such an infinitesimal bicomodule structure
amounts to providing this object with a left Hopf $\Gamma$-module structure.
The right $\Gamma$-module structure of the object $H_n = LG(H^*(\FM_n))$
actually corresponds to this left Hopf $\Gamma$-module structure
of the cohomology $H^*(\FM_n)$
when we take the objectwise realization functor of~\eqref{equ:GOmega2}
for $\Gamma$-modules.

By restriction, we get that the cooperad of graphs $\Graphs_n$ inherits an infinitesimal $\Com^c$-bicomodule structure
and hence, a left Hopf $\Gamma$-module structure like the cohomology cooperad $H^*(\FM_n)$.
We then consider the collection $\Graphs_n(S)$ associated to arbitrary finite sets $S$,
which we define by an obvious generalization of the definitions
of Section~\ref{sec:graphs}.
We take this object $\Graphs_n(S)$ to define the image of a pointed set $S_* = \{*\}\sqcup S$ under our functor on the category $\Gamma$,
while we determine the morphism $f_*: \Graphs_n(S)\rightarrow\Graphs_n(T)$ associated to a map $f: S_*\rightarrow T_*$
of pointed sets $S_* = \{*\}\sqcup S$ and $T_* = \{*\}\sqcup T$
as follows.
Let $\gamma\in\Graphs(S)$.
We consider the external vertices $v_s$ indexed by the elements such that $s\in f^{-1}(*)$.
We take $f_*\gamma = 0$ if some of these external vertices $v_s$
have an incident edge in $\gamma$.
We discard these external vertices otherwise and we merge the external vertices $v_s$ associated to the elements $s\in f^{-1}(t)$
in the fiber of an index $t\in T$ to obtain the graph $f_*\gamma\in\Graphs_n(T)$
(with the merged external vertex indexed by the corresponding element $t\in T$).

The collection $\ICG_n$ is preserved by the left $\Gamma$-module structure of the cooperad $\Graphs_n$
and hence, inherits the structure of a left $\Gamma$-module
in the category of graded vector spaces.
The cross effect of this collections $\cro\ICG_n$
is identified with the collection of internally connected graphs
with no isolated external vertices.
We easily see that this left $\Omega$-module $\cro\ICG_n$ is free in the sense that every morphism of left $\Omega$-modules on $\cro\ICG_n$
is uniquely determined by a morphism of left $\Sigma$-modules on a left $\Sigma$-module of generators $\ICGp_n\subset\cro\ICG_n$,
which consists of graphs all of whose external vertices have valence exactly one.
We will call such internally connected graphs primitive. For example:
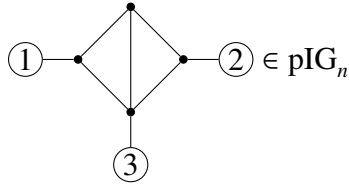
\[
\begin{tikzpicture}[scale=.7,baseline=-.65ex]
\node[int] (v1) at (-1,0){};
\node[int] (v2) at (0,1){};
\node[int] (v3) at (1,0){};
\node[int] (v4) at (0,-1){};
\node[ext] (w1) at (-2,0) {$1$};
\node[ext] (w2) at (2,0) {$2$};
\node[ext] (w3) at (0,-2) {$3$};
\draw (v1)  edge (v2) edge (v4) edge (w1) (v2) edge (v4) (v3) edge (v2) edge (v4) (v4) edge (w3) (v3) edge (w2);
\end{tikzpicture}
\in
\ICGp_n
\]
is a primitive graph.

\subsection{Induction and restriction of operadic bimodules}\label{sec:indres}

\begin{prop}\label{prop:resind}
Let $\op P\to \op Q$ be a morphism of well-pointed reduced operads (see Section \ref{sec:htpy theory ibimod}). Then the induction and restriction functors
\[
\Ind_{\op P}^{\op Q}: \IBimod_{\op P} \rightleftarrows \IBimod_{\op Q} :\Res^{\op Q}_{\op P}
\]
form a Quillen adjunction, where we equip the categories of infinitesimal bimodules with the Reedy model structure.
\end{prop}

\begin{proof}
This statement is the counterpart for infinitesimal bimodules and the Reedy model structure of the claim of \cite[Theorem 15.B]{Frbook2} for right modules
and the projective model structure.
The argument of this reference applies without change in our setting.
Namely, we just use that fibrations and weak equivalences of infinitesimal bimodules over operads are created in the category of right $\La$-modules
and that $\Res^{\op Q}_{\op P}$ is the identity at the $\La$-module level
to conclude that the restriction functor $\Res^{\op Q}_{\op P}$ preserves fibrations
and weak equivalences. The result follows.
\end{proof}

We denote the derived functor associated to $\Ind_{\op P}^{\op Q}$ by $L\Ind_{\op P}^{\op Q}$.

Recall that $M_* = M\cup\{\infty\}$ denotes the pointed $m$-manifold defined by adding a point at infinity to our open subset $M\subset\R^m\subset \R^m\cup\{\infty\}=S^m$.
Below, by $M_*^{\times\bullet}$, we denote the right $\Gamma$-module (equivalently the infinitesimal $\Com$-bimodule) that assigns the space of pointed maps $S_*\to M_*$
to any finite pointed set $S_*$. %We have the following statement.

\begin{prop}\label{prop:indEM}
We have $L\Ind_{\FM_m}^{\Com}\IF_M\simeq M_*^{\times\bullet}$.
\end{prop}

\begin{proof}
The object $\IF_M$ is already Reedy cofibrant by Proposition \ref{prop:EMReedyCof}.
Hence $L\Ind_{\FM_m}^{\Com} \IF_M \simeq \Ind_{\FM_m}^{\Com} \IF_M$.
The latter object is obtained by contracting all boundary components to points, and hence we obtain just $M_*^{\times\bullet}$.
%\todo{I realized there is a little problem here. For the Reedy model structure of the category of IBimodules in our paper~\cite{DFT}, we required $\Sigma$-cofibrancy of the reduced operad over which bimodules are considered. This was necessary for our explicit construction of the Reedy fibrant replacement.}
\end{proof}

\section{Hopf $\Gamma$-modules and hairy graph complexes}\label{s:hopf}

\subsection{Model structure}\label{ss:hopf_model_str}
Recall that we use the expression `left Hopf $\Gamma$-module' for the category of left $\Gamma$-modules in dg commutative algebras.
In what follows, we similarly call `left Hopf $\Sigma$-modules' the objects of the category of left $\Sigma$-modules
in dg commutative algebras.
To avoid clumsy notation, we will abbreviate the notation of the morphism sets $\Mor_{\DGCA^{\Sigma}}(\dots)$
of the category of left Hopf $\Sigma$-modules $\DGCA^{\Sigma}$
by $\Mor_{H\Sigma}(\dots)$
and we will similarly abbreviate the notation of the morphism sets $\Mor_{\DGCA^{\Gamma}}(\dots)$ of the category of left Hopf $\Gamma$-modules $\DGCA^{\Gamma}$
by $\Mor_{H\Gamma}(\dots)$. We will also use the short notation $\Mor_{\Sigma}$, $\Mor_{\Gamma}$, $\Mor_{\Omega}$ for the morphisms sets of the categories
of left $\Sigma$-modules, of left $\Gamma$-modules and of left $\Omega$-modules
in dg vector spaces.

We endow the category $\HopfColl$ with a model structure by declaring the fibrations and weak equivalences to be objectwise fibrations and weak equivalences, respectively.
%The left Hopf $\Gamma$-modules are equivalent to Hopf infinitesimal $\Com^c$-bicomodules.
We endow the category $\HopfGM$ with the projective model structure.
Note that, in particular, every object in $\HopfGM$ is fibrant since every object in $\HopfColl$ is.

%We have forgetful functors
%\[
%\HopfGM \to \Mod_\Gamma \to \Mod_{\Sigma}.
%\]
%The free Hopf right $\Gamma$-module generated by an $\Sigma$-module $\op X$ is the image under the left adjoint of the forgetful functor $F_{H\Gamma}$. Concretely, we have that
%\[
%F_{H\Gamma}(X) = F_H(F_\Gamma(X)),
%\]
%where $F_\Gamma : \Mod_\Gamma \to \HopfGM$ and $F_\Gamma: \Mod_{\Sigma}\to \Mod_\Gamma$ are the left adjoints to the respective forgetful functors. Note in particular that $F_\Gamma(\cdot)$ just takes the free symmetric algebra of its argument objectwise.

We will need the following statement.

\begin{lemma}\label{lem:simplframe}
For any object $\op M\in\HopfGM$ the simplicial left Hopf $\Gamma$-module $\op M^{\Delta^\bullet} := \op M \otimes \Omega^*(\Delta^\bullet)$ is a simplicial frame for $\op M$.
\end{lemma}

\begin{proof}
This result follows from the classical properties of the Sullivan dg algebra $\Omega^*(\Delta^\bullet)$,
namely that we have $H^*(\Omega^*(\Delta^k)) = \Q$, for any simplicial dimension $k\geq 0$,
and that the morphism $\Omega^*(\Delta^\bullet)\rightarrow\Omega^*(\partial\Delta^\bullet)$
is surjective degree-wise (see for instance~\cite[Theorem II.7.1.5]{FrII} for an account of the applications of the Sullivan dg algebra
for the definition of simplicial frames in the category of dg commutative algebras).
The first acyclicity claim implies that the morphism $\op M = \op M^{\Delta^0}\to\op M^{\Delta^k}$
is a quasi-isomorphism,
while the latter claim implies that the matching morphism $\op M^{\Delta^k}\to\op M^{\p\Delta^k}$
is a fibration,
because the $k$th matching object $M_k(\op M^{\Delta^\bullet})$ of the simplicial left Hopf $\Gamma$-module $\op M^{\Delta^\bullet}$
is identified with the tensor product $M^{\p\Delta^k} = \op M \otimes\Omega^*(\p\Delta^\bullet)$ (see \emph{loc. cit.}
for the proof of an analogous identity in the category of dg commutative algebras).
\end{proof}

\subsection{Hairy graph complex }\label{sec:HGC}

\begin{lemma}\label{lem:preHGC}
Suppose that $\op M$ is a left Hopf $\Gamma$-module.
Then there is a bijection
\[
\Mor_{H\Gamma}(\Graphs_n^o, \op M^o)\cong\Mor_{\Sigma}(\ICGp_n^o,\cro\op M^o),
\]
where the superscript $(-)^o$ means that we consider objects equipped with a zero differential, $\cro$ is the cross-effect functor (see Section \ref{sec:gammaomega})
and $\ICGp_n$ is the $\Sigma$-module of the primitive internally connected graphs as in Section \ref{sec:graphs}.
\end{lemma}

\begin{proof}
Since $\Graphs_n^o$ is free as a left Hopf $\Sigma$-module, generated by the internally connected graphs $\ICG_n$, we have that
\beq{equ:li1}
\Mor_{H\Gamma}(\Graphs_n^o, \op M^o)
=
\Mor_{\Gamma}(\ICG_n^o, \op M^o).
\eeq
By Proposition \ref{prop:crosseffect} the cross effect functor induces an equivalence of categories between left $\Gamma$-modules and left $\Omega$-modules in any abelian category.
We hence find
\beq{equ:li2}
\Mor_{\Gamma}(\ICG_n^o, \op M^o)= \Mor_{\Omega}(\cro\ICG_n^o, \cro\op M^o).
\eeq
Now $\cro\ICG_n^o$ is a free left $\Omega$-module, generated by the $\Sigma$-module $\ICGp_n$, so that
\beq{equ:li3}
\Mor_{\Omega}(\cro\ICG_n^o, \cro\op M^o) = \Mor_{\Sigma}(\ICGp_n^o, \cro\op M^o).
\eeq
\end{proof}

We introduce the notation
\[
\HGC_{\op M,n} := \iHom_{\Sigma}(\ICGp_n^o, \cro\op M^o)
\]
and call $\HGC_{\op M,n}$ the space of $\op M$-decorated hairy graphs. Here $\iHom_{\Sigma}(-,-)$
denotes the graded vector space of homogeneous homomorphisms
of left $\Sigma$-modules.
%In fact we have the following.

\begin{prop}\label{prop:HGCLinfty}
The space $\HGC_{\op M,n}$ is equipped with an $L_\infty$-structure $\{\ell_k\}_{k\geq 1}$ (functorially in $\op M$),
with differential $\delta = \ell_1$ and higher $L_\infty$-operations $\ell_k$, $k\geq 2$,
together with a descending complete filtration $\HGC_{\op M,n}= \mF^1\HGC_{\op M,n}\supset \mF^2\HGC_{\op M,n}\supset \cdots$
that is compatible with the $L_\infty$-structure in the sense that
\begin{equation}\label{equ:Linfty filtration compat}
    \ell_k(\mF^{p_1}\HGC_{\op M,n},\dots, \mF^{p_k}\HGC_{\op M,n})\subset \mF^{p_1+\dots+p_k}\HGC_{\op M,n}.
\end{equation}
The Maurer--Cartan elements of $\HGC_{\op M,n}$ are in one-to-one correspondence with the morphisms of left Hopf $\Gamma$-modules $\Graphs_n\to\op M$.
\end{prop}

We recall the explicit definition of a Maurer--Cartan element in a complete $L_\infty$-algebra later on (in Section~\ref{ss:MC}),
when we tackle examples of applications of our result.
For the moment, we only need a conceptual definition of $L_\infty$-structures and of Maurer--Cartan elements
in terms of zeros of vector fields on graded affine schemes that we shall quickly recall, see for instance \cite[Section 4]{FW2}.
Consider an $L_\infty$-algebra $\fg$ with differential $\delta = \ell_1$ and higher $L_\infty$-operations $\ell_k$, $k\geq 2$,
that is equipped with a complete descending filtration
\[
\fg = \mF^1\fg \supset \mF^2\fg \supset \cdots
\]
compatible with the $L_\infty$-operations as in \eqref{equ:Linfty filtration compat}.
Then for any graded commutative algebra $\ggR$ we may consider the completed tensor product $\ggR\hotimes \fg$,
which is again a filtered complete $L_\infty$-algebra by $\ggR$-linear extension of the operations.
We may then consider the functions
\begin{equation}\label{equ:UR from Linfty}
    \begin{aligned}
        \mU^{\ggR} : (\ggR\hotimes \fg)_0 \to (\ggR\hotimes \fg)_{-1} \\
        x\mapsto \sum_{k\geq 1} \frac 1 {k!} \ell_k(x,\dots,x).
        \end{aligned}
\end{equation}
We can recover the operations $\ell_k$, $k\geq 1$, by graded polarization from these functions, for a varying $\ggR$ (see \cite[Section 4]{FW2}).
We may conversely define a sequence of operations $\ell_k$, $k\geq 1$, by providing the function $\mU^{\ggR}$,
for each graded commutative algebra $\ggR$,
as long as we can verify that this function has a power series expansion as above,
with terms $\ell_k$ being $\ggR$-linear extensions of multilinear functions defined over $\Q$.
We moreover get that the structure relations of $L_\infty$-algebras for the operations $\ell_k$
are then equivalent to the
relation
\begin{equation}\label{equ:UMC}
    \mU^{\ggR[\epsilon]}(x+\epsilon\mU^{\ggR}(x)) = \mU^{\ggR[\epsilon]}(x),
\end{equation}
for the power series $\mU^{\ggR}$, for any graded commutative algebra $\ggR$, any element $x\in(\ggR\hotimes\fg)_0$,
and where $\epsilon$ is a formal variable of degree $+1$.
Let us note that with, this approach, we define the differential of our $L_{\infty}$-algebra as the linear term $\delta = \ell_1$
of our power series $\mU^{\ggR}$, and the above relation \eqref{equ:UMC} integrates the relation of differential $\delta^2 = 0$
when we focus on the linear term.
This representation of the structure of $L_\infty$-algebras has the advantage of avoiding signs.
Geometrically, and as it is used in previous literature, one interprets $\fg$ as a pro-algebraic graded variety with functor of points $\ggR \to (\ggR\hotimes \fg)_0$.
Then the power series $\mU^{\ggR}$ encodes a vector field $Q$ on this graded variety of degree -1, and the $L_\infty$-relations \eqref{equ:UMC} state that $Q^2=0$, i.e., that $Q$ is a homological vector field. That said, we will never use this latter algebro-geometric interpretation, and the reader may safely ignore it.

\begin{proof}[Proof of Proposition \ref{prop:HGCLinfty}]
We equip $\HGC_{\op M,n} = \iHom_{\Sigma}(\ICGp_n^o, \cro\op M^o)$ with the descending complete filtration inherited from the filtration on $\ICGp_n$ by the number of edges in graphs.
That is, $\mF^p\HGC_{\op M,n}$ consists of those homomorphisms that vanish on all graphs that have fewer than $p$ edges.
Now consider a graded algebra $\ggR$ and note that
$$
\ggR\hotimes \HGC_{\op M,n} \cong
\iHom_{\Sigma/\ggR}(\ggR\otimes \ICGp_n^o,\ggR\otimes \cro\op M^o),
$$
where we take $\ggR$-linear homomorphisms, i.e., we extend our ground ring to $\ggR$. For this isomorphism we also use that there are only finitely many graphs with a given number of edges.

Lemma \ref{lem:preHGC} naturally extends to $\ggR$-coefficients to give us an isomorphism
\[
    (\ggR\hotimes \HGC_{\op M,n})_0 \cong \Mor_{\Sigma/\ggR}(\ggR\otimes \ICGp_n^o,\ggR\otimes \cro\op M^o) \xrightarrow{\Phi}
    \Mor_{H\Gamma/\ggR}(\ggR\otimes\Graphs_n^o, \ggR\otimes\op M^o).
\]
Recall also that the inverse map is given by restriction to generators and projection to cogenerators. Explicitly, the inverse sends a morphism $F$ on the right-hand side to $\pi\circ F\circ \iota$, where $\iota:\ICGp_n\to \Graphs_n$ is the natural inclusion and $\pi:\op M\to \cro \op M$ is the projection, hiding $\ggR$-linear extension from the notation.

We then define the function $\mU^{\ggR}: (\ggR\hotimes \HGC_{\op M,n})_0 \to (\ggR\hotimes \HGC_{\op M,n})_{-1}$ to be
\begin{equation}\label{equ:URdef}
    \mU^{\ggR}(x) := \pi \circ (d_{\op M} \circ \Phi(x)-\Phi(x)\circ d_{\Graphs_n}) \circ \iota.
\end{equation}
For this proof we shall also abbreviate the commutator with the differentials appearing in this formula to
\[
[d,(-)] :=   d_{\op M} \circ (-) - (-) \circ d_{\Graphs_n}.
\]
Let us first verify the $L_\infty$-relations in the form \eqref{equ:UMC}.
To this end we first note that for any morphism $F\in \Mor_{H\Gamma/\ggR}(\ggR\otimes\Graphs_n^o, \ggR\otimes\op M^o)$ the combination
\[
F + \epsilon [d,F]
:
\ggR[\epsilon]\otimes\Graphs_n^o \to \ggR[\epsilon]\otimes\op M^o
\]
is a morphism of Hopf $\Gamma$-modules over $\ggR[\epsilon]$, where $\epsilon$ is again a formal variable of degree $+1$.
Note also that for any morphism $\tilde F:\ggR[\epsilon]\otimes\Graphs_n^o \to \ggR[\epsilon]\otimes\op M^o$ of Hopf $\Gamma$-modules we have
\[
\Phi(\pi \circ \tilde F\circ \iota)
=
\tilde F
\]
and hence
\[
\mU^{\ggR[\epsilon]}(\pi \circ \tilde F\circ \iota)
=
\pi \circ [d,\tilde F] \circ \iota.
\]
We use this to verify \eqref{equ:UMC} as follows:
\begin{align*}
    \mU^{\ggR[\epsilon]}(x+\epsilon\mU^{\ggR}(x))
    &=
    \mU^{\ggR[\epsilon]}(\pi \circ \Phi(x) \circ \iota +\epsilon \pi \circ [d,\Phi(x)]\circ \iota)
    \\
    &=
    \mU^{\ggR[\epsilon]}\left(\pi \circ \left(
         \Phi(x)+ \epsilon [d,\Phi(x)]
         \right) \circ \iota \right)
    \\
    &=
    \pi \circ \left(
        [d,\Phi(x)]  + \underbrace{[d,\epsilon [d, \Phi(x)]]}_{=0}
    \right) \circ \iota
    = \mU^{\ggR[\epsilon]}(x).
\end{align*}
For the final simplification we used that the differentials square to zero.

Furthermore, note that $\mU^{\ggR}(x)$ is obtained by composing copies of $x$ with $\ggR$-linear extensions of the structure maps (product, $\Gamma$ structure and differential) of $\op M$ and $\Graphs_n$. Hence it is clear that $\mU^{\ggR}(x)$ is indeed of the form \eqref{equ:UR from Linfty}, obtained using the $\ggR$-linear extensions of $k$-linear maps $\ell_k$ on $\HGC_{\op M,n}$, and hence we obtain indeed an $L_\infty$-structure on $\HGC_{\op M,n}$.
Furthermore, all structure maps used in the definition of $\ell_k$ either leave the number of edges of graphs in $\Graphs_n$ constant, or decrease the number of edges. Hence the compatibility \eqref{equ:Linfty filtration compat} follows.

Finally, by very definition the Maurer-Cartan elements in $\ggR\hotimes \HGC_{\op M,n}$ are those satisfying $\mU^{\ggR}(x)=0$, and hence precisely those elements that correspond under the isomorphism $\Phi$ to left Hopf $\Gamma$-module maps $\ggR\otimes \Graphs_n \to \ggR\otimes \op M$.

%It follows (see e.g., \cite[Section 3]{FTW} or \cite[Section 4]{FW2} for similar constructions) that the commutator with the differentials

\end{proof}

As a corollary we find the following result.

\begin{prop}\label{prop:graphscofibrant}
The left Hopf $\Gamma$-module $\Graphs_n$ (see Section \ref{sec:graphs}) is cofibrant in the projective model structure.
\end{prop}

\begin{proof}
Recall that $\Com^c$ denotes the dual cooperad of the operad of commutative algebras in vector spaces (see Section~\ref{sec:graph-Gamma-module}).
This cooperad $\Com^c$ is identified with the initial object of the category of left Hopf $\Gamma$-modules.
Let the following lifting problem be given.
\[
\begin{tikzcd}
\Com^c \ar{d} \ar{r} & \mA \ar[twoheadrightarrow]{d}{\sim} \\
\Graphs_n \ar{r} & \mB
\end{tikzcd}\,.
\]
By the preceding proposition this lifting problem is equivalent to the problem of lifting some Maurer--Cartan element $m\in \HGC_{\mB,n}$ to a Maurer--Cartan element $m'\in \HGC_{\mA,n}$ along the morphism of $L_\infty$-algebras
\[
\HGC_{\mA,n}\to \HGC_{\mB,n}.
\]
However, this morphism is compatible with the descending complete filtrations on both sides, and the associated graded morphism is a surjection and a quasi-isomorphism. Hence the lifting is unobstructed.
(The obstructions lie in the homology of the kernels of the surjective quasi-isomorphisms $\gr^p\HGC_{\mA,n}
\xrightarrow{\sim} \gr^p\HGC_{\mB,n}$.)
\end{proof}

\subsection{Explicit combinatorial formulas for the $L_\infty$-structure on the hairy graph complex}\label{sec:HGCLie}
The definition of the $L_\infty$-structure on $\HGC_{\op M, n}$ in the proof of Proposition \ref{prop:HGCLinfty} is concise and in particular takes care of all prefactors and signs.
However, it is not very explicit.
In this section we work out the explicit combinatorial form of the operations $\ell_k$, modulo signs and prefactors,
in the special case $\op M=R^{\otimes \bullet}$, where $R$ is an augmented dg commutative algebra.
The left Hopf $\Gamma$-module structure of this object $\op M=R^{\otimes \bullet}$
is determined by the commutative algebra structure of $R$
and we have $\cro\op M=\bar R^{\otimes\bullet}$.
For this $\op M$ we denote the graph complex alternatively by
\[
\HGC_{\bar R,n} := \HGC_{\op M,n}.
\]
It will be advantageous to identify
\[
\HGC_{\bar R,n}
=
\iHom_{\Sigma}(\ICGp_n, \bar R^{\bullet})
\cong
\prod_{r\geq 1}
\ICGp_n(r)^* \hotimes_{\Sigma_r} \bar R^{\otimes r}
=: \ICGp_n^* \hotimes_{\Sigma} \bar R^{\bullet},
\]
where we complete the tensor product with respect to the filtration by number of edges. We shall think of elements of the dual space $\ICGp_n^*$ of $\ICGp_n$ also as graphs. We can hence depict elements of $\HGC_{\bar R,n}$ as series of graphs with external legs which are decorated by elements of $\bar R$.
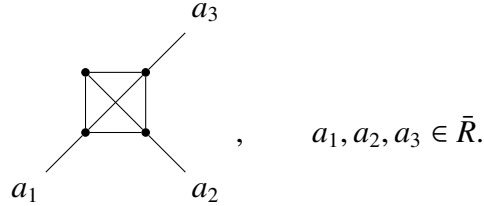
\[
\begin{tikzpicture}[scale=.8,baseline=-.65ex]
\node[int] (v1) at (0,0){};
\node[int] (v2) at (0,1){};
\node[int] (v3) at (1,0){};
\node[int] (v4) at (1,1){};
\node (w1) at (-1,-1) {$a_1$};
\node (w2) at (2,-1) {$a_2$};
\node (w3) at (2,2) {$a_3$};
\draw (v1)  edge (v2) edge (v3) edge (v4) edge (w1) (v2) edge (v4) (v3) edge (v2) edge (v4) (v4) edge (w3) (v3) edge (w2);
\end{tikzpicture}
\, , \quad\quad
a_1,a_2,a_3\in \bar R.
\]

Using this picture, let us describe combinatorially the $L_\infty$-structure on the graph complex $\HGC_{\bar R,n}$, by tracing the construction of the proof of Proposition~\ref{prop:HGCLinfty}.
Concretely, let us compute the $L_\infty$-operation
\[
\ell_k(\Gamma_1,\dots,\Gamma_k),
\]
for $\Gamma_1,\dots,\Gamma_k\in \HGC_{\bar R,n}$ graphs with $\bar R$-decorated legs.
To this end we have to compute the $\epsilon_1\cdots \epsilon_k$ coefficient in the series (see \eqref{equ:URdef})
\[
\mU^{\Q[\epsilon_1,\dots, \epsilon_k]}(\underbrace{\epsilon_1\Gamma_1 + \cdots + \epsilon_k \Gamma_k}_{=: \Gamma}),
\]
where $\epsilon_j$ is a variable of degree negative to the degree of $\Gamma_j$. Let us also write $\ggR:=\Q[\epsilon_1,\dots,\epsilon_n]$.
Recall also that the linear term of our series $\mU^{\ggR}$, which corresponds to the case $k=1$ of this computation,
collects the terms of the differential of our $L_{\infty}$-algebra.

We first have to extend our $\Gamma$ to a map
\begin{multline*}
\Gamma' \in \Mor_{\Omega/\ggR}(\ggR\otimes \cro\ICG_n^o, \ggR\otimes\cro\op M^o)\subset \\
 \Mor_{\Sigma/\ggR}(\ggR\otimes\cro\ICG_n^o, \ggR\otimes\cro\op M^o)\cong \ggR\hotimes \cro \ICG_n^* \hat \otimes_{\Sigma} \bar R^\bullet
\end{multline*}
according to \eqref{equ:li3}.
Note that elements of the right-hand side can again be considered as graphs, but now with multiple external vertices that are $\bar R$-decorated, and that can have valency $\geq 1$.
Tracing the construction one can then see that the desired element $\Gamma'$ is combinatorially obtained by summing over all ways of fusing arbitrary subsets of hairs to external vertices, multiplying the decorations, as indicated in the following picture.
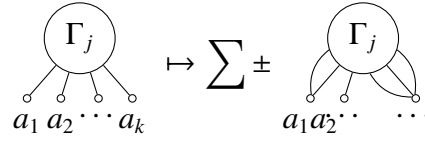
\[
\begin{tikzpicture}[baseline=-.8ex]
\node[draw,circle] (v) at (0,.3) {$\Gamma_j$};
\node[ext,label=-90:{$a_1$}] (w1) at (-.7,-.5) {};
\node[ext,label=-90:{$a_2$}] (w2) at (-.25,-.5) {};
\node[ext,label=-90:{$\dots$}] (w3) at (.25,-.5) {};
\node[ext,label=-90:{$a_k$}] (w4) at (.7,-.5) {};
\draw (v) edge (w1) edge (w2) edge (w3) edge (w4);
\end{tikzpicture}
\mapsto
\sum \pm
\begin{tikzpicture}[baseline=-.8ex]
\node[draw,circle] (v) at (0,.3) {$\Gamma_j$};
\node[ext,label=-90:{$a_1a_2$}] (w1) at (-.7,-.5) {};
\node[ext,label=-90:{$\dots$}] (w2) at (-.25,-.5) {};
%\node[ext,label=-90:{$\dots$}] (w3) at (.25,-.5) {};
\node[ext,label=-90:{$\dots$}] (w4) at (.7,-.5) {};
\draw (v) edge (w1) edge[bend right] (w1) edge (w2) edge[bend left] (w4) edge[bend right] (w4) edge (w4);
\end{tikzpicture}
\]
Next we have to take $\Gamma'$ and extend it further to a map
\[
\Gamma'' \in \Mor_{\Gamma/\ggR}(\ggR\otimes \ICG_n^o, \ggR\otimes \op M^o) \cong \ggR\hotimes \ICG_n^* \hat \otimes_{\Sigma} R^{\bullet}
\]
according to \eqref{equ:li2}.
Graphically, this just amounts to summing over all ways of inserting zero-valent external vertices in our graphs, decorated by the unit element $1\in R$.
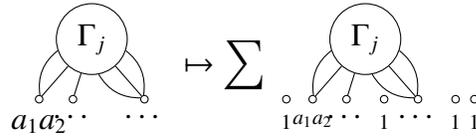
\[
    \begin{tikzpicture}[baseline=-.8ex]
        \node[draw,circle] (v) at (0,.3) {$\Gamma_j$};
        \node[ext,label=-90:{$a_1a_2$}] (w1) at (-.7,-.5) {};
        \node[ext,label=-90:{$\dots$}] (w2) at (-.25,-.5) {};
        %\node[ext,label=-90:{$\dots$}] (w3) at (.25,-.5) {};
        \node[ext,label=-90:{$\dots$}] (w4) at (.7,-.5) {};
        \draw (v) edge (w1) edge[bend right] (w1) edge (w2) edge[bend left] (w4) edge[bend right] (w4) edge (w4);
        \end{tikzpicture}
        \mapsto
\sum
\begin{tikzpicture}[baseline=-.8ex]
\node[draw,circle] (v) at (0,.3) {$\Gamma_j$};
\node[ext,label=-90:{$\scriptstyle 1$}] at (-1.05,-.5) {};
\node[ext,label=-90:{$\scriptstyle  a_1a_2$}] (w1) at (-.7,-.5) {};
\node[ext,label=-90:{$\dots$}] (w2) at (-.25,-.5) {};
%\node[ext,label=-90:{$\dots$}] (w3) at (.25,-.5) {};
\node[ext,label=-90:{$\scriptstyle 1$}] at (.25,-.5) {};
\node[ext,label=-90:{$\dots$}] (w4) at (.7,-.5) {};
\node[ext,label=-90:{$\scriptstyle 1$}] at (1.2,-.5) {};
\node[ext,label=-90:{$\scriptstyle 1$}] at (1.45,-.5) {};
\draw (v) edge (w1) edge[bend right] (w1) edge (w2) edge[bend left] (w4) edge[bend right] (w4) edge (w4);
\end{tikzpicture}
\]
Finally we have to work out the identification \eqref{equ:li1} to obtain an element
\begin{multline}\label{equ:FinMor}
    F \in \Mor_{H\Gamma/\ggR}(\ggR\otimes\Graphs_n^o,\ggR\otimes \op M^o) \subset \\
    \Mor_{\Sigma}(\ggR\otimes\Graphs_n^o, \ggR\otimes\op M^o) \subset \ggR\hotimes\Graphs_n^*\hat \otimes R^{\otimes\bullet}.
\end{multline}
Given an element $\gamma=\gamma_1\cdots \gamma_r\in \Graphs_n$ that decomposes into $r$ internally connected components, our desired map $F$ acts as $F(\gamma)=\Gamma''(\gamma_1)\cdots \Gamma''(\gamma_r)$.
Graphically this means that $F$, as an element of the very right-hand side of \eqref{equ:FinMor}, is obtained by fusing
several copies of $\Gamma''$ at the external vertices, thus taking formally an "exponential" of $\Gamma''$. Here is a schematic picture of the graphs produced.
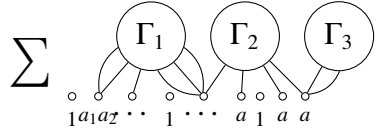
\[
\sum
\begin{tikzpicture}[baseline=-.8ex]
\node[draw,circle] (v) at (0,.3) {$\Gamma_1$};
\node[draw,circle] (v1) at (1.25,.3) {$\Gamma_2$};
\node[draw,circle] (v2) at (2.5,.3) {$\Gamma_3$};
\node[ext,label=-90:{$\scriptstyle 1$}] at (-1.05,-.5) {};
\node[ext,label=-90:{$\scriptstyle  a_1a_2$}] (w1) at (-.7,-.5) {};
\node[ext,label=-90:{$\dots$}] (w2) at (-.25,-.5) {};
%\node[ext,label=-90:{$\dots$}] (w3) at (.25,-.5) {};
\node[ext,label=-90:{$\scriptstyle 1$}] (w5) at (.25,-.5) {};
\node[ext,label=-90:{$\dots$}] (w4) at (.7,-.5) {};
\node[ext,label=-90:{$\scriptstyle a$}] (w6) at (1.2,-.5) {};
\node[ext,label=-90:{$\scriptstyle 1$}] (w7) at (1.45,-.5) {};
\node[ext,label=-90:{$\scriptstyle a$}] (w8) at (1.75,-.5) {};
\node[ext,label=-90:{$\scriptstyle a$}] (w9) at (2.05,-.5) {};
\draw (v) edge (w1) edge[bend right] (w1) edge (w2) edge[bend left] (w4) edge[bend right] (w4) edge (w4)
(v2) edge (w9) edge[bend left] (w9) (v1) edge (w4) edge (w6) edge (w9) edge (w8);
\end{tikzpicture}
\]
To the element $F$ we then need to apply the differential, given by the differentials on source and target. Afterwards we restrict the map again to the generators $\ICGp$ and project to $\bar R^{\otimes\bullet}$.
Note in particular that the restriction to $\ICGp$ is the same as the projection $\Graphs_n^*\to \ICGp^*$ that discards all graphs that are not internally connected or have external vertices of valency $\neq 1$.
Furthermore, recall that we only need to keep the coefficient of $\epsilon_1\cdots \epsilon_k$, or equivalently only those terms in which each $\Gamma_j$ appears exactly once.
Terms in $F$ that contribute non-trivially to the end result $\ell_k(\Gamma_1,\dots,\Gamma_k)$ thus can easily be seen to have either of two forms.
\begin{enumerate}
\item\label{description_L_infinity_operations_HGC}
All external vertices have valence exactly one, except for one, which has higher valence.
Then the differential on $\Graphs_n^*$ splits off all edges on the unique vertex to make it into a valence one vertex as well.
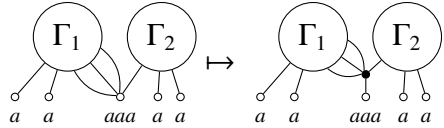
\[
\begin{tikzpicture}[baseline=-.8ex]
\node[draw,circle] (v) at (0,.3) {$\Gamma_1$};
\node[draw,circle] (v1) at (1.25,.3) {$\Gamma_2$};
%\node[ext,label=-90:{$\scriptstyle a$}] at (-1.05,-.5) {};
\node[ext,label=-90:{$\scriptstyle  a$}] (w1) at (-.7,-.5) {};
\node[ext,label=-90:{$\scriptstyle a$}] (w2) at (-.25,-.5) {};
%\node[ext,label=-90:{$\dots$}] (w3) at (.25,-.5) {};
%\node[ext,label=-90:{$\scriptstyle a$}] (w5) at (.25,-.5) {};
\node[ext,label=-90:{$\scriptstyle aaa$}] (w4) at (.7,-.5) {};
\node[ext,label=-90:{$\scriptstyle a$}] (w6) at (1.2,-.5) {};
\node[ext,label=-90:{$\scriptstyle a$}] (w7) at (1.5,-.5) {};
\draw (v) edge (w1)  edge (w2) edge[bend left] (w4) edge[bend right] (w4) edge (w4)
(v1) edge (w4) edge (w6) edge (w7);
\end{tikzpicture}
\mapsto
\begin{tikzpicture}[baseline=-.8ex]
\node[draw,circle] (v) at (0,.3) {$\Gamma_1$};
\node[draw,circle] (v1) at (1.25,.3) {$\Gamma_2$};
%\node[ext,label=-90:{$\scriptstyle a$}] at (-1.05,-.5) {};
\node[ext,label=-90:{$\scriptstyle  a$}] (w1) at (-.7,-.5) {};
\node[ext,label=-90:{$\scriptstyle a$}] (w2) at (-.25,-.5) {};
%\node[ext,label=-90:{$\dots$}] (w3) at (.25,-.5) {};
%\node[ext,label=-90:{$\scriptstyle a$}] (w5) at (.25,-.5) {};
\node[ext,label=-90:{$\scriptstyle aaa$}] (w4b) at (.7,-.5) {};
\node[int] (w4) at (.7,-.2) {};
\node[ext,label=-90:{$\scriptstyle a$}] (w6) at (1.2,-.5) {};
\node[ext,label=-90:{$\scriptstyle a$}] (w7) at (1.5,-.5) {};
\draw (v) edge (w1)  edge (w2) edge[bend left] (w4) edge[bend right] (w4) edge (w4)
(v1) edge (w4) edge (w6) edge (w7) (w4b) edge (w4);
\end{tikzpicture}
\]
The pieces thus produced contribute to both the differential and higher $L_\infty$-operations as depicted in \eqref{equ:deltajoin}, \eqref{equ:bracketpic}.
\item\label{description_differential_HGC}
The graph is internally connected (in particular $k=1$) and all vertices are already of valence one.
Then there are contributing pieces of the differential on $\Graphs_n^*$ splitting internal vertices, producing \eqref{equ:deltasplit}, and also another piece by the differential on $\bar R$.
\end{enumerate}
In the case of the computation of $\ell_1$, we retrieve the terms of the differential of $\HGC_{\bar R,n}$ depicted in the introduction
of the paper. Namely, in~(\ref{description_L_infinity_operations_HGC}), we retrieve the term $\delta_{join}$
of this differential $\delta  = d_R + \delta_{split} + \delta_{join}$,
whereas~(\ref{description_differential_HGC}) gives the term $\delta_{split}$
and the term $d_R$.
%
%To summarize the discussion, the $L_\infty$-structure on the hairy graph complex $\HGC_{\bar R,n}$ has the following combinatorial description.
%The differential acts by splitting vertices, acts on the decorations in $\bar R$, and fuses some subset of at least two hairs to a newly created vertex with a hair. The $\bar R$-decoration at the new hair is the product of the decorations at the hairs fused.
For the higher operations $\ell_k$, for $k\geq 2$, we retrieve the picture
given in the introduction \eqref{equ:bracketpic}
in the particular case $k=2$.
Indeed, from~(\ref{description_L_infinity_operations_HGC}), we obtain that the operations $\ell_k$, $k\geq 2$,
take $k$ graphs and glue a non-empty subset of hairs of each graph to a newly created vertex with a hair,
which we decorate by the product of the decorations of the fused hairs.

%See again \eqref{equ:deltasplit}, \eqref{equ:deltajoin}, \eqref{equ:bracketpic} in the introduction for an illustration.

In the sequel we use the following notation.
\begin{notation}\label{not:HGC}
For a non-unital dg commutative algebra $A$, we define
\[
\HGC_{A,n} := \HGC_{\bar A_*,n},
\]
where $A_* = \Q1\oplus A$ is obtained by adjoining a unit to $A$, so that $A = \bar A_*$.
\end{notation}

\subsection{Hairy graph complex for homotopy commutative algebras}\label{sec:homotopy_commutative_extension}\label{ss:HGC_homot}
The $L_\infty$-algebra $\HGC_{\bar R,n}$ of the previous subsection depends on the choice of a model $R$ for $M_*$, which can be relatively big in general.
In this section, we improve the situation by extending the construction $\HGC_{\bar R,n}$ to homotopy commutative algebras.
This allows us to take for $R$ the cohomology $H^*(M_*)$ of the space $M_*$,
which we equip with a homotopy commutative algebra structure that we can transfer from the Sullivan algebra $\Omega^*(M_*)$
using perturbation techniques. We refer to Kadeishvili's memoir~\cite{KadeishviliMemoir} for the application of this idea
in rational homotopy theory (see also~\cite{KadeishviliSurvey} for a survey).

We actually modify Kadeishvili's definition of this model of the rational homotopy of a space on the cohomology.
Namely, Kadeishvili uses $C_{\infty}$-algebras (called commutative $A_{\infty}$-algebras in the first cited reference),
which are identified with algebras over the operadic cobar construction of the (operadic suspension
of the) Lie cooperad $\Omega\Lie^c$ (see~\cite{GetzJones}, see also~\cite[Section 13.1]{LV} for an account of this correspondence).
But the transfer argument used by Kadeishvili to provide the cohomology $\bar H^*(M_*)$ with a homotopy commutative algebra structure
can also be applied to the cobar construction $\Omega L_{\infty}^c$,
where, instead of the Lie cooperad $\Lie^c$,
we consider the cooperad of $L_{\infty}$-coalgebras $L_{\infty}^c$.

In what follows, we use the notation  $\Lie$ for the operadic desuspension of the usual Lie operad
in the category of graded vector spaces,
so that the structure of an algebra over this operad $\Lie$
is governed by a Lie bracket of degree $1$,
as we assume in our grading conventions (see Section~\ref{ss:gen}).
Recall also that the operad that governs the category of $L_{\infty}$-algebras
is identified with the operadic cobar construction $L_{\infty} = \Omega\Com_+^c$
of the cooperad of cocommutative non-counital coalgebras $\Com_+^c$
(when we adopt our grading conventions for $L_{\infty}$-algebra structures again).
%To be precise, this operad $L_{\infty} = \Omega\Com_+^c$
%is generated by operations $\ell_r$
%of arity $r\geq 2$.
%Thus, when we use this operadic approach, we do not integrate a unary operation $\ell_1$ in the definition of an $L_{\infty}$-algebra structure,
%in contrast to the definition of Section~\ref{sec:HGC}-\ref{sec:HGCLie}.
%In fact, we can integrate the unary operation $\ell_1$ of this previous definition of an $L_{\infty}$-algebra
%in the differential of the dg vector space
%that underlies our object, and we follow this convention when we use the operadic approach
%of this paragraph. (We will adopt similar conventions for the definition of the homotopy commutative algebra structures that we use in this paragraph.)

This operad $L_{\infty} = \Omega\Com_+^c$ defines a resolution of the Lie operad $\Lie$,
with a quasi-isomorphism $L_{\infty} = \Omega\Com_+^c\xrightarrow{\sim}\Lie$
deduced from the Koszul duality of operads (see again~\cite{GetzJones} or~\cite[Section 13.2]{LV}).
The cooperads $\Lie^c$ and $L_{\infty}^c$, which we considered in the previous paragraph,
are dual to the operads $\Lie$ and $L_{\infty} = \Omega\Com_+^c$
in the category of (differential) graded vector spaces.
We also have $L_{\infty}^c = B\Com_+$, where we take the operadic bar construction $B$ of the operad of non-unital commutative algebras $\Com_+$.
We can dualize the quasi-isomorphism $L_{\infty}\xrightarrow{\sim}\Lie$ to get a quasi-isomorphism
of dg cooperads $\Lie^c\xrightarrow{\sim}L_{\infty}^c$.
This quasi-isomorphism induces a quasi-isomorphism of dg operads when we pass to the cobar construction $\Omega\Lie^c\xrightarrow{\sim}\Omega L_{\infty}^c$
and we also have a quasi-isomorphism of dg operads $\Omega L_{\infty}^c\xrightarrow{\sim}\nuCom$
by the operadic bar duality (see loc. cit.).
Hence, both dg operads $\Omega\Lie^c$ and $\Omega L_{\infty}^c$
are identified with resolutions
of the operad of non-unital commutative algebras $\nuCom$.

We consider the category of algebras over the cobar construction of the cooperad of $L_{\infty}$-coalgebras $\Omega L_{\infty}^c$
as a model for the category of homotopy commutative algebras
rather than the category of algebras over the dg operad $C_{\infty} = \Omega\Lie^c$
because we have an explicit formula for the $L_{\infty}$-structure of the hairy graph complex in this context,
whereas we only have a theoretical result asserting the existence of such a structure in the case of algebras
over the dg operad $C_{\infty} = \Omega\Lie^c$.
We just give brief explanations on this theoretical result at the end of the section.

For the moment, we can assume that $\COp$ is any dg cooperad among $\Lie^c$ and $L_{\infty}^c$.
Briefly recall that the structure of an algebra $A$ over the cobar construction of a dg cooperad $\Omega\COp$
is equivalent to the structure defined by a twisting coderivation $D_A$
on the cofree $\COp$-coalgebra $\COp(A)$,
where we call twisting coderivation a map $D_A: \COp(A)\rightarrow\COp(A)$,
of (cohomological) degree $1$, which is a coderivation with respect to the operations of the $\COp$-coalgebra structure on $\COp(A)$,
which is trivial on the dg vector space $A\subset\COp(A)$,
and which satisfies the relation $\delta D_A + D_A\delta + D_A^2 = 0$
with respect to the natural differential $\delta$ induced by the internal differential of the cooperad $\COp$ and of the dg vector space $A$
on the cofree $\COp$-coalgebra $\COp(A)$.
This twisting coderivation can also be determined by a map $\rho_A: \COp(A)\rightarrow A$, of degree $1$,
and which is also trivial on the dg vector space $A\subset\COp(A)$.
Furthermore, we can express the equation of twisting coderivations $\delta D_A + D_A\delta + D_A^2 = 0$
in terms of this map (see~\cite{FrCobar} for a detailed account of this correspondence).
(The assumption that $\rho_A$ vanishes on the dg vector space $A\subset\COp(A)$
reflects our convention that we integrate the unary operation
of a homotopy commutative algebra structure
in the differential $d_A$
of the dg vector space that underlies our object $A$.)

The vector spaces $L_{\infty}^c(r) = B\nuCom(r)$ are spanned by rooted trees $T$ with $r$ leaves indexed by $1,\dots,r$.
The degree of a tree $T$ with $p$ internal vertices is given by $-p$ (in cohomological conventions).
The map $\rho_A: \COp(A)\rightarrow A$, which determines the structure of an $\Omega L_{\infty}^c$-algebra on $A$,
can therefore be defined by giving a collection of maps
\[
\rho_A^T : A^{\otimes r}\rightarrow A,
\]
associated to such trees $T$, and such that $\rho_A^T(a_{\sigma(1)},\dots,a_{\sigma(r)}) = \rho_A^{\sigma T}(a_1,\dots,a_r)$ for every permutation $\sigma$,
where $\sigma T$ is defined by applying the permutation $\sigma$
to the indices of the leaves
of our tree.
We just take $\rho_A^{|} = 0$ for a trivial tree without vertices $T = |$, because this operation $\rho_A^{|}$
represents the value of the map $\rho_A: \COp(A)\rightarrow A$
on the summand $A\subset\COp(A)$ (which is trivial under our conventions).
We assume that $\rho_A^T$ has degree $1-p$ (for a tree with $p$ vertices) and the relations of an $\Omega L_{\infty}^c$-algebra structure,
equivalent to the equation of twisting coderivations for the coderivation $D_A$ corresponding to $\rho_A$,
are equivalent to the relations
\[
d_A\rho_A^T + \pm\rho_A^T d_A + \sum_{e}\rho_A^{T/e} + \sum_{T = S'\circ_{i_e} S''}\rho_A^{S'}\circ_{i_e}\rho_A^{S''} = 0,
\]
for the operations associated to the trees $T$,
where $d_A$ denotes the internal differential of the dg vector space $A$ (which acts by derivation in the case of the tensor product $A^{\otimes r}$),
the first sum runs over edge contraction operations $T\mapsto T/e$,
and the second sum runs over the operadic decompositions $T = S'\circ_{i_e} S''$
in the category of trees.
The sign $\pm$ corresponds to the permutation of the map $d_A$ of degree $1$ and of the operation $\rho_A^T$ of degree $1-p$,
and is determined by the general conventions of differential graded algebra.
Recall that the operadic composite $S'\circ_{i_e} S''$ of trees $S'$ and $S''$ is obtained by plugging the tree $S''$
into the leave of the tree $S'$ indexed by $i_e$.
In our sum, we actually consider decompositions $T = S'\circ_{i_e} S''$
where we can perform a shuffle of the indices of the ``free'' leaves of $S'$ and $S''$
inside the tree $T$.
The index $i_e$ is a dummy variable, because we actually assume that our sum runs over isomorphism classes
of such decompositions $T = S'\circ_{i_e} S''$.
The operation $\rho_A^{S'}\circ_{i_e}\rho_A^{S''}$ is given by $\rho_A^{S'}\circ_{i_e}\rho_A^{S''}(a_1,\dots,a_r) = \pm\rho_A^{S'}(a_{i_1},\dots,\rho_A^{S''}(a_{j_1},\dots,a_{j_{\ell}}),\dots,a_{i_k})$,
for all $a_1,\dots,a_r\in A$, where $(i_1,\dots,\widehat{i_e},\dots,i_k)$ and $(j_1,\dots,j_{\ell})$
reflects the shuffle of the indices of the ``free'' leaves of the trees $S'$ and $S''$
inside $T$. (The sign $\pm$ corresponds to the permutation of variables involved in this formula
and is again determined by the general conventions of differential graded algebra.)

Every (non-unital) dg commutative algebra is an $\Omega L_{\infty}^c$-algebra by restriction of structure
through the operad morphism $\Omega L_{\infty}^c\rightarrow\nuCom$.
The $\Omega L_{\infty}^c$-algebra structure of a dg commutative algebra
can also be determined by the following formula:
\[
\rho_A^T(a_1,\dots,a_r) = \begin{cases}
0, & \text{if $T$ has more than 1 vertex}, \\
a_1\cdots a_r, & \text{otherwise}.
\end{cases}
\]
By general results of the homotopy theory of algebras over operads (for which we refer to~\cite{Hinich}),
the category of $\Omega L_{\infty}^c$-algebras
is endowed with a natural model structure (like the category of non-unital dg commutative algebras),
which is transferred from the category of dg vector spaces.
Besides, we obtain that the functors of extension of structure and of restriction of structure
associated to the quasi-isomorphism $\Omega L_{\infty}^c\rightarrow\nuCom$
define a Quillen equivalence.
In particular, every $\Omega L_{\infty}^c$-algebra is quasi-isomorphic to non-unital dg commutative algebra.

%There are several definitions of unital (and augmented) homotopy commutative algebras.
%For our purpose an augmented homotopy commutative algebra structure on $A$ is just defined to be a homotopy commutative algebra structure on the kernel $\bar A=\ker \epsilon$ of some surjective map of complexes $\epsilon : A\to \Q$.

We now describe our definition of an $L_\infty$-algebra structure on $\HGC_{A,n}$ when $A$ is an $\Omega L_{\infty}^c$-algebra.
We modify the definition of the component $\delta_{join}$ of the differential $\delta  = d_A + \delta_{split} + \delta_{join}$
by taking:
\[
\delta_{join}
\begin{tikzpicture}[baseline=-.8ex]
\node[draw,circle] (v) at (0,.3) {$\Gamma$};
\node (w1) at (-.7,-.5) {$a_1$};
\node (w2) at (-.25,-.5) {$a_2$};
\node (w3) at (.25,-.5) {$\dots$};
\node (w4) at (.7,-.5) {$a_k$};
\draw (v) edge (w1) edge (w2) edge (w3) edge (w4);
\end{tikzpicture}
= \sum_T\sum_{\substack{S\subset {\rm hairs} \\ |S|\geq 2 }} \pm
\begin{tikzpicture}[baseline=-.8ex]
\node[draw,circle] (v) at (0,.3) {$\Gamma$};
\node (w1) at (-.7,-.5) {$a_1$};
\node (w2) at (-.25,-.5) {$\dots$};
\node[draw, circle, dotted] (T) at (.5,-.66) {$T$};
\node (w4) at (.5,-1.66) {$\rho_A^T(\{a_i\}_{i\in S})$};
\draw (v) edge (w1) edge (w2) edge[bend left] (T) edge (T) edge[bend right] (T) (w4) edge (T);
\end{tikzpicture}\,.
\]
The variable $T$ runs over the elements of a tree basis of $L_{\infty}^c = B\nuCom$.
The variable $S$ runs over the subsets of hairs of the graph $\Gamma$.
The hairs of $S$ are plugged in the leaves of the tree $T$ and the root of that tree becomes a new hair decorated by $\rho_A^T$
applied to the respective decorations of the hairs in $S$.
In this process, we just discard the tree $T = |$ (and we therefore assume that our hair set $S$ has at least two elements)
since we have $\rho_A^{|} = 0$.

For graphs $\Gamma_1,\dots,\Gamma_r\in \HGC_{A,n}$ with $k_1,\dots,k_r$ hairs,
decorated by elements $a_1^i,\dots,a_{k_1}^i\in A$, $i = 1,\dots,r$,
we define the higher $r$-ary $L_\infty$-operations $\ell_r$, $r\geq 2$,
by the following similar formulas:
\[
\ell_r(\Gamma_1,\dots,\Gamma_r)
=
\sum_T
\sum_S
\pm
\begin{tikzpicture}[baseline=-.8ex]
  \node[draw,circle] (v1) at (0,.3) {$\Gamma_1$};
  \node[draw,circle] (v2) at (1.25,.3) {$\Gamma_2$};
  \node (vd) at (2.5,.3) {$\cdots$};
  \node[draw,circle] (vr) at (3.75,.3) {$\Gamma_r$};
  \node[draw, circle, dotted] (T) at (2.5,-1) {$T$};
  \node (fT) at (2.5,-2) {$\rho_A^T(\{a_i\}_{i\in S})$};
  \draw (T) edge (fT) (v1) edge (T) edge[bend right] (T)
  (v2) edge (T) (vr) edge (T) (v1) edge +(-1,-1) (vr) edge +(0,-1) edge +(1,-1);
  %\draw (v) edge (w1) edge[bend right] (w1) edge (w2) edge[bend left] (w4) edge[bend right] (w4) edge (w4)
  %(v2) edge (w9) edge[bend left] (w9) (v1) edge (w4) edge (w6) edge (w9) edge (w8);
  \end{tikzpicture}\, .
\]
The variable $T$ runs again over the elements of a tree basis of $L_{\infty}^c = B\nuCom$.
The variable $S$ runs over the subsets of hairs of the graphs $\Gamma_1,\dots,\Gamma_r$
such that each graph contributes by at least one hair.

The validity of this construction is easy to check from our explicit description of the structure of an $\Omega L_{\infty}^c$-algebra.
We also immediately see that this construction is functorial with respect to morphisms of $\Omega L_{\infty}^c$-algebras
and that a quasi-isomorphism of $\Omega L_{\infty}^c$-algebras induces a quasi-isomorphism
on the hairy graph complex by a standard spectral sequence argument.
Hence our construction provides a coherent generalization of the decorated graph complex of commutative algebras.
Note simply that we deal with non-unital algebras when we consider the category of $\Omega L_{\infty}^c$-algebras,
which are therefore equivalent to the augmentation ideal $\bar R$ of augmented unital dg commutative algebras $R$
rather than to the (augmented) unital dg commutative algebras
that we considered so far.

In the context of homotopy commutative algebras, we also have a class of homotopy morphisms (also called $\infty$-morphisms in the literature),
which are equivalent to morphisms of the homotopy category.
In general, a homotopy morphism of algebras over the cobar construction $\Omega\COp$ of a dg cooperad $\COp$,
can be defined by giving a morphism of $\COp$-coalgebras $F: (\COp(A),D_A)\rightarrow(\COp(B),D_B)$,
where we consider the twisted cofree $\COp$-coalgebras $(\COp(A),D_A)$ and $(\COp(B),D_B)$
associated to the $\Omega\COp$-algebra structures on $A$ and $B$ (see \cite{FrCobar} for an account of the correspondence
between this class of homotopy morphisms
and the morphisms of the homotopy category of $\Omega\COp$-algebras).
Such a morphism can also be determined by a map $\psi: \COp(A)\rightarrow B$, of degree $0$, and which satisfies coherence constraints
which are equivalent to the preservation of differentials
when we pass to the morphism of $\COp$-coalgebras $F: (\COp(A),D_A)\rightarrow(\COp(B),D_B)$ (see again \cite{FrCobar} for an account of this correspondence).
Note that we do not take the convention that the map $\psi: \COp(A)\rightarrow B$
vanishes on the summand $A\subset\COp(A)$
in our definition of homotopy morphisms.
This linear component of our map can actually be identified with a morphism of dg vector spaces $f: A\rightarrow B$
which underlies our homotopy morphism
of $\Omega\COp$-algebras.

In the case $\COp = L_{\infty}^c$, the map $\psi: \COp(A)\rightarrow B$,
which determines a homotopy morphism of $\Omega L_{\infty}^c$-algebras,
can be defined by giving a collection of maps
\[
\psi_T : A^{\otimes r}\rightarrow B,
\]
associated to the trees $T$,
and such that we again have the equivariance relation $\psi_T(a_{\sigma(1)},\dots,a_{\sigma(r)}) = \psi_{\sigma T}(a_1,\dots,a_r)$
for every permutation $\sigma$.
Note that the linear part of $\psi$, thus the morphism of dg vector spaces $f: A\rightarrow B$
underlying $\psi$,
is encoded by the map $\psi_{|}: A\to B$ associated to the trivial tree $T = |$.
We assume that $\psi_T$ has degree $-p$ (for a tree with $p$ internal vertices) and the coherence constraints of homotopy morphisms
is equivalent to the relations
\[
d_B\psi_T - \pm\psi_T d_A + \sum_{T = S'(S''_1,\dots,S''_m)}\rho_B^{S'}(\psi_{S''_1},\dots,\psi_{S''_m})
= \sum_{e}\psi_{T/e} + \sum_{T = S'\circ_{i_e} S''} \psi_{S'}\circ_{i_e}\rho_A^{S''}
\]
where $d_A$ (respectively, $d_B$) denotes the internal differential of the dg vector space $A$ (respectively, $B$),
and $\rho_A^S$ (respectively, $\rho_B^S$) denotes the operations that determine the $\Omega L_{\infty}^c$-algebra structure
of $A$ (respectively, $B$).
The first sum runs over the operadic tree decompositions $T = S'(S''_1,\dots,S''_m)$, where $S'$ is a tree with $m$ leaves in which we plug the trees $S''_1,\dots,S''_m$.
We also assume that we can perform a shuffle of the indices of the leaves of the trees $S''_1,\dots,S''_m$
inside $T$.
The map $\rho_B^{S'}(\psi_{S''_1},\dots,\psi_{S''_m})$
is given by $\rho_B^{S'}(\psi_{S''_1},\dots,\psi_{S''_m})(a_1,\dots,a_r) = \pm\rho_B^{S'}(\psi_{S''_1}(a_{i_1^1},\dots,a_{i_{k_1}^1}),\dots,\psi_{S''_1}(a_{i_1^m},\dots,a_{i_{k_m}^m}))$,
for all $a_1,\dots,a_r\in A$, where $(i_1^1,\dots,i_{k_1}^1,\dots,i_1^m,\dots,i_{k_m}^m)$ reflects
again the shuffle of the indices of the leaves of the trees $S''_1,\dots,S''_m$
inside $T$.

To any such homotopy morphism, we associate the $L_\infty$-morphism
\[
\Psi \colon \HGC_{A,n}\to \HGC_{B,n}
\]
whose $r$-ary component acts on graphs schematically as follows
\[
\Psi_r(\Gamma_1,\dots,\Gamma_r)
=
\sum
\pm
\begin{tikzpicture}[baseline=-.8ex]
  \node[draw,circle] (v1) at (0,.3) {$\Gamma_1$};
  \node[draw,circle] (v2) at (1.25,.3) {$\Gamma_2$};
  \node (vd) at (2.5,.3) {$\cdots$};
  \node[draw,circle] (vr) at (3.75,.3) {$\Gamma_r$};
  \node[draw, circle, dotted] (Tk) at (3,-1) {$T_k$};
  \node (vdv) at (1.75,-1) {$\cdots$};
  \node[draw, circle, dotted] (T1) at (.5,-1) {$T_1$};
  \node (fT1) at (.5,-2) {$\psi_{T_1}(\dots)$};
  \node (fTk) at (3,-2) {$\psi_{T_k}(\dots)$};
  \draw (T1) edge (fT1) (Tk) edge (fTk)
  (v1) edge (T1) edge[bend right] (T1)
  (v2) edge (T1) (v2) edge (Tk)
  (vr) edge (Tk) edge[bend left] (Tk)   ;
  %(v1) edge +(-1,-1) (vr) edge +(0.3,-1) edge +(1,-1);
  %\draw (v) edge (w1) edge[bend right] (w1) edge (w2) edge[bend left] (w4) edge[bend right] (w4) edge (w4)
  %(v2) edge (w9) edge[bend left] (w9) (v1) edge (w4) edge (w6) edge (w9) edge (w8);
  \end{tikzpicture}
\]
Here we sum over all ways of connecting the hairs to a forest. Each hair of every graph $\Gamma_i$ must go into some tree $T_j$
(which can be  the trivial tree~$|$).
The new hair at the root of the tree $T_j$ in the forest is decorated by $\psi_{T_j}$ applied to the decorations of the hairs the tree connects to. Note that
each map $\Psi_r$ has degree zero.

In the case of the cohomology $H^*(X)$ of a space $X$, the $\Omega L_{\infty}^c$-algebra that makes $H^*(X)$ a model
of the rational homotopy of the space $X$
is obtained by transfer, after picking a quasi-isomorphism of dg vector spaces $H^*(X)\xrightarrow{\sim}\Omega^*(X)$
where we regard $H^*(X)$ as a dg vector space equipped with a trivial differential.
The transferred structure can be defined by applying perturbation methods to the coalgebras (as in~\cite{KadeishviliMemoir}
and in~\cite{CL})
or by model category arguments (see~\cite{BM}
and \cite{FrTransfer}). In all cases, the obtained object $H^*(X)$ is connected to $\Omega^*(X)$
by a zigzag of quasi-isomorphisms of $\Omega L_{\infty}^c$-algebras,
and therefore, we can take this model in our applications to graph complexes.
Note simply that we take the augmentation ideal of the cohomology algebra for the pointed space $X = M_*$
rather than the cohomology algebra itself in these applications.

We mentioned at the beginning of this subsection that we can also lift the definition of the decorated hairy graph complex
to the classical category of $C_{\infty}$-algebras (at least theoretically).
We can obtain such a result by observing that we can define a morphism $\Omega L_{\infty}^c\rightarrow\Omega\Lie^c$
by lifting the quasi-isomorphism $\Omega\Lie^c\xrightarrow{\sim}\nuCom$,
using that $\Omega L_{\infty}^c$ forms a cofibrant object in the category of dg operads.
Indeed, this observation implies that every $C_{\infty}$-algebra $A$ inherits an $\Omega L_{\infty}^c$-algebra structure by restriction of structure,
and as a result, we can use the construction of this paragraph to associate a decorated hairy graph complex $\HGC_{A,n}$
to $A$.

We finally note that to compute the set of equivalence classes of MC elements,
it is enough to know a $C_\infty$ structure of~$A$.
In other words, one does not need to produce its lift to an $\Omega L_{\infty}^c$ structure, see Remark~\ref{r:UTT}.

\section{Proof of the main Theorems}\label{sec:the proofs}

\subsection{Proof of Theorem \ref{thm:main1}}\label{ss:proof_thm_main1}
Our goal is to compute the homotopy type of the mapping space
\[
\IBimod^h_{\FM_m}(\IF_M, \FM_n^\Q)=\IBimod^h_{\FM_m}(\IF_M,(\Res^{\FM_n^\Q}_{\FM_m}\FM_n^\Q))
\]
Here we consider the category of infinitesimal bimodules as equipped with the Reedy model structure. (Recall
that the homotopy type of the derived mapping space between two objects is the same for the projective and Reedy model structures.)
Let us transform the mapping space as follows, using diagram \eqref{equ:formality diag}
%\begin{align*}
%\IBimod^h_{\FM_m}(\IF_M, (\Res^{\FM_n}_{\FM_m}\FM_n)^\Q)
%&\simeq
%\IBimod^h_{\FM_m}(\IF_M, \Res^{\FM_n^\Q}_{\FM_m}\FM_n^\Q)
%\simeq
%\IBimod^h_{\FM_m}(\IF_M, \Res^{R_n^\Q}_{\FM_m}R_n^\Q)
%\\&\simeq
%\IBimod^h_{\FM_m}(\IF_M, \Res^{H_n}_{\FM_m}H_n)
%=
%\IBimod^h_{\FM_m}(\IF_M, \Res^{\Com}_{\FM_m}\Res^{H_n}_{\Com}H_n).
%\end{align*}
\begin{multline*}
\IBimod^h_{\FM_m}(\IF_M, \Res^{\FM_n^\Q}_{\FM_m}\FM_n^\Q)
\simeq\IBimod^h_{\FM_m}(\IF_M, \Res^{R_n^\Q}_{\FM_m}R_n^\Q)
\simeq \\
\IBimod^h_{\FM_m}(\IF_M, \Res^{H_n}_{\FM_m}H_n)
\simeq\IBimod^h_{\FM_m}(\IF_M, \Res^{\Com}_{\FM_m}\Res^{H_n}_{\Com}H_n).
\end{multline*}
Here we used that weak equivalences of infinitesimal bimodules induce weak equivalences on the derived mapping space.
For the final equality we used that the lower composition in \eqref{equ:formality diag}
factorizes through the commutative operad $\Com = *$.

We now use the adjunction between induction and restriction (Proposition \ref{prop:resind}) to obtain the following extra simplification:
\[
\IBimod^h_{\FM_m}(\IF_M, \Res^{\Com}_{\FM_m}\Res^{H_n}_{\Com}H_n)
\cong
\IBimod^h_{\Com}(L\Ind_{\FM_m}^{\Com} \IF_M, \Res^{H_n}_{\Com}H_n).
\]
We finally use Proposition \ref{prop:indEM} to find that
\[
\IBimod^h_{\FM_m}(\IF_M, (\Res^{\FM_n}_{\FM_m}\FM_n)^\Q)
\simeq
\IBimod^h_{\Com}(M_*^{\times\bullet}, \Res^{H_n}_{\Com}H_n).
\]
Now, infinitesimal $\Com$-bimodules are the same as topological right $\Gamma$-modules. We can hence write the result of this relation as
\[
\IBimod^h_{\Com}(M_*^{\times\bullet},\Res^{H_n}_{\Com}H_n) = \TopCat^{\Gamma^{op},h}(M_*^{\times\bullet},H_n).
\]
%As yet, we are using the Reedy model structure on right $\Gamma$-modules.
%We will now switch to the projective model structure, again using that the choice of model structure is irrelevant for the mapping space given the weak equivalences.

Recall that the operad $H_n$ is given by $H_n = LG(H^*(\FM_n))$, where we consider the derived functor of the Sullivan realization functor on dg commutative algebras $G$,
and that this identity $H_n = LG(H^*(\FM_n))$ also holds in the category of right $\Gamma$-modules.
From the observations of Section~\ref{sec:gamma_rational_homotopy},
%the adjunction \eqref{equ:GOmega2} (and the Quillen equivalence between the model categories of right $\Gamma$-modules in topological spaces and in simplicial sets),
we can simplify our result further to
\begin{multline*}
\TopCat^{\Gamma^{op},h}(M_*^{\times\bullet},LG(H^*(\FM_n)))
\simeq
\DGCA^{\Gamma,h}(H^*(\FM_n),\Omega^*(M_*^{\times\bullet}))
\simeq \\
\DGCA^{\Gamma,h}(H^*(\FM_n),R^{\otimes\bullet}),
\end{multline*}
where in the last relation we pick an (augmented) Sullivan model $R$ of the space~$M_*$.
%Note that at this point we are using the injective model structure on Hopf right $\Gamma$-modules.
%Again we invoke \cite[Proposition 4.4]{DwyerKan} and switch gears to the projective model structure.

To compute the derived mapping space we use the graph complex model
\[
\Graphs_n\simeq H^*(\FM_n)
\]
from Section \ref{sec:graphs}.
By Proposition \ref{prop:graphscofibrant} we have that $\Graphs_n$ is cofibrant as a left Hopf $\Gamma$-module (with respect to the projective model structure).
Furthermore $R^{\otimes\bullet}$ is fibrant since so is any object in the projective model structure.
We can hence write
\[
\DGCA^{\Gamma,h}(H^*(\FM_n),R^{\otimes\bullet})\simeq\DGCA^{\Gamma}(\Graphs_n,R^{\otimes\bullet})
\]
to compute the mapping space we use the simplicial frame from Lemma \ref{lem:simplframe}.
By combining this expression with the result of Proposition~\ref{prop:HGCLinfty}, we obtain:
\begin{align*}
\DGCA^{\Gamma}(\Graphs_n, R^{\otimes\bullet}) & = \Mor_{H\Gamma}(\Graphs_n,  R^{\otimes\bullet}\otimes \Omega^*(\Delta^\bullet))\\
& = \MC(\Def(\Graphs_n,R^{\otimes\bullet}\otimes\Omega^*(\Delta^\bullet)))\\
& = \MC(\Def(\Graphs_n,R^{\otimes\bullet})\hat\otimes\Omega^*(\Delta^\bullet))
= \MC_\bullet(\HGC_{\bar R,n}).	
\end{align*}
%\begin{align*}
%\DGCA^{\Gamma,h}(\Graphs_n, R^{\otimes\bullet})
%&=
%\Mor_{H\Gamma} (\Graphs_n,  R^{\otimes\bullet}\otimes \Omega^*(\Delta^\bullet))
%=
%\MC(\Def(\Graphs_n,  R^{\otimes\bullet}\otimes \Omega^*(\Delta^\bullet)) )
%\\ &=
%\MC(\Def(\Graphs_n,  R^{\otimes\bullet}) \hat \otimes \Omega^*(\Delta^\bullet) )
%= \MC_\bullet(\HGC_{n, \bar R}).	
%\end{align*}

%\section{Proof of Theorem \ref{thm:main2}}
%
%\section{Version for manifolds without basepoint}\label{sec:nonlong}
%Let $M\subset \R^m$ be an open subset.
%By mapping $\R^m$ homeomorphically to (e.g.) a ball we may assume that $M$ is bounded.
%Then we may add a disjoint contractible open neighborhood to $M$ to form an open $M_*\subset \R^m$.
%It is clear that
%\[
% \Embbar_\p(M_*,\R^n) \simeq \Embbar(M,\R^n)\, .
%\]
%Hence our methods are applicable also the ordinary embedding space modulo immersions $\Embbar(M,\R^n)$.
%
%If $R$ is a (non-augmented) DGCA model for $M$, then adjoining a unit we obtain an augmented DGCA model
%\[
% R_*=R\oplus \Q 1
%\]
%for $M_*$. Inserting into Theorem \ref{thm:main1} we find that the graph complex controlling the rational homotopy type of connected components  of the embedding space (for $n-m\geq 3$) is
%\[
% \HGC_{\bar R_*,n} = \HGC_{R,n}.
%\]

\subsection{Proof of Theorem \ref{thm:FM}}\label{ss:proof_thm_FM}
Our next goal is to prove that the mapping spaces $\IBimod^h_{\FM_m}(\IF_M,\FM_n)$ are nilpotent under the assumptions of Theorem \ref{thm:FM},
that the rationalization of these mapping spaces are equivalent to the mapping spaces $\IBimod^h_{\FM_m}(\IF_M,\FM_n^{\Q})$
with values in the rationalization of the infinitesimal bimodule $\FM_n$,
and that the rationalization map is finite-to-one
on homotopy classes.

We rely on Mienn\'e's theory of Postnikov decompositions of operads and of infinitesimal bimodules over operads,
as we briefly explained in the introduction
of the paper.
We refer to~\cite{MienneThesis} for a detailed account of this theory in the context of operads empty in arity zero
(the category of non-unitary operads in the terminology of loc. cit.).
The extension of the theory to reduced operads (called unitary operads in loc. cit.)
and to infinitesimal bimodules
is the subject of the memoir~\cite{MienneMemoir} in preparation.
The theory is very similar in this case, using concepts introduced in~\cite{FrII} for the study of the homotopy of reduced operads
and in~\cite{DFT} for the study of the homotopy of infinitesimal bimodules.
Notably, we use that every reduced operad $\op P$ has a decomposition ${\op P} = \lim_s\cosk^{\Lambda}_s{\op P}$
such that
\[
\Mor_{\Op}(-,\cosk_{\leq s}\op P) = \Mor_{\Op_{\leq s}}(-,\op P),
\]
where $\Op_{\leq s}$ denotes the category of $s$-truncated (reduced) operads,
the category of operads that are defined
up to arity $\leq s$,
and we consider the image of the object $\op P$ under the obvious forgetful functor $(-)_{\leq s}: \Op\rightarrow\Op_{\leq s}$.

%We refer to \cite[Section II.8.4]{FrII} for the definition of this coskeletal decomposition ${\op P} = \lim_s\cosk^{\Lambda}_s{\op P}$,
%of the operad ${\op P}$ (the word coskeletal is the terminology adopted in this reference).
%We also refer to \cite{DFT} for the definition of an analogous coskeletal decomposition in the context of infinitesimal bimodules.
%In what follows, we mainly use that we have $\cosk^{\Lambda}_s{\op P}(s) = {\op P}(s)$
%and that $\cosk^{\Lambda}_{s-1}{\op P}(s)$
%is identified with the $s$th matching object of the reduced operad ${\op P}$,
%the simplicial set such that
%\[
%M({\op P})(s) = \lim_{{\substack u\in\Mor_{\Lambda}(\underline{r},\underline{s})}\atop {r<s}}{\op P}(r),
%\]
%where we consider the action of the category $\Lambda$ on ${\op P}$.
%We have an analogous result in the category of infinitesimal bimodules.

The coskeletal decomposition  ${\op P} = \lim_s\cosk^{\Lambda}_s{\op P}$ was introduced in
 \cite[Section II.8.4]{FrII} in the context of reduced operads and was used in~\cite{DFT}  in the context of infinitesimal bimodules. In both settings $\cosk^{\Lambda}_s{\op P}$ is determined,  as a symmetric sequence, 
 by the $\Lambda$-structure of ${\op P}$:
 \[
\cosk^{\Lambda}_s{\op P}(n) = \lim_{{\substack u\in\Mor_{\Lambda}(\underline{r},\underline{n})}\atop {r\leq s}}{\op P}(r),
\]
In what follows, we mainly use that we have $\cosk^{\Lambda}_s{\op P}(s) = {\op P}(s)$
and that $\cosk^{\Lambda}_{s-1}{\op P}(s)$
is identified with $M({\op P})(s)$, the $s$th matching object of the reduced operad ${\op P}$.

If ${\op P}$ is Reedy fibrant, then the morphisms $\cosk^{\Lambda}_s{\op P}\rightarrow\cosk^{\Lambda}_{s-1}{\op P}$
that define this coskeletal decomposition of the operad are Reedy fibrations,
and in particular, form a fibration of simplicial sets $\cosk^{\Lambda}_s{\op P}(r)\rightarrow\cosk^{\Lambda}_{s-1}{\op P}(r)$
in each arity~$r$.
The Postnikov decompositions that we consider are decompositions of these morphisms
into towers of morphisms:
\begin{multline*}
\cosk^{\Lambda}_s{\op P} = \lim_t P_t(\cosk^{\Lambda}_s{\op P}/\cosk^{\Lambda}_{s-1}{\op P})
\rightarrow\cdots\\
\cdots\rightarrow P_t(\cosk^{\Lambda}_s{\op P}/\cosk^{\Lambda}_{s-1}{\op P})
\rightarrow P_{t-1}(\cosk^{\Lambda}_s{\op P}/\cosk^{\Lambda}_{s-1}{\op P})\rightarrow\cdots\\
\cdots\rightarrow P_0(\cosk^{\Lambda}_s{\op P}/\cosk^{\Lambda}_{s-1}{\op P}) = \cosk^{\Lambda}_{s-1}{\op P},
\end{multline*}
which we obtain by applying the classical construction
of Postnikov decompositions
to the morphisms $\cosk^{\Lambda}_s{\op P}(r)\rightarrow\cosk^{\Lambda}_{s-1}{\op P}(r)$
in the category of simplicial sets (we just check that the operad structure
passes to the Postnikov sections).
The morphisms $P_t(\cosk^{\Lambda}_s{\op P}/\cosk^{\Lambda}_{s-1}{\op P})\rightarrow P_{t-1}(\cosk^{\Lambda}_s{\op P}/\cosk^{\Lambda}_{s-1}{\op P})$
are still Reedy fibrations.

In the simplifying case where ${\op P}$ consists of simply connected spaces, the main result of Mienn\'e's theory asserts
that such morphisms fit in Postnikov extension diagrams of the following form
in the category of $s$-truncated operads:
\[
\begin{tikzcd}
P_t(\cosk^{\Lambda}_s{\op P}/\cosk^{\Lambda}_{s-1}{\op P})_{\leq s}\ar{r}\ar{d} & L(N\pi_t{\op P}(s),t+1)\ar{d} \\
P_{t-1}(\cosk^{\Lambda}_s{\op P}/\cosk^{\Lambda}_{s-1}{\op P})_{\leq s}\ar{r} & K(N\pi_t{\op P}(s),t+1)
\end{tikzcd}
\]
where $N\pi_t{\op P}(s)$ is the homotopy group $\pi_t$ of the fiber of the matching map ${\op P}(s)\rightarrow M{\op P}(s)$
regarded as an additive operad concentrated in arity $s$,
we consider the associated Eilenberg--MacLane space $K(N\pi_t{\op P}(s),t+1)$
and the corresponding path space fibration sequence $K(N\pi_t{\op P}(s),t)\rightarrow L(N\pi_t{\op P}(s),t+1)\rightarrow K(N\pi_t{\op P}(s),t+1)$
with $L(N\pi_t{\op P}(s),t+1) = P K(N\pi_t{\op P}(s),t+1)\sim *$.

We have an analogous statement in the context of infinitesimal bimodules.
In the case of an operad $\op P$ (which we regard as an infinitesimal bimodule over itself),
we can also use the Postnikov decompositions and the above Postnikov extension diagrams
in the category of operads to get the Postnikov decomposition of our object $\op P$
in the category of infinitesimal bimodules.

We use these decompositions to study mapping spaces of infinitesimal bimodules $\IBimod^h_{\op R}({\op M},{\op P})$,
where $\op R$ is any operad endowed with an operad map ${\op R}\to{\op P}$.
We get, at the mapping space level, the tower decomposition
\begin{multline*}
\IBimod^h_{\op R}({\op M},{\op P}) = \lim_s\IBimod^h_{\op R,\leq s}({\op M},{\op P})
\rightarrow\cdots\\
\cdots\rightarrow\IBimod^h_{\op R,\leq s}({\op M},{\op P})\xrightarrow{(*)}\IBimod^h_{\op R,\leq s-1}({\op M},{\op P})
\rightarrow\cdots\\
\cdots\rightarrow\IBimod^h_{\op R,\leq 0}({\op M},{\op P}) = *,
\end{multline*}
which we can associate to the coskeletal decomposition
of the target object
since $\IBimod^h_{\op R,\leq s}({\op M},{\op P}) = \IBimod^h_{\op R}({\op M},\cosk_s^{\Lambda}{\op P})$,
and we decompose the morphisms of this tower further,
as
\begin{multline*}
\IBimod^h_{\op R,\leq s}({\op M},{\op P}) = \lim_t\IBimod^h_{\op R,\leq s}({\op M},P_t(\cosk^{\Lambda}_s{\op P}/\cosk^{\Lambda}_{s-1}{\op P}))
\rightarrow\cdots\rightarrow\\
%\cdots\rightarrow
\IBimod^h_{\op R,\leq s}({\op M},P_t(\cosk^{\Lambda}_s{\op P}/\cosk^{\Lambda}_{s-1}{\op P}))
\xrightarrow{(**)}\IBimod^h_{\op R,\leq s}({\op M},P_{t-1}(\cosk^{\Lambda}_s{\op P}/\cosk^{\Lambda}_{s-1}{\op P}))
%\rightarrow\cdots
\\
\rightarrow\cdots\rightarrow\IBimod^h_{\op R,\leq s}({\op M},P_0(\cosk^{\Lambda}_s{\op P}/\cosk^{\Lambda}_{s-1}{\op P}))
= \IBimod^h_{\op R,\leq s-1}({\op M},{\op P}),
\end{multline*}
by using the Postnikov decomposition on the target.
We get that the morphisms (**),
%\begin{equation}\tag{$*$}\label{eq:double_tower_corners}
in the latter tower of mapping spaces, are identified with the principal fibrations of simplicial sets
with the mapping spaces
\[
G\langle s,t\rangle = \IBimod^h_{\op R,\leq s}({\op M},K(N\pi_t{\op P}(s),t))
\]
as fibers. We have in general:
\[
\pi_i\IBimod^h_{\op R,\leq s}({\op M},K(\pi(s),t)) = \tilde{H}^{t-i}_{\Sigma_s}(L\Indec{\op M}(s),\pi(s)),
\]
for any $\Sigma_s$-module $\pi(s)$ (regarded as an infinitesimal bimodule concentrated in arity $s$),
where $L\Indec$ denotes the (derived) functor of indecomposables from the category of infinitesimal bimodules
to the category of right $\Lambda$-modules in pointed spaces,
and we consider the reduced $\Sigma_s$-equivariant cohomology of the based $\Sigma_s$-space $L\Indec{\op M}(s)$
with coefficients in $\pi(s)$.
The indecomposable quotient $\Indec{\op M}$ of an infinitesimal bimodule ${\op M}$ is defined by moding out the spaces ${\op M}(r)$
by the $\Sigma_r$-subspaces spanned by the image of the composition operations $\circ_i: {\op R}(k)\rightarrow{\op M}(r-k+1)\rightarrow{\op M}(r)$
and $\circ_i: {\op M}(r-k+1)\rightarrow{\op R}(k)\rightarrow{\op M}(r)$
such that $k\geq 2$ (when we assume that the operad $\op R$ is reduced).
This object $\Indec{\op M}$ naturally inherits the structure of a right $\Lambda$-module in pointed simplicial sets.
Note that, in the above formula, we consider the reduced homology relative to the natural base point
of the indecomposable quotient.

This refined tower decomposition of mapping spaces of infinitesimal bimodules is used to establish the claims of the following statement.
We proceed by induction, by using the expression of the fibers of our decomposition,
and we use an analysis of the connectivity of these fibers
to pass to the limit of the tower
whenever necessary.

\begin{thm}[{\cite{MienneMemoir}}]
Let $\op P$ be a Reedy fibrant object in the category of infinitesimal bimodules over an operad $\op R$.
For simplicity, we are still going to assume that $\op P$ consists of simply connected spaces.
Let $\op M$ be any infinitesimal $\op R$-bimodule.
\begin{enumerate}
\item
If we have a finite dimension bound $m(s)$, in each arity $s\geq 1$,
such that
\[
H^i_{\Sigma_s}(L\Indec{\op M}(s),\pi(s)) = 0\quad\text{for $i\leq m(s)$},
\]
for any choice of $\Sigma_s$-module of coefficients $\pi(s)$, then the mapping spaces $\IBimod^h_{\op R,\leq s}({\op M},{\op P})$
are nilpotent, for all $s\geq 1$,
and satisfy the relation
\[
\IBimod^h_{\op R,\leq s}({\op M},{\op P})_{\phi}^{\Q}\sim\IBimod^h_{\op R,\leq s}({\op M},{\op P}^{\Q})_{\hat{\phi}},
\]
for any choice of base point $\phi\in\IBimod^h_{\op R,\leq s}({\op M},{\op P})$,
where $\hat{\phi}$ denotes the composite of $\phi: {\op M}_{\leq s}\rightarrow{\op P}_{\leq s}$
with the rationalization map ${\op P}_{\leq s}\rightarrow{\op P}_{\leq s}^{\Q}$.
\item
If we assume further that the matching map ${\op P}(s)\rightarrow M{\op P}(s)$ is $n(s)$-connected
for a connectivity bound $n(s)$ such that $n(s)-m(s)\rightarrow\infty$,
then the connectivity of the maps $\IBimod^h_{\op R,\leq s}({\op M},{\op P})\rightarrow\IBimod^h_{\op R,\leq s-1}({\op M},{\op P})$,
which connect the truncated mapping spaces
of the previous assertion,
tends to infinity when $s\rightarrow\infty$, and we can also pass to the limit $s\rightarrow\infty$
in our statement.
Thus, we get that the ``total'' mapping space $\IBimod^h_{\op R}({\op M},{\op P})$ is also nilpotent in this case,
and we still have the relation
\[
\IBimod^h_{\op R}({\op M},{\op P})_{\phi}^{\Q}\sim\IBimod^h_{\op R}({\op M},{\op P}^{\Q})_{\hat{\phi}}
\]
at this level, for any choice of base point $\phi\in\IBimod^h_{\op R}({\op M},{\op P})$,
where we again denote by $\hat{\phi}$ the composite of this morphism $\phi: {\op M}\rightarrow{\op P}$
with the rationalization map ${\op P}\rightarrow{\op P}^{\Q}$.
\item
In the previous assertions, if we assume in addition that the homotopy groups $N\pi_t{\op P}(s)$ associated to the infinitesimal bimodule ${\op P}$
are finitely generated in each degree $t$ and in each arity $s$,
and that the cohomology groups $H^i_{\Sigma_s}(L\Indec{\op M}(s),\pi(s))$ are finitely generated abelian groups in each degree~$i$,
for any choice of finitely generated $\Sigma_s$-module of coefficients $\pi(s)$,
then the rationalization map ${\op P}\rightarrow{\op P}^{\Q}$
also induces a finite-to-one map on the connected components
of our mapping spaces.
\end{enumerate}
\end{thm}

We checked in \cite[Lemma~10.7]{FTW} that the matching map ${\op P}(s)\rightarrow M{\op P}(s)$ is $(n-2)(s-1)$-connected in the case of the operad ${\op P} = \FM_n$. On the other hand, one has that $\IF_M$ is
cofibrant and therefore $L\Indec(\IF_M)=\Indec(\IF_M)$. In each arity, $\IF_M(s)$ is a manifold with corners and
\[
\Indec(\IF_M)(s)=\IF_M(s)/\partial\IF_M(s)=M_*^{\wedge s}/\Delta^sM_*.
\]
The right-most space is the quotient of the $s$-th smash power of $M_*$ by the fat diagonal $\Delta^sM_*$.
The space $M_*^{\wedge s}/\Delta^sM_*$ is a $\Sigma_s$-cofibrant pointed space of  dimension~$ms$. 
In case the codimension $n-m\geq 3$, we can take $m(s)=ms$ and one has that 
\[
n(s)-m(s)=(n-2)(s-1)-ms\to \infty.
\]
In case the codimension $n-m=2$, one has that $M\neq \R^m$ and therefore the handle dimension of $M_*$
is $\leq m-1$. The assignment $M_*\mapsto M_*^{\wedge s}/\Delta^sM_*$ is homotopy invariant and would 
give an equivalent pointed $\Sigma_s$-space if $M_*$ is replaced by a homotopy equivalent 
 pointed CW-complex $X_*$ of dimension $\leq m-1$. The dimension of 
$X_*^{\wedge s}/\Delta^sX_*$ is $m(s)=(m-1)s$ and therefore the convergency requirement
$n(s)-m(s)\to\infty$ is again satisfied. Finiteness property follows from the fact that any homotopy
group of the configuration space $C(s,\R^n)\simeq \FM_n(s)$, $n\geq 3$, is finitely generated (see again~\cite[Lemma~10.7]{FTW}\footnote{An
 expression of $\pi_*\FM_n(s)$ in terms of the homotopy groups of spheres is given in the proof of this reference.}) and that $M$  (and therefore $M_*^{\wedge s}/\Delta^sM_*$ as well)
is homotopy equivalent  to a finite CW complex.
%\todo{Determine the dimension bound $m(s)$ for the bimodule $\IF_M$.}

\subsection{Range improvement}\label{ss:range2}
In this subsection we explain why Theorem~\ref{thm:main1} and Corollary~\ref{cor_main}.a-c hold for more general manifolds $M^m$ as described in Section~\ref{ss:range1}. One has $M\subset M_*=M\cup\{*\}$ and the immersion $i\colon M_*\looparrowright S^m=\R^m\cup\{\infty\}$
satisfies $i^{-1}(\infty)=*$. The corresponding infinitesimal $\FM_m$-bimodule $\IF_M$ is defined as follows. Let $i\mathcal{O}(M)$ denote the category
of open subsets $U\subset M$, such that $U\cup\{*\}$ is open in $M_*$ and $i|_U$ is injective.  Define
\[
\IF_M:=\colim_{U\in i\mathcal{O}(M)}\IF_{i(U)}
\]
as a colimit in the category of infinitesimal $\FM_m$-bimodules.
 Note that as a space $\IF_M(k)$ is the Fulton--MacPherson--Axelrod--Singer
local compactification~\cite{Sinha} of the configuration space  of $k+1$ distinct labeled points in $M_*$, with the first being fixed to be~$*$.

The equivalence~\eqref{equ:AT} is proved in exactly the same way as \cite[Theorem~2.1 or Theorem~6.3]{Turchin4}. Propositions~\ref{prop:EMReedyCof},
\ref{prop:indEM} and Theorems~\ref{thm:main1}, \ref{thm:FM} are also proved by exactly the same arguments.

\section{Examples and applications}\label{s:examples}
In this section we provide some applications.
In general the cohomology of the graph complex $\HGC_{A,n}$ is not known.
However, one can perform low degree computations and make certain qualitative statements.

\subsection{Maurer--Cartan elements}\label{ss:MC}
The goal of this section is to understand the type of hairy  diagrams that appear in degree 0 and that can contribute to Maurer--Cartan (MC) elements.
Recall that a MC element of a complete $L_\infty$-algebra $\mathfrak{g}$ is, under our grading convention, a degree zero element $m\in\mathfrak{g}$
that satisfies the MC equation:
\[
dm+\frac 12[m,m]+\frac 1{3!}[m,m,m]+\ldots = \sum_{i=1}^\infty \frac 1{i!}\ell_i(\underbrace{m,\ldots,m}_i) =0.
\]
The $m$-twisted $L_\infty$ algebra $\mathfrak{g}^m$ has the differential
\[
d^m = d+[m,-]+\frac 12 [m,m,-]+\ldots =\sum_{i=0}^\infty\frac 1{i!} \ell_{i+1}(\underbrace{m,\ldots,m}_i,-)
\]
and (higher) brackets
\[
\ell_k^m(\underbrace{-,\ldots,-}_k)=\sum_{i=0}^\infty\frac 1{i!}\ell_{i+k}(\underbrace{m,\ldots,m}_i,\underbrace{-,\ldots,-}_k).
\]

Also recall that  $\MC_\bullet(\mathfrak{g})$ is the simplicial set $\MC(\mathfrak{g}\hat\otimes\Omega^*(\Delta^\bullet))$ of MC elements. In particular the gauge
relations are  1-simplices in this  set, or, in other words, MC elements of $\mathfrak{g}\hat\otimes\Omega^*(\Delta^1)$.

For a dg commutative algebra or homotopy commutative algebra $A$ we denote by $\dim(A)$, the maximal degree where $A$ is non-zero. 
%Note that if the homology groups of $A$ are trivial in degrees $>d$, we can always replace $A$ by its quasi-isomorphic  truncation $A_{\leq d}$ defined as $0$ in degree $j>d$, kernel of the differential $A_d\to A_{d+1}$ in degree~$d$ and as $A_j$ in degree $j<d$. However, it is an open question if a finite dimensional CW-complex (equivalently, a compact manifold) always admits a Sullivan model of finite type~\cite{FLST}. The main
%problem appears with spaces whose  first rational homology group $H_1$ is non-trivial.
 However, we do not know in general whether a space $M_*$ of the form considered in our main theorem admits a Sullivan model of finite dimension.
On the other hand,  under the assumptions of Theorem~\ref{thm:FM}, one can always choose a homotopy commutative algebra model $A$ of $M_*$
of \emph{algebraic codimension} $n-\dim(A)\geq 3$.
For example, we can take the augmentation ideal of the cohomology algebra $H^*(M_*)$
equipped with the transferred homotopy commutative algebra structure
that makes this object a model of the pointed space $M_*$
in the category of homotopy commutative algebras (see the account of Section~\ref{sec:homotopy_commutative_extension}).
In many relevant applications natural dg commutative algebra models of finite type and of that codimension exist also, and hence, we can take them instead, avoiding the use of homotopy commutative algebras.
%\todo{Reference for this result?}

\begin{prop}\label{prop:MC}
Let $A$ be a dg commutative algebra or homotopy commutative algebra of finite type with $\Delta:=n-\dim(A)-3\geq 0$, then the complex $\HGC_{A,n}$ is bounded below and is finite-dimensional
in every degree. Moreover, its part of non-positive degree $\left(\HGC_{A,n}\right)_{\leq 0}$ is spanned by trees with $\leq\frac{n-3}{1+\Delta}$ leaves.
Any such tree with $N$ leaves has all its leaves labeled by elements of $A$ of degree $\geq (N-1)(1+\Delta)+1$.
\end{prop}

\begin{proof}
The first statement and the fact that all hairy graphs of genus $\geq 1$ appear in strictly positive degrees are easy to check,
see \cite[Lemma~4.3]{STT} for the proof of a similar result.
For the second statement we note that the trees with $N$ leaves have the smallest degree when they are unitrivalent. Such trees
have $2N-3$ edges and $N-2$ internal vertices. The smallest possible degree of such a tree is
$(2N-3)(n-1)-(N-2)n-N(n-3-\Delta)$. Assuming that this degree is $\leq 0$ implies $N\leq \frac{n-3}{1+\Delta}$.

If a unitrivalent tree of degree $\leq 0$ has one leaf labeled by an element of degree $I$ and all the other $N-1$  leaves  labeled
by elements of the maximal possible degree $n-3-\Delta$, one has $(2N-3)(n-1)-(N-2)n-(N-1)(n-3-\Delta)-I\leq 0$. The
latter inequality implies $I\geq(N-1)(1+\Delta)+1$.
\end{proof}

\begin{cor}\label{cor:MC}
Under the assumptions of Proposition~\ref{prop:MC}, we have the relations
\begin{align*}
\left.\MC\left(\HGC_{A,n}\right)\right/{\sim} & = \left.\MC\left(\HGC_{A_{>0},n}\right)\right/{\sim}\\
\text{and}\quad H_0\left(\HGC_{A,n}\right) & = H_0\left(\HGC_{A_{>0},n}\right).
\end{align*}
\end{cor}
\noindent By $A_{>0}$ we understand the naive truncation $(A_{>0})_i=\begin{cases}A_i,&i>0,\\0,&i=0.\end{cases}$

\begin{proof}
By Proposition~\ref{prop:MC}, only linear combinations of trees can produce a MC element of $\HGC_{A,n}$. Moreover all its $A$ labels must be of degree
$\geq (N-1)(1+\Delta)+1\geq (2-1)(1+0)+1 = 2$. Similarly  only trees with $A$  labels of degree $\geq 1$ can contribute
to the gauge relations.
For $H_0$ the argument is the same.%\footnote{In case $A$ has zero differential one can conclude $\left.\MC\left(\HGC_{A,n}\right)\right/{\sim} = \left.\MC\left(\HGC_{A_{>1},n}\right)\right/{\sim}.$}
\end{proof}

\subsection{MC elements in terms of unitrivalent trees}\label{ss:utt}
In this subsection we explain how MC elements can be encoded in terms of unitrivalent trees modulo $IHX$ relations.
We hope that this description  will be easier to relate  to the known higher dimensional knot invariants,
see sections~\ref{ss:strings}-\ref{ss:gauge}.

For any dg commutative algebra $A$ and an integer $n\geq 3$, consider the space $\UTT_{A,n}$ defined as follows. It is spanned by the \emph{unitrivalent trees} of
$\HGC_{A,n}$ (which means trees, whose all internal vertices are trivalent). In addition to the orientation relation as in $\HGC_{A,n}$, we quotient this space
by $IHX$ relations --- images of $\delta_{split}$ of trees whose all internal vertices are trivalent except one, which is four-valent. Note that $\UTT_{A,n}$ is
a quotient-complex of $\HGC_{A,n}$, moreover  $\UTT_{A,n}$ inherits from $\HGC_{A,n}$ a dg Lie algebra structure by means of this quotient map.
\begin{prop}\label{prop:utt}
Assuming $n-\dim(A)\geq 3$, the quotient map $\HGC_{A,n}\to\UTT_{A,n}$ induces a bijection on the sets of gauge equivalence classes of MC elements.
\end{prop}
\begin{proof}
Denote by $\TGC_{A,n}$ the  $L_\infty$-algebra  spanned by all trees of $\HGC_{A,n}$.  It is naturally a quotient $L_\infty$-algebra
of $\HGC_{A,n}$.
It follows from Proposition~\ref{prop:MC}, that $\TGC_{A,n}$  has the same set of Maurer--Cartan
elements and gauge relations as $\HGC_{A,n}$. On the other hand, both $\TGC_{A,n}$ and $\UTT_{A,n}=\TGC_{A,n}/IHX$ admit a complete filtration by the number of hairs
minus one. The quotient map $\TGC_{A,n}\to\HGC_{A,n}$ induces  surjective quasi-isomorphisms of associated graded quotients $\gr_i\TGC_{A,n}\to\gr_i\UTT_{A,n}$, $i\geq 1$. We use here the fact that the cyclic $L_\infty$ operad is quasi-isomorphic to the cyclic Lie operad.
By the same argument as in the proof of Proposition~\ref{prop:graphscofibrant}, we get an equivalence of simplicial sets
$\MC_\bullet(\TGC_{A,n})\simeq \MC_\bullet(\UTT_{A,n})$. In particular, we get a bijection on the sets of connected components.
\end{proof}

\begin{rem}\label{r:UTT}
The discussion and result apply mutatis mutandis to the case of a homotopy commutative algebra $A$, except that in this case $\UTT_{A,n}$ is generally not a dg Lie algebra, but inherits an $L_\infty$-structure from $\HGC_{A,n}$. It is easy to see that the induced $L_\infty$-algebra structure of $\UTT_{A,n}$ is determined by the
$C_\infty$-algebra structure of~$A$  restricted  along the operad inclusion $C_\infty \to\Omega L_\infty^c$.
\end{rem}

\subsection{Finiteness}\label{ss:finite}
One important question is when the set of isotopy classes of embeddings is finite. Thanks to Corollary~\ref{cor_main}.b we are able to conclude
that the set $\pi_0\Embbar_\partial(M,\R^n)$ is finite if the corresponding set $\MC/{\sim}$ of gauge equivalence classes of Maurer--Cartan elements is finite.

\begin{ex}\label{ex:finite}
$\pi_0\Embbar(S^3\times S^3,\R^{11})$ is finite. In particular there are finitely many isotopy classes of embeddings $S^3\times S^3\hookrightarrow
\R^{11}$ which are trivial as immersions.\footnote{In fact, by the Haefliger-Zeeman unknotting theorem (see~\cite{Haefliger0,Zeeman}, \cite[Theorem~2.6.b]{Skopenkov08}), the space $\Emb(S^3\times S^3,\R^{11})$ is connected.}
\end{ex}

Indeed, first we notice that $S^3\times S^3$ is formal. On the other hand, according to Proposition~\ref{prop:MC}, degree zero diagrams in
$\HGC_{H^*(S^3)\otimes H^*(S^3),11}$ are trees with $\leq \frac{11-3}{1+2} = 2+\frac{2}{3}$ leaves. But any such tree
$\begin{tikzpicture}[baseline=-.65ex]
\node (v) at (0,0) {$\alpha$};
\node (w) at (1,0) {$\beta$};
\draw (v) edge (w);
\end{tikzpicture}$
has never degree zero. Thus the set $\MC\left(\HGC_{H^*(S^3)\otimes H^*(S^3),11}\right)$ has only one element~0.

\subsection{$M$-unknots}\label{ss:unknot}
The hairy graph complex  $\HGC_{\bar R,n}$ depends only on the Sullivan model $R$ of $M_*$.
The following is an immediate consequence of Corollary~\ref{cor_main}.c-d.

\begin{cor}\label{cor:unknot}
Let $i\colon M\subset\R^m\subset\R^n$ and $i'\colon M'\subset\R^{m'}\subset\R^n$ satisfy the assumptions of Theorem~\ref{thm:FM}.
If the models $\Omega^*(M_*)\simeq\Omega^*(M'_*)$ are quasi-isomorphic as augmented dg commutative algebras, then the components of the trivial embedding are rationally equivalent:
$
\Embbar_\partial(M,\R^n)_i\simeq_\Q \Embbar_\partial(M',\R^n)_{i'}.
$
\end{cor}

\begin{ex}\label{ex:unknot}
It is easy to check that $\Embbar(\R P^2,\R^6)$ is connected.\footnote{This is because $\Emb(\R P^2,\R^6)$ is connected and $\Imm(\R P^2,\R^6)$
is simply-connected.} Since $\R P^2$ is embeddable in $\R^4$ and $\Omega^*(\R P^2)\simeq\Q$, one gets $\Embbar(\R P^2,\R^6)\simeq_\Q *$. More generally,
$\Embbar(\coprod_r\R P^2,\R^6)\simeq_\Q C(r,\R^6)\simeq\FM_6(r)$.
\end{ex}

%\subsection{Case when $M_*$ is formal}\label{ss:formal}
%In all the examples that we consider in this paper  manifolds $M\cup\{\infty\}=M_*$ are rationally formal. It turns out
%that in this situation the component of the unknot has the biggest ranks of the rational homotopy groups.

\begin{prop}\label{prop:formal}
Under the assumptions of Theorem~\ref{thm:FM}, the rational homotopy groups of any component  $\Embbar_\partial(M,\R^n)_\psi$ are subquotients
of those of the component of the trivial embedding $\Embbar_\partial(M,\R^n)_i$.
\end{prop}

\begin{proof}
Let $m$ be a Maurer--Cartan element of $\HGC_{\bar R,n}$ corresponding to $\psi$.
In the complex $\HGC_{\bar R,n}^m$ consider the filtration by the genus of the graphs plus the number of hairs.
The differential $d_0$ in the associated spectral sequence is the non-deformed differential of $\HGC_{\bar R,n}$.
% Indeed,
%this is because $m$ is a linear combination of diagrams with all hair labels of strictly positive degrees, see
%Proposition~\ref{prop:MC}.
  Thus the $E^1$ term in positive degrees is isomorphic to the positive degree
homology of $\HGC_{\bar R,n}$ which  exactly consists of the rational homotopy groups of $\Embbar_\partial(M,\R^n)_i$.
\end{proof}

In Section~\ref{sss:spherical_hopf} we show that the component of the trivial embedding and that of the Hopf link of
$\Embbar(S^{m_1}\coprod S^{m_2},\R^{m_1+m_2+1})$ have different rational homotopy groups.

\subsection{String links}\label{ss:strings}
In this subsection we apply our results to the spaces  $\Embbar_\p\left(\coprod_{i=1}^r\R^{m_i},
\R^n\right)$, $n-\max(m_i)\geq 3$, of string links modulo immersions which were defined in the introduction.
%
%Then the space of embeddings $\Embbar_\p(M,\R^n)$ is equivalent to the space of string links (modulo immersions), with strings of dimensions $d_1,\dots,d_k$, such as studied by the second author and Songhafouo-Tsopm\'en\'e \cite{STT}
In this case the corresponding manifold $M_*=M\cup\{ \infty\}$ deformation retracts to a wedge of spheres $\vee_{i=1}^rS^{m_i}$.
We can hence take as a Sullivan model
\[
R=\Q[\omega_1,\dots,\omega_r]/(\omega_i\omega_j)_{i,j=1,\dots,r},
\]
where $\omega_i$ is a generator of degree $m_i$, and all products of generators vanish.
In particular $\bar R$ is an $r$-dimensional graded vector space spanned by $\omega_1,\dots,\omega_r$.
It follows that $\HGC_{\bar R,n}$ is a complex of hairy graphs with hairs coming in $r$ colors, corresonding to the possible decorations of the hairs by the $\omega_i$. This graph complex was denoted by $\HGC_{m_1\ldots m_r;n}$ in~\cite{STT}.

In this case the $L_\infty$-structure is fairly simple: The differential is just $\delta=\delta_{split}$ (cf. \eqref{equ:delta}), since $d_R=0$ and $\delta_{join}=0$ by the vanishing of all products in $\bar R$.
By the same reason the $L_\infty$-structure is trivial.
This reflects the fact that the space of string links is in fact a loop space~\cite{Duc}.

By our computations we can then conclude that for $n-\max(m_i)\geq 3$  and $\psi\in  \Embbar_\p\left(\coprod_{i=1}^r\R^{m_i},
\R^n\right)$ we have that
\[
 \Embbar_\p\left(\sqcup_{i=1}^r\R^{m_i},
\R^n\right)_\psi
\simeq_\Q
\MC_{\bullet}(\HGC_{\bar R,n})_{>0}^\alpha,
\]
where $\alpha$ is a Maurer--Cartan element corresponding to (the connected component of) $\psi$.
In particular this means that the rational homotopy groups may be computed as
\[
\pi_k(\Embbar_\p(M,\R^n),\psi)\otimes \Q
\cong
H_k(\HGC_{\bar R,n}^\alpha)
=
H_k(\HGC_{\bar R,n})
\]
for $k\geq 1$. For the last identity we used that the $L_\infty$-structure is trivial and hence the twist does not alter our complex.
(Note also that all components of the space of string links have the same homotopy type as the space in question is a loop space.)
The statement is  true for $k=1$ since by the loop space structure, $\pi_1$ is abelian. This result has been proved for
$n\geq 2\max(m_i)+2$ in~\cite{STT} and was conjectured to hold in codimension greater than or equal to three \cite[Conjecture~3.1]{STT}.
%In fact, the statement is also true for $k=0$ by \cite[Theorem 3.2]{STT}, and hence \cite[Conjecture 3.1]{STT} follows.
Similarly, we obtain that
\[
H_*(\Embbar_\p(M,\R^n)_\psi,\Q)
\cong
S\bigl(H_{>0}(\HGC_{\bar R,n})\bigr).
\]
%where $\HGC_{\bar R,n}^{\leq 0}$ is the truncation of $\HGC_{\bar R,n}$.
In other words, the homology is just the symmetric product of the homotopy, which again could be deduced by the fact the space of string links is a loop space.

Since the $L_\infty$ structure is abelian, the Maurer--Cartan equation and gauge action are abelian as well and we get
\beq{equ:MC-H0}
\MC(\HGC_{\bar R,n})/{\sim} = H_0( \HGC_{\bar R,n}).
\eeq
On the other hand,  using Haefliger's work on the isotopy classes of spherical links~\cite{Haefliger1,Haefliger2.5} together with  \cite[Corollary~20]{FTW} of
the authors' work on the rational homotopy type of the delooping of $\Embbar_\p(\R^m,\R^n)$, $n-m\geq 3$, Songhafouo-Tsopm\'en\'e and the second author proved
\cite[Theorem~4.2]{STT} that $\pi_0 \Embbar_\p\left(\coprod_{i=1}^r\R^{m_i},\R^n\right)$ with respect to the concatenation of links is a finitely generated abelian group isomorphic to $H_0(\HGC_{\bar R,n})$ when tensored with $\Q$:
\beq{equ:pi0}
\Q\otimes\pi_0 \Embbar_\p\left(\sqcup_{i=1}^r\R^{m_i},\R^n\right)\simeq
H_0(\HGC_{\bar R,n}).
\eeq
 Note also that by Proposition~\ref{prop:MC} only trees can contribute to the non-positive degrees. Using this and the fact
that the cyclic $L_\infty$ operad is quasi-isomorphic to the cyclic $Lie$ operad, the space $H_0(\HGC_{\bar R,n})$ can be described as the space
spanned by unitrivalent trees of degree zero with leaves labeled by $1,\ldots,r$, and quotiented out by $IHX$ relations. This fact is quite remarkable as
it relates the study of  spaces of higher dimensional links to the theory of Vassiliev invariants of classical knots and links in~$\R^3$, where the spaces of unitrivalent diagrams modulo $IHX$
relations naturally appear~\cite{BarNatan1,BarNatan2}. For more general embedding spaces, the MC element as a knot invariant can also be formulated purely in terms of unitrivalent trees modulo $IHX$, see Proposition~\ref{prop:utt}.

\subsection{Spherical links}\label{ss:spherical}
A.~Haefliger proved that $\pi_0\Emb\left(\coprod_{i=1}^r S^{m_i},\R^n\right)$, $n>\max(m_i)+2$, is always a finitely generated abelian group~\cite{Haefliger1,Haefliger2,Haefliger2.5}. The product on this set is defined as follows. Given two links $f$ and $g$, we place them in two disjoint balls. We then  connect each $f(S^{m_i})$ with $g(S^{m_i})$ by a thin tube along any path. Note
that because of the codimension condition the complement is simply connected and the choice of the paths does not matter for the resulting isotopy class of the link. The product
is obviously commutative and associative. There are inverses because in codimension $> 2$ the isotopy classes of links coincide with concordance classes.

This product can be easily extended to $\pi_0\Embbar\left(\coprod_{i=1}^r S^{m_i},\R^n\right)$, $n>\max(m_i)+2$.

\begin{prop}\label{pro:pi0}
For $n>\max(m_i)+2$, $$\pi_0\Embbar\left(\sqcup_{i=1}^r S^{m_i},\R^n\right) = \pi_0 \Embbar_\p\left(\sqcup_{i=1}^r\R^{m_i},\R^n\right).$$
\end{prop}

\begin{proof}
The surjective projection
\[
\pi_0\Embbar\left(\sqcup_{i=1}^r S^{m_i},\R^n\right) \to \bigoplus_{i=1}^r \pi_0
\Embbar\left( S^{m_i},\R^n\right)
\]
splits. Denote its kernel by $\pi_0^U\Embbar\left(\coprod_{i=1}^r S^{m_i},\R^n\right)$. The
upper script $U$ stands for \lq\lq{}unknots\rq\rq{}  --- every component of such links is unknotted.
Similarly define the groups $\pi_0^U\Emb\left(\coprod_{i=1}^r S^{m_i},\R^n\right)$,
$\pi_0^U\Embbar_\p\left(\coprod_{i=1}^r \R^{m_i},\R^n\right)$, $\pi_0^U\Emb_\p\left(\coprod_{i=1}^r \R^{m_i},\R^n\right)$. One can easily see that all these four groups are isomorphic:
\begin{multline*}
\pi_0^U\Embbar\left(\sqcup_{i=1}^r S^{m_i},\R^n\right)=
\pi_0^U\Emb\left(\sqcup_{i=1}^r S^{m_i},\R^n\right)=\\
\pi_0^U\Emb_\p\left(\sqcup_{i=1}^r \R^{m_i},\R^n\right)=
\pi_0^U\Embbar_\p\left(\sqcup_{i=1}^r \R^{m_i},\R^n\right),
\end{multline*}
see~\cite[Lemma~4.7]{STT} or~\cite[Section~2.6]{Haefliger2.5} for a similar statement. On the other hand, for one component embedding spaces, the isomorphism
\[
\pi_0\Embbar\left( S^{m_i},\R^n\right) = \pi_0 \Embbar_\p\left(\R^{m_i},\R^n\right)
\]
is proved in~\cite[Theorem~1.1]{FTW2} (see also Section~\ref{ss:emb_long}).
\end{proof}

In fact one also has a bijection on the corresponding sets $\MC/{\sim}$ of Maurer--Cartan elements. In this case $M_*$ is homotopy equivalent to $\left(\coprod_{i=1}^r
S^{m_i}\right)\coprod\{*\}$. Hence,
\beq{equ:cohom}
\bar R =H^*\left(\sqcup_{i=1}^r S^{m_i}\right)=\oplus_{i=1}^r H^*(S^{m_i}).
\eeq
Let $1_i\in H^0(S^{m_i})$ and $\omega_i\in H^{m_i}(S^{m_i})$ denote the generating cohomology classes of the summand $S^{m_i}$.

The corresponding hairy graph complex is spanned by graphs with hairs of $2r$ possible labels: $1_1$, $\omega_1$, $\ldots$, $1_r$, $\omega_r$.
The relations are $\omega_i\omega_j=0$, $1_i1_j=\delta_{ij}1_j$, $1_i\omega_j=\delta_{ij}\omega_j$, $1\leq i,j\leq r$.

It follows from Corollary~\ref{cor:MC}, Proposition~\ref{pro:pi0} and equations~\eqref{equ:MC-H0} and~\eqref{equ:pi0}, that in the range $n>\max(m_i)+2$,
\begin{multline*}
\Q\otimes\pi_0\Embbar\left(\sqcup_{i=1}^r S^{m_i},\R^n\right)=
\Q\otimes\pi_0\Embbar_\p\left(\sqcup_{i=1}^r \R^{m_i},\R^n\right)=\\
H_0\left(\HGC_{H^*(\sqcup_i S^{m_i}),n}\right) = H_0\left(\HGC_{\tilde H^*(\vee_i S^{m_i}),n}\right)=\\
\left.\MC\left(\HGC_{H^*(\sqcup_i S^{m_i}),n}\right)\right/{\sim} = \left.\MC\left(\HGC_{\tilde H^*(\vee_i S^{m_i}),n}\right)\right/
{\sim}.
\end{multline*}

\subsubsection{Example of components of different homotopy type}\label{sss:spherical_hopf}
Unlike in the case of string links, the rational homotopy type of different components of $\Embbar\left(\sqcup_{i=1}^r S^{m_i},\R^n\right)$ can be different.
For example, we consider the space $\Embbar\left( S^{m_1}\sqcup S^{m_2},\R^{m_1+m_2+1}\right)$.
In the corresponding non-deformed hairy graph complex, the graph
\[
m=
\begin{tikzpicture}[baseline=-.65ex]
\node (v) at (0,0) {$\omega_1$};
\node (w) at (1.5,0) {$\omega_2$};
\draw (v) edge (w);
\end{tikzpicture}
\]
is a cycle of degree zero. We expect that, in the Maurer--Cartan element assigned to a link $i: S^{m_1}\sqcup S^{m_2}\hookrightarrow\R^{m_1+m_2+1}$,
the coefficient of this graph $m$ is given by the linking number of the components of the link.
In particular, the Maurer--Cartan element corresponding to the Hopf link is $m$.
Now if we look at the graph
\[
S=
\begin{tikzpicture}[baseline=-.65ex]
\node (v) at (0,0) {$1_1$};
\node (w) at (1.5,0) {$1_2$};
\draw (v) edge (w);
\end{tikzpicture},
\]
it is a non-trivial cycle (of degree $m_1+m_2$) in the non-deformed complex. But it is no more a cycle for the $m$-twisted graph complex:
\[
d^m(S)=[m,S]=
\begin{tikzpicture}[baseline=0]
\node (v1) at (-.6,0) {$\omega_1$};
\node (v2) at (0,0) {$\omega_2$};
\node (v3) at (.6,0) {$1_1$};
\node [int] (w) at (0,.8) {};
\draw (w) edge (v1) edge (v2) edge (v3);
\end{tikzpicture}
+
\begin{tikzpicture}[baseline=0]
\node (v1) at (-.6,0) {$\omega_1$};
\node (v2) at (0,0) {$\omega_2$};
\node (v3) at (.6,0) {$1_2$};
\node [int] (w) at (0,.8) {};
\draw (w) edge (v1) edge (v2) edge (v3);
\end{tikzpicture}.
\]

Topologically this corresponds to the fact that the fibration
\beq{equ:hopf}
\Embbar\left( S^{m_1}\sqcup S^{m_2},\R^{m_1+m_2+1}\right)\to S^{m_1+m_2}
\eeq
that assigns to a link the direction of the vector between the images of the basepoints of the spheres, has a section for the component of the trivial link, and does not
admit a section for the Hopf link. In the former case the homotopy groups of $S^{m_1+m_2}$ appear as direct summands
of those of  the link space, while for the latter case this is not true even rationally.

\subsection{Spherical and long embeddings}\label{ss:emb_long}%\todo{This section has to be withdrawn and postponed to another paper.}
The approach with graph complexes that we develop in this paper can also be used to compare the homotopy type of different embedding spaces.
The first non-trivial  question is how the homotopy type of $\Embbar(S^m,\R^n)$ is related to  that of $\Embbar_\p(\R^m,\R^n)$.
In~\cite{FTW2} we study this problem. Note that  \cite[Section 4]{BudneyCohen} gives a connection between
 the embedding spaces $\Emb(S^m,\R^n)$ and the long embedding spaces $\Emb_\p(\R^m,\R^n)$. It turns out though that working modulo immersions
and assuming $n-m>2$ makes this connection more straightforward. By the authors' \cite[Theorem~1.1]{FTW2}, one has a fiber sequence
\beq{eq:fiber_seq}
\Embbar(S^m,\R^n)\to S^{n-m-1}\to B\Embbar_\p(\R^m,\R^n),
\eeq
$n-m>2$, where $B$ denotes the classifying space functor.

On the level of graph-complexes this fiber sequence can be seen as follows.
The graph complexes corresponding to $\Embbar_\p(\R^m,\R^n)$ and $\Embbar(S^m,\R^n)$ are $\HGC_{\bar R,n}$ and $\HGC_{R,n}$,
respectively, where
\[
R = H^*(S^m)=\Q[\omega]/\omega^2=\Q1\oplus\Q\omega.
\]
Explicitly, elements of $\HGC_{R,n}$ are hairy graphs with hairs being either decorated by $1$ or $\omega$,
while elements of $\HGC_{\bar R,n}$ are hairy graphs all of whose hairs are decorated by $\omega$.
The graphs
\begin{align*}
L&=
\begin{tikzpicture}[baseline=-.65ex]
\node (v) at (0,0) {$1$};
\node (w) at (1,0) {$\omega$};
\draw (v) edge (w);
\end{tikzpicture}
&
T&=
\begin{tikzpicture}[baseline=0]
\node (v1) at (-.6,0) {$1$};
\node (v2) at (0,0) {$\omega$};
\node (v3) at (.6,0) {$\omega$};
\node [int] (w) at (0,.6) {};
\draw (w) edge (v1) edge (v2) edge (v3);
\end{tikzpicture}
\end{align*}
are the ones that correspond to the rational homotopy of $S^{n-m-1}$.
The graph $L$ is always non-zero and $T$ is nonzero if and only if $n-m$ is odd.
Let $U$ be a subspace in $\HGC_{R,n}$ spanned by $L$ and $T$.
One has that the (graded vector space) direct sum $U\oplus\HGC_{\bar R,n}$ is a subcomplex in $\HGC_{R,n}$.
The authors proved (see \cite[Theorem~2.1]{FTW2}) that the inclusion
\[
U\oplus\HGC_{\bar R, n}\subset \HGC_{R,n}
\]
is a quasi-isomorphism.

The fiber sequence~\eqref{eq:fiber_seq} implies that $\pi_0\Embbar(S^m,\R^n)=\pi_0\Embbar_\p(\R^m,\R^n)$ and that all components of
$\Embbar(S^m,\R^n)$ have the same homotopy type. The latter fact  can also be seen from the observation that $\Embbar_\p(\R^m,\R^n)$
acts on $\Embbar(S^m,\R^n)$ for which $\pi_0\Embbar(S^m,\R^n)$  is a torsor of $\pi_0\Embbar_\p(\R^m,\R^n)$.

The group $\pi_0\Embbar(S^m,\R^n)=\pi_0\Embbar_\p(\R^m,\R^n)$, $n-m>2$, is a finitely generated abelian group usually torsion except two cases
\[
\text{(a) }n=4k-1,\,\,\,\, m=2k-1,\,\,\,\, k\geq 2;\quad\quad
\text{(b) } n=6k,\,\,\,\, m=4k-1,\,\,\,\, k\geq 1.
\]
%\begin{enumerate}[label=(\alph*)]
%\item $n=4k-1$, $m=2k-1$, $k\geq 2$;
%\item $n=6k$, $m=4k-1$, $k\geq 1$.
%\end{enumerate}
In the latter two cases this group is infinite and has rank one~\cite[Corollary~20]{FTW}.\footnote{This fact can also be easily deduced from Haefliger's~\cite[Corollary~6.7 and Remark~6.8]{Haefliger2}.}  The corresponding Maurer--Cartan elements (generators of the groups $H_0(\HGC_{\bar R,n})$) are, respectively,
\begin{align*}
L_\omega&=
\begin{tikzpicture}[baseline=-.65ex]
\node (v) at (0,0) {$\omega$};
\node (w) at (1,0) {$\omega$};
\draw (v) edge (w);
\end{tikzpicture}
&
&\text{and}&
T_\omega&=
\begin{tikzpicture}[baseline=0]
\node (v1) at (-.6,0) {$\omega$};
\node (v2) at (0,0) {$\omega$};
\node (v3) at (.6,0) {$\omega$};
\node [int] (w) at (0,.6) {};
\draw (w) edge (v1) edge (v2) edge (v3);
\end{tikzpicture}.
\end{align*}
Geometrically, the class $L_\omega$ appears as image, under inclusion
\[
\Omega^{2k}V_{2k-1}(\R^{4k-1})\to \Embbar_\p(\R^{2k-1},\R^{4k-1}),
\]
of the $\SO(2k)$ Euler class in the rational homotopy of the Stiefel manifold $V_{2k-1}(\R^{4k-1})=\SO(4k{-}1)/\SO(2k)$. In case (b), the  MC element $T_\omega$ corresponds to the Haefliger trefoil $S^{4k-1}\hookrightarrow\R^{6k}$~\cite{Haefliger1}.

The twisting by $L_\omega$ changes neither the differential, nor any (higher) bracket of $\HGC_{H^*(S^{2k-1}),4k-1}$. This is because
for even codimension $n-m$, any graph with  two $\omega$-hairs attached to an internal vertex is zero.  The twisting by $T_\omega$ changes both the differential
and the $L_\infty$ structure of $\HGC_{H^*(S^{4k-1}),6k}$. One can show that $\HGC_{H^*(S^{4k-1}),6k}^{T_\omega}$
is $L_\infty$ isomorphic to the non-deformed one $\HGC_{H^*(S^{4k-1}),6k}$.

In \cite{FTW2}, we also determine the rational homotopy type of $\Embbar(S^m,\R^n)$, $n-m>2$ (see~\cite[Corollary~1.3]{FTW2}). We prove
that each component of $\Embbar(S^m,\R^n)$ is rationally a product of $K(\Q,j)$'s in all cases except when $n$ is odd and $m$ is even. In the latter case
 $\Embbar(S^m,\R^n)\simeq_\Q \Embbar_\p(\R^m,\R^n)\times S^{n-m-1}$ and the failure of not being rationally abelian is only in the factor $S^{n-m-1}$.

\subsection{Non-linear Maurer--Cartan equation}\label{eq:MC_non_lin}
In all the examples that we considered so far the Maurer--Cartan equation for a hairy graph complex $\HGC_{A, n}$
was reduced to a linear equation and the set $\MC(\HGC_{A, n})/{\sim}$ is identified with $H_0(\HGC_{A, n})$.
However, it is not generally the case. Consider, for example, $\Embbar(S^2\times S^2,\R^7)$. Since $S^2\times S^2$
is embeddable in $\R^5$, our approach can be applied. One has
\[
R = \left.\Q[\omega_1,\omega_2]\right/\omega_1^2{=}\omega_2^2{=}0,
\]
where $|\omega_1|=|\omega_2|=2$.
In degree zero $\HGC_{R,n}$ has two graphs
\[
L_1=
\begin{tikzpicture}[baseline=-.65ex]
\node (v) at (0,0) {$\omega_1$};
\node (w) at (2,0) {$\omega_1\wedge\omega_2$};
\draw (v) edge (w);
\end{tikzpicture};
\qquad
L_2=
\begin{tikzpicture}[baseline=-.65ex]
\node (v) at (0,0) {$\omega_2$};
\node (w) at (2,0) {$\omega_1\wedge\omega_2$};
\draw (v) edge (w);
\end{tikzpicture}.
\]
(In fact there is one more graph of degree zero: the $H$-shaped one with all its four hairs labelled by $\omega_1\wedge\omega_2$,
but it can be killed by gauge transformations as it is the boundary of the $X$-shaped graph again with all its four hairs labelled
by  $\omega_1\wedge\omega_2$. Compare also with Proposition~\ref{prop:utt}.)
If $m=\lambda_1 L_1+\lambda_2 L_2$, $\lambda_1,\lambda_2\in\Q$, then the Maurer--Cartan equation becomes
\[
0=\frac 12[m,m]=\lambda_1\lambda_2 \times
\begin{tikzpicture}[baseline=5]
\node (v1) at (-2,0) {$\omega_1\wedge\omega_2$};
\node (v2) at (0,0) {$\omega_1\wedge\omega_2$};
\node (v3) at (2,0) {$\omega_1\wedge\omega_2$};
\node [int] (w) at (0,1.5) {};
\draw (w) edge (v1) edge (v2) edge (v3);
\end{tikzpicture}.
\]
Thus, the set
\[
\MC/{\sim}=\{\lambda_1 L_1+\lambda_2 L_2\,|\, \lambda_1=0 \text{ or } \lambda_2=0\}.
\]
%This gives some insight on the set $\pi_0\Embbar(S^2\times S^2,\R^7)$.
We conjecture that in this case,  the  MC element is related to the Bo\'echat-Haefliger
invariant~\cite{Boechat,BoeHae}
\[
BH\colon \pi_0\Emb(S^2\times S^2,\R^7)\to H^2(S^2\times S^2,\Z)=\Z^2,
\]
which was shown in~\cite{BoeHae} to have as image $2\Z\times 0\bigcup 0\times 2\Z$. Up to an action of the torsion $\pi_0\Emb(S^4,\R^7)=\Z_{12}$
(that preserves $BH$),
this  invariant   determines the isotopy class of embedding~\cite{CrowSkop,Skopenkov10}.

\subsection{Non-linear gauge relation}\label{ss:gauge}
Even if the MC equation happens to be linear and the MC elements are just zero-cycles, the gauge action can still be non-linear. As an example
consider $\Embbar(S^1\times S^2,\R^6)$. Since the manifold $S^1\times S^2$ is formal, we take
\[
R = H^*(S^1\times S^2)=\Q[\alpha,\beta]\bigr/\beta^2,
\]
where $|\alpha|=1$, $|\beta|=2$. Consider the following diagrams in $\HGC_{R,6}$:
\[
{{\raisebox{2.5ex}{$L_\alpha=$}
\begin{tikzpicture}[baseline=-3.0ex]
\node (v) at (0,0) {$\alpha$};
\node (w) at (2,0) {$\alpha\wedge\beta$};
\draw (v) edge (w);
\end{tikzpicture}}
\atop
{\raisebox{1ex}{$L_\beta=$}
\begin{tikzpicture}[baseline=-1.5ex]
\node (v) at (0,0) {$\beta$};
\node (w) at (2,0) {$\alpha\wedge\beta$};
\draw (v) edge (w);
\end{tikzpicture}}}
\quad\quad\quad\quad
T_{\alpha\wedge\beta}=
\begin{tikzpicture}[baseline=5]
\node (v1) at (-2,0) {$\alpha\wedge\beta$};
\node (v2) at (0,0) {$\alpha\wedge\beta$};
\node (v3) at (2,0) {$\alpha\wedge\beta$};
\node [int] (w) at (0,1.5) {};
\draw (w) edge (v1) edge (v2) edge (v3);
\end{tikzpicture}
\]

By Proposition~\ref{prop:MC}, $L_\beta$ and $T_{\alpha\wedge\beta}$ are the only graphs in $\HGC_{R,6}$ of degree zero. Any MC element has the form
$\lambda L_\beta+\mu T_{\alpha\wedge\beta}$, $\lambda,\mu\in\Q$. We claim that the gauge relation is
\beq{eq:gauge}
\lambda L_\beta+\mu T_{\alpha\wedge\beta}\,\,\sim\,\,\lambda L_\beta\quad \text{ if $\lambda\neq 0$.}
\eeq
Indeed, the following
\[
m=\lambda L_\beta+\mu t T_{\alpha\wedge\beta} -\frac{\mu}{\lambda}dt L_\alpha\in \MC\left(\HGC_{R,6}\otimes\Omega^*(\Delta^1)\right)
\]
is a MC element (note that $[L_\alpha,L_\beta]=T_{\alpha\wedge\beta}$) and $m|_{t=0}=\lambda L_\beta$, $m|_{t=1}=\lambda L_\beta+\mu T_{\alpha\wedge\beta}$. It is easy to see that all the other gauge equivalences are generated by this one. We conclude
\[
\MC/{\sim}=\{ \lambda L_\beta+\mu T_{\alpha\wedge\beta}\,|\, \mu=0 \text{ or } \lambda=0\}.
\]

We conjecture that the coefficient $\lambda$ in the MC element corresponds to the Whitney invariant~\cite[Section~1]{Skopenkov08.5}
\[
W\colon \pi_0\Emb(S^1\times S^2,\R^6)\to H_1(S^1\times S^2,\Z)\simeq H^2(S^1\times S^2,\Z)\simeq \Z,
\]
while the coefficient $\mu$ corresponds to the action of $\pi_0\Emb(S^3,\R^6)=\Z\simeq_{\Q}\pi_0\Embbar(S^3,\R^6)$ (the generators are Haefliger's trefoil in $\pi_0\Emb(S^3,\R^6)=\Z$~\cite{Haefliger1}, and the tripod $T_\omega\in H_0(\HGC_{H^*(S^3),6})=\Q$, see Section~\ref{ss:emb_long}).
It was shown by A.~Skopenkov in~\cite{Skopenkov08.5} that
\begin{itemize}
\item for any integer~$i$ there exists an embedding $S^1\times S^2\hookrightarrow \R^6$
of Whitney invariant~$i$;
\item two isotopy classes of such embeddings have the same Whitney number if and only if they are in the same orbit of the $\pi_0\Emb(S^3,\R^6)$-action;
\item all $\pi_0\Emb(S^3,\R^6)$-orbits in $\pi_0\Emb(S^1\times S^2,\R^6)$ are finite except the one of Whitney invariant~0.
\end{itemize}

In fact in~\cite{Skopenkov08.5} A. Skopenkov completely classifies isotopy classes of smooth embeddings of any closed connected orientable 3-manifold in~$\R^6$
in terms of the Whitney invariant and the $\pi_0\Emb(S^3,\R^6)$-action. On the other hand any compact orientable 3-manifold is parallelizable~\cite{MilnorStasheff}
and therefore is immersible in~$\R^4$. According to Section~\ref{ss:range1} our approach can be applied. In fact, the above computation stays exactly the same, except that for a compact connected orientable $M$ the class $L_\alpha$ above gets replaced by classes $L_{\alpha_i}$ labelled by a basis $\{\alpha_i\}_i$ of $H^1(M)$. This can be compared to and is precisely in agreement with Skopenkov's results.
Note that in the general case we can still take the cohomology $R=H^*(M)$ as a model for $M$, but now we must account for higher homotopy commutative operations. These higher operations however all produce graphs with $\geq 2$ vertices and hence do not enter the above computation.

\section{Sullivan models}\label{s:Sullivan}
For completeness, let us  remark that Theorem \ref{thm:main1} also allows us to write down a Sullivan model for the components of the embedding space, using standard results in the literature. Recall that by $\HGC_{\bar H^*(M_*), n}$ we understand the $L_\infty$ algebra induced by the homotopy commutative structure
of~$\bar H^*(M_*)$, see Section~\ref{ss:HGC_homot}.
\begin{prop}
  In the setting of Theorem \ref{thm:main1} and for $n\geq 4$ and $H^*(M_*)$ concentrated in degrees $\leq n-3$ we have that every connected component of
  $\MC_\bullet( \HGC_{\bar H^*(M_*), n})
  $
  is nilpotent and of finite (homological and homotopical) type.
  Let $m\in \MC(\HGC_{\bar H^*(M_*), n})$ be a Maurer--Cartan element, let $\HGC_{\bar H^*(M_*), n}^{m}$ be the twisted $L_\infty$-algebra and let $(\HGC_{\bar H^*(M_*), n}^{m})_{> 0}$ be its positive degree truncation.
  Then the (cohomological) Chevalley-Eilenberg complex
  \[
    C_{CE}^*\left((\HGC_{\bar H^*(M_*), n}^{m})_{>0}\right)
  \]
  is a Sullivan model for the connected component $\MC_\bullet( \HGC_{\bar H^*(M_*), n})_m$ of $m$.
  In particular, the cohomology of this component is the Chevalley-Eilenberg cohomology
  \[
    H^k(\MC_\bullet( \HGC_{\bar H^*(M_*), n})_m)
  \cong
  H_{CE}^k\left((\HGC_{\bar H^*(M_*), n}^{m})_{> 0}\right) \, .
  \]
  \end{prop}

  \begin{proof}
  By \cite[Theorem 5.5]{Be} the homotopy groups of the connected component corresponding to the MC element $m$ are computed by $H_*(\HGC_{\bar H^*(M_*), n}^m)$, which is clearly degree-wise nilpotent and of finite type as the underlying $L_\infty$-algebra is.
  Hence the Maurer--Cartan spaces are nilpotent.

  The statement about the Chevalley-Eilenberg complex is obtained by applying \cite[Corollary 1.3]{Be}. %Note that the $L_\infty$-algebra $\alg g$ in that Theorem is required to be nilpotent of finite type. %, while ours is only pronilpotent of finite type -- this can be however seen to be sufficient.
  %\todo{Maybe add a better reference here.}

  \end{proof}

  \begin{rem}
    \newcommand{\HG}{\mathsf{HG}}
  One can also build a model (essentially) out of the Chevalley-Eilenberg complex of $(\HGC_{\bar R, n}^{m})_{>0}$ for a general dg commutative algebra model $R$ of $M_*$ of finite type. %\footnote{Note, however, that it is an open question if any finite CW-complex has a Sullivan model of a finite type~\cite{FLST}.}.
    One just has to interpret the Chevalley-Eilenberg complex appropriately, as follows.
  One first notes that $\HGC_{\bar R,n}$ is in fact the dual of an $L_\infty$-coalgebra $\HG_{\bar R,n}$.
  Concretely, while elements of $\HGC_{\bar R,n}$ are formal (possibly) infinite series of graphs, elements of $\HG_{\bar R,n}$ are finite linear combinations of (the dual) graphs.
  Then the homological Chevalley complex of the positive degree truncation $(\HG_{\bar R,n}^m)_{>0}$ of the twist of $\HG_{\bar R,n}$,  is a dg commutative algebra model for the corresponding connected component of $\MC_\bullet( \HGC_{\bar R, n})$.
  To see this one first notes that we have a map of $L_\infty$-coalgebras $\HG_{\bar R,n}\to \HG_{\bar H^*(M_*),n}$ and accordingly a map of the (homological) Chevalley-Eilenberg complexes
  \[
  C^{CE}_*\left((\HG_{\bar R,n}^m)_{>0}\right) \to C^{CE}_*\left((\HG_{\bar H^*(M_*),n}^{m'})_{>0}\right),
  \]
  where $m$ is the image of the Maurer--Cartan element $m'\in \MC(\HGC_{\bar H^*(M_*),n})$.  %\footnote{This is not a big restriction on $m$. By the Goldman-Millson Theorem \cite{DolRog} we always have a Maurer--Cartan element in the same connected component that is in the image of the map \eqref{equ:HGCHHGCA}. But then the Chevalley complexes for two MC elements in the same connected component are quasi-isomorphic.}
  The map above is again a quasi-isomorphism of dg commutative algebras, as one can see from the spectral sequence associated to the bounded below exhaustive filtration on the number of edges in graphs on both sides.
  \end{rem}

\end{document}